\documentclass[11pt,reqno,oneside]{amsart}

\usepackage{mdframed}

\usepackage{vmargin}
\usepackage{amsmath} %
\usepackage{amsthm} %
\usepackage{amssymb} %
\usepackage{epsfig} %
\usepackage{psfrag} %
\usepackage{enumerate} %
\usepackage{dsfont}

\theoremstyle{plain}

\newtheorem{theorem}{Theorem}[section]
\newtheorem{lemma}[theorem]{Lemma}
\newtheorem{corollary}[theorem]{Corollary}
\newtheorem{proposition}[theorem]{Proposition}
\newtheorem{claim}[theorem]{Claim}

\theoremstyle{definition}

\newtheorem{definition}[theorem]{Definition}
\newtheorem{remark}[theorem]{Remark}

\numberwithin{equation}{section}
\numberwithin{figure}{section}

\newcommand{\N} {\mathbb N}

\newcommand{\R} {\mathbb R}

\renewcommand{\P} {\mathbb P}
\newcommand{\E} {\mathbb E}

\def \ind{\mathds{1}}
\newcommand{\ER} {Erd\H{o}s-R\'enyi }
\newcommand{\cadlag}{c\`adl\`ag}

\newcommand{\lpord}{\ell_p^\downarrow}
\newcommand{\ltwoord}{\ell_{2}^\downarrow}
\newcommand{\lthreeord}{\ell_3^\downarrow}

\newcommand{\linford}{\ell_\infty^\downarrow}
\newcommand{\lzeroord}{\ell_0^\downarrow}
\newcommand{\lzero}{\ell_0}
\newcommand{\loneord}{\ell_1^\downarrow}

\newcommand{\barg}{\bar{g}}

\newcommand{\Var}{\operatorname{Var}}

\def \dist {\mathrm{d}}

\newcommand{\bb}{\mathbf{b}}

\newcommand{\bm}{\mathbf{m}}

\newcommand{\ub}{\underline{b}}
\newcommand{\uc}{\underline{c}}
\newcommand{\um}{\underline{m}}

\newcommand{\cG}{\mathcal{G}}
\newcommand{\cU}{\mathcal{U}}
\newcommand{\cT}{\mathcal{T}}
\newcommand{\cM}{\mathcal{M}}
\newcommand{\cI}{\mathcal{I}}

\newcommand{\ORDX}{\mathcal{E}^\downarrow}
\newcommand{\EX}{\mathcal{E}}

\newcommand{\exmeasure}{\operatorname{Exp}}
\newcommand{\expheightgaps}{\operatorname{Exp^{\prec}}}

\newcommand{\Exp}{\operatorname{Exp}}
\newcommand{\isd}{\stackrel{d}{=}}
\newcommand{\tod}{\stackrel{d}{\to}}

\newcommand{\BMPD}{\mathrm{BMPD}} 
 
\newcommand{\argmin}{\operatornamewithlimits{argmin}}

\newcommand{\Down}{R^{\downarrow}}
 
\title
[Rigid representations of the multiplicative coalescent with
deletion]
{Rigid representations of the multiplicative coalescent with linear deletion}

\author {James B.\ Martin\textsuperscript{1}}
\author{Bal\'azs R\'ath\textsuperscript{2}}

\begin{document}

\begin{abstract}

We introduce the 
\emph{multiplicative coalescent with linear deletion}, 
a continuous-time
Markov process describing the evolution of a collection of blocks.
Any two blocks of sizes $x$ and $y$ merge
at rate $xy$, and any block of size $x$ is deleted 
with rate $\lambda x$ (where $\lambda\geq 0$ is a fixed parameter).
This process arises for example in connection with a variety 
of random-graph models which exhibit self-organised
criticality. We focus on results describing states
of the process in terms of collections of excursion lengths of random functions. For the case $\lambda=0$ (the coalescent without
deletion) we revisit and generalise 
previous works by authors including
Aldous, Limic, Armendariz, Uribe Bravo, and Broutin and Marckert,
in which the coalescence is related to a ``tilt" of a
random function, which increases with time; for $\lambda>0$ 
we find a novel representation in which this tilt
is complemented by a ``shift" mechanism which produces the
deletion of blocks. We describe and illustrate other 
representations which, like the tilt-and-shift representation,
are ``rigid", in the sense that the coalescent process 
is constructed as a projection of some process which
has all of its randomness in its initial state. We 
explain some applications of these constructions to models
including mean-field forest-fire and frozen-percolation processes.

\bigskip

\end{abstract}

\maketitle

\footnotetext[1]{Department of Statistics, University of Oxford, United Kingdom. Email: \texttt{martin@stats.ox.ac.uk}}
\footnotetext[2]{MTA-BME Stochastics Research Group, Hungary. Email: \texttt{rathb@math.bme.hu}}

\section{Introduction}

\subsection{The multiplicative coalescent and MCLD($\lambda$)}

The \textit{multiplicative coalescent} (or briefly MC) is a continuous-time
Markov process describing the evolution of a collection of blocks (components). 
The dynamics are as follows: for each pair of components with 
masses $m_i$ and $m_j$, the pair coalesces at rate $m_i m_j$ 
to form a single component of mass $m_i+m_j$. 

The process is simple to construct from any initial state
with finitely many components. In \cite[Section 1.5]{aldous_mc},
Aldous uses a graphical construction to show that the process 
is well-defined starting from any initial state 
in which the sum of the squares of the masses is finite. 
Writing the masses in decreasing order, 
define the space 
\begin{equation}\label{elltwodef}
\ltwoord=\{\um=(m_1, m_2,\dots): m_1\geq m_2\geq \dots \geq 0, \;
\sum_{i=1}^\infty m_i^2 <\infty\},
\end{equation}
with the distance
\begin{equation}\label{eq_def_d_metric}
 \dist(\um ,\um')=\Vert \um-\um' \Vert_2= \left( \sum_{i \geq 1} (m_i-m'_i)^2  \right)^{1/2}.
 \end{equation}
Then Proposition 5 of \cite{aldous_mc} says that the MC is a Feller process in $(\ltwoord, \dist)$. 

\medskip

We generalise the multiplicative coalescent 
to include \textit{deletion} as well as coalescence. 
Let $\lambda\geq 0$. Now we want the dynamics of the process 
to satisfy:
\begin{equation}\label{mcld_informal_def}
\begin{array}{l}
 \text{(i) any two components of mass $m_i$ and $m_j$ merge with rate $m_i  m_j$,} \\
 \text{(ii) any component of mass $m_i$ is deleted with rate $\lambda m_i$.}
 \end{array}
\end{equation}

We call such a process a \emph{multiplicative coalescent with linear deletion}, with deletion 
rate $\lambda$, and denote it by $\mathrm{MCLD}(\lambda)$.

Again, if the initial state has finitely many components, 
the process is easily seen to be well-defined. 
For initial states in $\ltwoord$, a graphical construction
is given in 
\cite{james_balazs_mcld_feller}, which gives rise 
to a well-behaved continuous-time Markov process 
taking values in $\ltwoord$. This process
has the Feller property with respect to the distance $\dist(\cdot,\cdot)$ --
see Theorem 1.2 of
\cite{james_balazs_mcld_feller}.

\medskip

The motivation of our study of the MCLD process is twofold. 

Firstly, ideas involving the representation of the state of 
a MC
process in terms of the set of excursion lengths 
of some random function began with Aldous \cite{aldous_mc},
and have been developed in a series of works including 
those of Aldous and Limic \cite{aldous_limic}, 
Armend\'ariz \cite{armendariz}, Uribe Bravo \cite{uribe_thesis},
Broutin and Marckert \cite{broutin-marckert} and
Limic \cite{limic_new}. We find that this theory 
has a surprising and elegant extension to 
processes involving deletion, which we build in this paper. 

Secondly, the MCLD  arises as a scaling limit
of certain discrete processes of coalescence and fragmentation
or deletion, such as the mean-field forest-fire model 
introduced by R{\'a}th and T{\'o}th in \cite{bt_br_forest}
and studied by Crane, Freeman and T{\'o}th in 
\cite{ec_nf_bt_forest},
and the mean-field frozen percolation process introduced
by R{\'a}th in \cite{br_frozen_2009}. 
These processes are of particular interest
because of their \textit{self-organised criticality}
properties. These scaling limits are the subject
of a future paper \cite{james_balazs_window_process},
but we discuss related properties here in Section \ref{section_applications}.

We further discuss related literature in Section \ref{subsec:related}.

\subsection{Tilt representations of  MC}
\label{section_intro_tilt_earlier_work_brief}
We start by revisiting previous results 
giving representations of the MC in terms of excursions of random functions. 

\smallskip

We first define 
\textit{Brownian motion with parabolic drift}.
\begin{definition}[Brownian motion with parabolic drift, $\BMPD(u)$]\label{brownian_parabolic_drift}
Let $B(x), x\geq 0$ denote standard Brownian motion. 
Given $u \in \R$, define
\begin{equation}\label{def_eq_bmpd}
h(x)=B(x)-\frac12 x^2+ux, \quad x \geq 0.
\end{equation}
Let us denote by $\BMPD(u)$ the law of the random function $h(x), x \geq 0$.
\end{definition}

For a function $h(x), x\geq 0$, let $\ORDX(h)$ 
denote the lengths of the excursions of $h$ above its
running minimum, written in non-increasing order
(see Section \ref{sec:main} below for a formal definition).
Aldous \cite{aldous_mc} showed that there
is an eternal multiplicative coalescent process,
whose marginal distributions can be written
in terms of excursions of a Brownian motion with parabolic drift.  

\begin{proposition}[Aldous]
\label{proposition:Aldous}
There exists a MC process
$(\bm_t, t\in\R)$ such that for each $t$,
the marginal distribution  $\bm_t$ is the same as 
that of $\ORDX(h_t)$ where $h_t\sim\BMPD(t)$. 
\end{proposition}

This process is known as the \textit{standard multiplicative 
coalescent}. In her PhD thesis, Armend\'ariz \cite{armendariz_phd, armendariz}
showed that in fact the whole process can be constructed
as a function of a single realisation of Brownian motion. 

\begin{proposition}[Armend\'ariz]
\label{theorem:Armendariz}
Let $h_0\sim \BMPD(0)$, 
and define $h_t(x)=h_0(x)+tx$ for all $t\in\R$
(so that $h_t\sim \BMPD(t)$ for all $t$). 
Then the process $\ORDX(h_t), t\in\R$ is the 
standard MC. 
\end{proposition}

We call this representation a \textit{tilt} representation;
for each $t$, the function $h_t$ is obtained from $h_0$ by adding 
the linear function $tx$. It is easy to see that the tilt representation
indeed produces a coalescent process as $t$ increases. Also note that only adjacent excursions
can coalesce under the tilt representation.

Part of the approach of Armend\'ariz (described in \cite[Section 4]{armendariz})  is further elaborated 
in Chapter 4 of the PhD thesis \cite{uribe_thesis} of Uribe Bravo.

\smallskip

In \cite[Corollary 4]{broutin-marckert}, Broutin and Marckert give an alternative proof of 
Proposition \ref{theorem:Armendariz}. 
Their proof involves considering the connected components
of an \ER graph process on $n$ vertices, and then taking
the weak limit as $n\to\infty$.
The key idea is to define an ordering of the vertices
(the \emph{Prim ordering}, related to 
invasion percolation and to the minimal spanning tree) 
which is consistent with 
the coalescent process in the sense that at all times, the 
components are intervals of the Prim order (and thus
only adjacent intervals can coalesce); nonetheless,
for each fixed time, exploring the graph in the Prim ordering
yields a random walk with the same distribution as is
obtained from a standard traversal in, say, 
depth-first or breadth-first order.
\smallskip

Aldous and Limic \cite{aldous_limic}
characterized the set of all eternal multiplicative coalescents, 
i.e.\ those defined for all times  $t \in (-\infty, \infty)$.
They defined a three-parameter set of random processes
$W^{\kappa, \tau, \uc}(x), x\geq 0$, such that if 
$(\bm_t, t\in\R)$ is a non-constant eternal MC which is 
extremal (that is, its law is not the mixture of the laws of other eternal MC processes) then
 for each $t \in \R$,
the marginal distribution of $\bm_t$ is the same as 
that of the sequence $\ORDX(h_t)$ of excursion lengths of $h_t(x):=W^{\kappa,\tau,\uc}(x)+tx$
for some $(\kappa, \tau, \uc)$.
The processes $W^{\kappa,\tau,\uc}$ 
have been called \textit{L\'evy processes without replacement};
see Section \ref{sec:eternal} for their definition and the precise statement of the cited result of \cite{aldous_limic}.

\smallskip

%
%
%

We will further discuss the links between our approach and those of Armend\'ariz \cite{armendariz}, Uribe Bravo \cite{uribe_thesis},
Broutin and Marckert \cite{broutin-marckert} and
Limic \cite{limic_new} in
Section \ref{subsec:related}.

\subsection{Contributions of this paper}
\label{section_contributions}
In this section we summarize the main results of our paper.
The central results are the following:

\begin{itemize}
\item {\bf MC admits a tilt representation from any initial state:}
In Theorem \ref{thm:tilt} we show that for any possible initial state $\um \in \ltwoord$, there exists
a random function $f_0: [0, \infty) \to [-\infty,0]$ such that if we define 
\begin{equation}\label{f_t_intro}
f_t(x)=f_0(x)+tx
\end{equation}
then
$\ORDX(f_t), t\in\R$ is a multiplicative coalescent process  with initial state $\um$.

\item {\bf MCLD admits a ``tilt-and-shift" representation:}
Let $\um \in \ltwoord$, and define $f_t$ for all $t \geq 0$ as 
in (\ref{f_t_intro}). Let $\lambda>0$. In Theorem \ref{thm:mcld_extension_introduction} 
we show that there exists a $\sigma(f_0)$-measurable non-decreasing  function 
$\Phi: [0, \infty) \to [0, \infty)$ such that if we define
\begin{equation}\label{g_t_f_t_intro}
  g_t(x)=
 f_t(x+\Phi(t)) + \lambda t - \int_0^t \Phi(s)\, \mathrm{d}s,  \quad t \geq 0
 \end{equation}
then $\ORDX(g_t), t \geq 0$ is a $\mathrm{MCLD}(\lambda)$ process with initial state $\um$.
\end{itemize}

Here $\Phi$ is a non-decreasing pure jump process;
$\Phi(t)$ is the total weight of components deleted by time $t$.
In addition to the ``tilt" given by (\ref{f_t_intro}),
we now have a ``shift" since in (\ref{g_t_f_t_intro}),
the function $f_t$ is applied at $x+\Phi(t)$. 
If a component of size $a$ is deleted, then 
$\Phi(t)-\Phi(t-)=a$, and we see a shift to the left
of size $a$ at time $t$; that is, 
the graph of $g_{t-}$ on $[a,\infty]$ becomes the graph
of $g_t$ on $[0,\infty]$. 

Note that the tilt representation of the MC is the $\lambda=0$ case of the tilt-and-shift representation of
the MCLD. Under the tilt-and-shift representation,
only adjacent excursions coalesce and
only the ``leftmost'' excursion gets deleted.
We call the above representations ``rigid'' because all of the randomness is contained in the initial state $f_0$, and the rest of the evolution of $f_t$ (resp.\ $g_t$) is deterministic.

The function $f_0$ above is constructed so that
it has the \textit{exponential excursion levels} property: conditional on 
the sequence of excursion lengths,
the levels of the excursions  are independent, 
and an excursion of length $m$
occurs at level $-E$ where $E\sim \Exp(m)$. 
One of the central observations behind the rigid representation results outlined above is that
$f_t$ and $g_t$  for any $t \geq 0$ inherit the exponential excursion levels property.

\medskip

We prove these main results in two stages: we first consider the case of finitely many blocks in
Section \ref{section:rigid_finite_state_space}, then we extend the tilt-and-shift representation 
to $\ltwoord$ using truncation
and approximation arguments in Section \ref{section:extension_to_ltwo}.

The proof of the finite tilt-and-shift representation result in Section \ref{section:rigid_finite_state_space}  involves  various other representations which are also of independent interest.
Their constructions rely heavily on the notion of size-biased orderings, and on the
realisation of such orderings
using independent exponential random variables, c.f.\ Section \ref{sec:expsizebiased}.

Now we outline these  representations and how they relate to each other.
 
\begin{itemize}
\item {\bf Interval coalescent representation:} Given an initial state $\um$ with finitely many blocks, we
consider a process whose states are 
finite sequences of lengths (depicted as intervals arranged in some order).  
In the initial state, the interval lengths are the masses
of blocks of $\um$, and they are arranged 
in a size-biased random order. 
An interval $I$ merges with the interval on its right at a rate equal to the product of the length of $I$ with the combined length of all of the intervals to the right of $I$. The leftmost interval is deleted at a rate
equal to $\lambda$ times the total length of intervals. For an illustration, see Figure \ref{fig:icld}.
In Section \ref{section:interval_coalescent_repr} we show that the decreasing rearrangement
of interval lengths evolves like $\mathrm{MCLD}(\lambda)$.
(If $\lambda=0$, this gives the MC.)

\begin{figure}[h]
\begin{center}
\large
  \includegraphics[width=\textwidth]{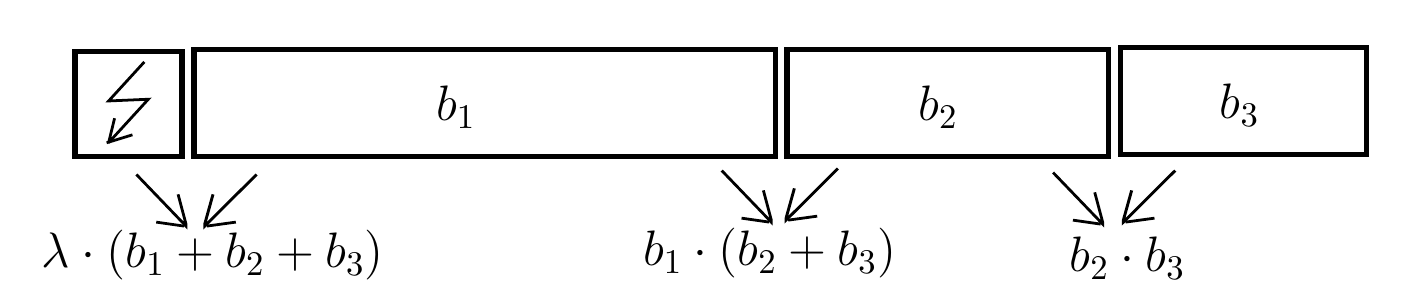}
\caption{An illustration of the rates of the interval coalescent process: the lengths of the three intervals (from left to right) are $b_1$, $b_2$ and $b_3$. The leftmost block marked with a lightning represents the ``cemetery". The rate at which the first interval gets deleted is $\lambda \cdot ( b_1+b_2+b_3)$.
The rate at which the first interval merges with the second one is $b_1 \cdot (b_2+b_3)$. The rate
at which the second and third intervals merge is $b_2 \cdot b_3$. }
\label{fig:icld}
\end{center}
\end{figure}

\item {\bf Particle representation:} 
Given an initial state $\um=(m_1,\dots,m_n)$ as above, we 
initially put a particle of mass $m_i$ at initial height $-E_i$, where 
$E_i \sim \Exp(m_i), 1 \leq i \leq n$ are independent. Now the particles start to move up until they reach height $0$
and die. The speed of a particle is equal to $\lambda$ plus the total weight of particles
strictly above it and strictly below $0$. Note that once a particle reaches the one above it, they
stick together until they die. Accordingly, we group the particles that share the same height into time-$t$ blocks.
In Section \ref{section_particle_representation} we show that the 
the vector of sizes of the time-$t$
blocks of the particle system, in decreasing
order of their height, evolves like the above described interval coalescent process.
See Figure \ref{fig:MCLDparticles} and 
Figure \ref{fig:MCparticles} for simulations;
in the first, $\lambda$ is positive and the 
system realises an MCLD, while in the second,
$\lambda$ is zero and the system realises a multiplicative
coalescent without deletion. 
\end{itemize} 
Note that the particle representation outlined above is also ``rigid''. 

\begin{figure}
\begin{center}
\includegraphics[width=\textwidth]{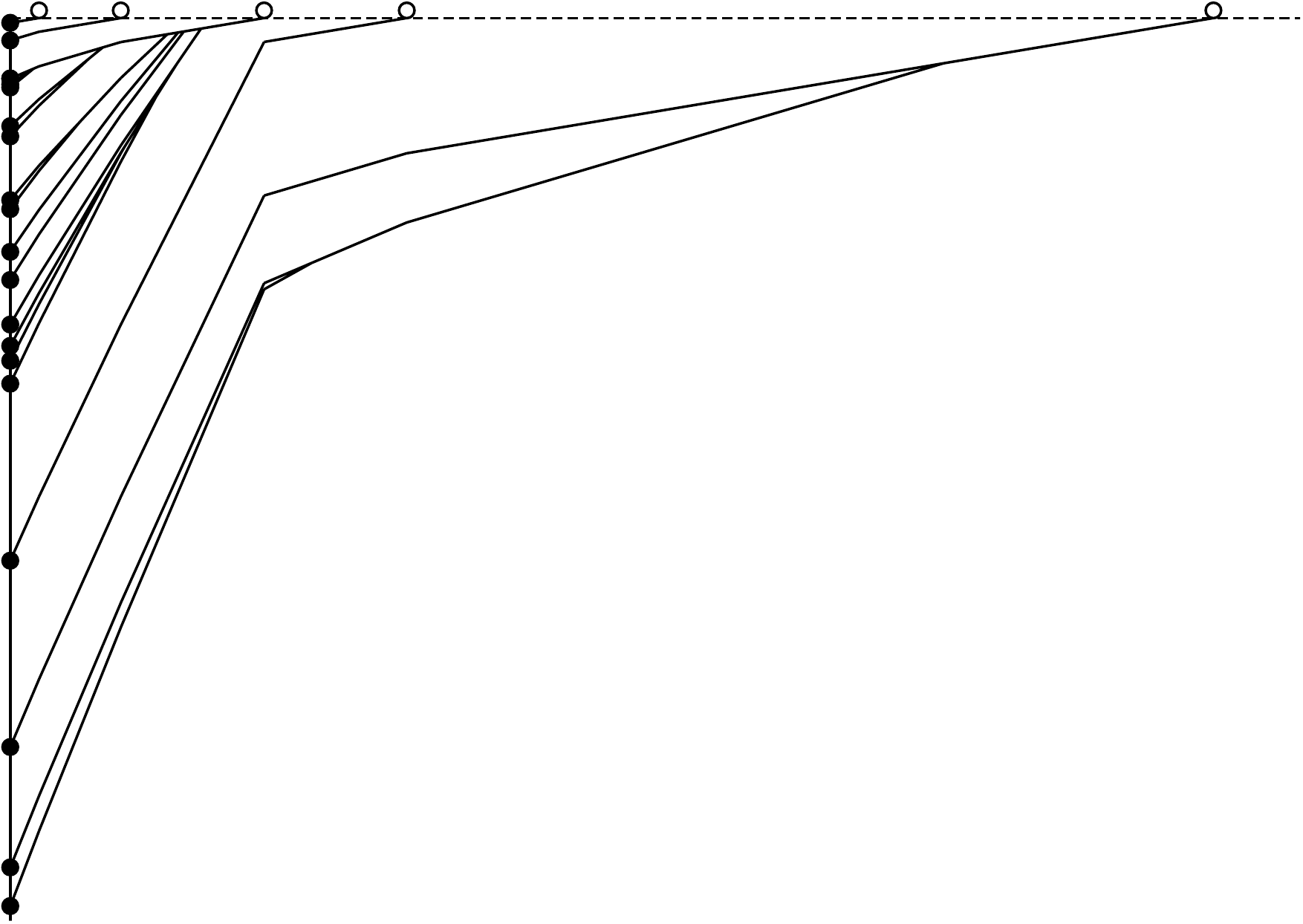}
\caption{
\label{fig:MCLDparticles}
A simulation of the particle system which realises the 
MCLD. Each particle moves upwards at rate equal to $\lambda$
plus the mass of particles strictly above it. When 
a particle reaches 0 it dies and is removed from the system.
Particles at the same height at time $t$ form a time-$t$ block. 
The figure shows a system with $n=20$ particles and 
$\lambda=1.3$, evolving on the time interval $[0,0.5]$.
We see 5 deletion events, and the deleted blocks have sizes 1, 1, 14, 1, 3 respectively. 
Compare to Figure \ref{fig:MCparticles}, where $\lambda=0$ 
(realising the multiplicative coalescent without deletion), and to Figure \ref{fig:FFparticles}, 
where deleted particles reenter at independent exponential heights
(realising the forest fire model).
}
\end{center}
\end{figure}

\begin{figure}
\begin{center}
\includegraphics[width=\textwidth]{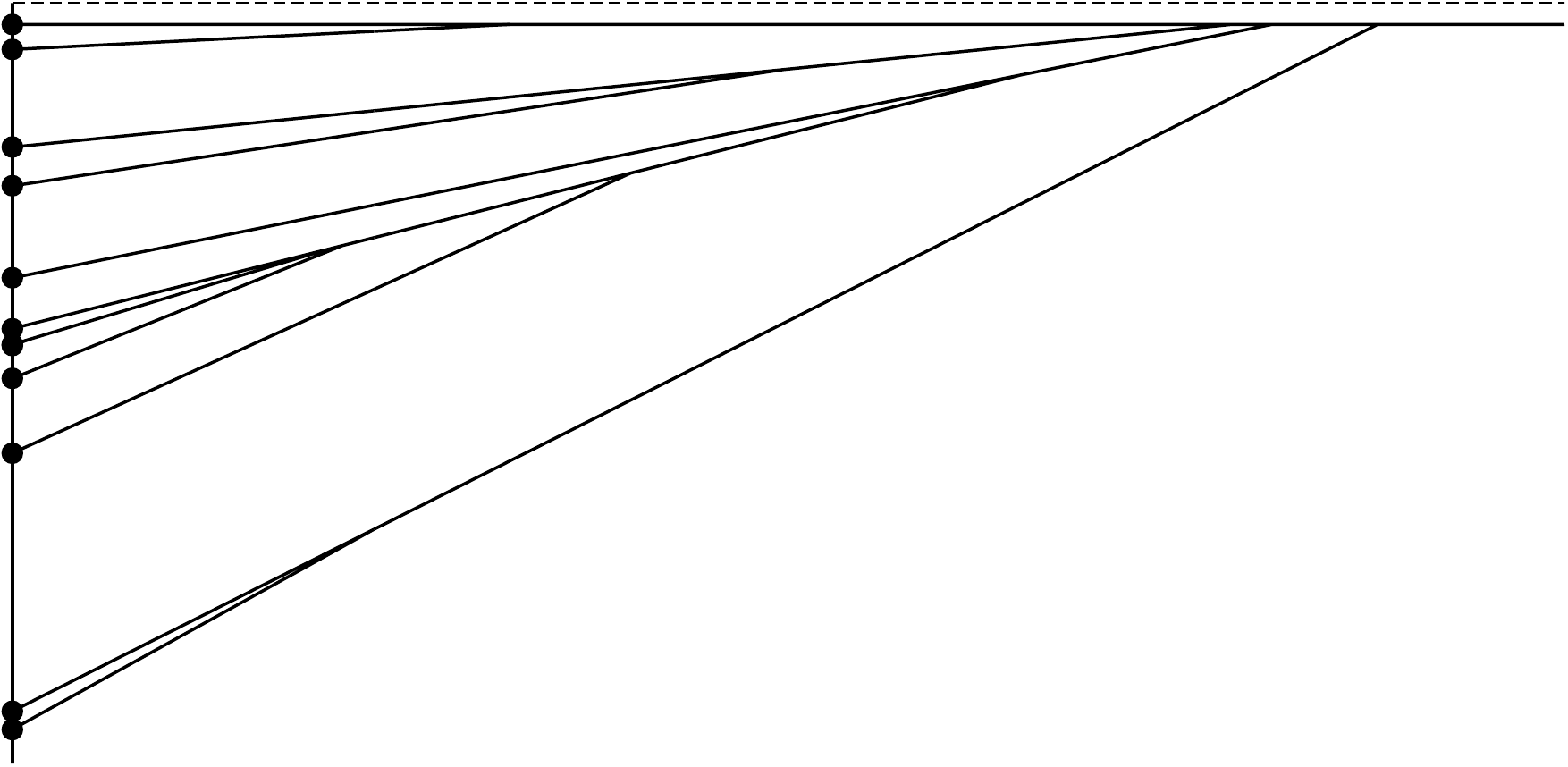}
\caption{
\label{fig:MCparticles}
This shows a similar system to the one in Figure
\ref{fig:MCLDparticles}, but now with $\lambda=0$ so that
the system realises the multiplicative coalescent, with no 
deletion. In this case $n=12$ and the figure shows the time interval 
$[0,0.25]$. By the end, all particles have coalesced into a single block. 
}
\end{center}
\end{figure}

In Section \ref{tiltproofl1} we explain how particles in the particle representation correspond to excursions
of $g_t$ in the tilt-and-shift representation, completing the 
 proof of the validity of the
 tilt-and-shift representation of 
$\mathrm{MCLD}(\lambda)$ with finitely many blocks.

\medskip

In Section \ref{section_applications} we give some applications of our rigid representation results:
\begin{itemize}

\item In Section \ref{sec:eternal} we give a sketch proof of the fact that the
L\'evy processes without replacement $W^{\kappa,\tau,\uc}$ have the 
exponential excursion levels property, and conclude that all eternal multiplicative coalescents
have a tilt representation. 
This gives an alternative way to approach the
 result of Limic \cite[Theorem 2]{limic_new}
(see Section \ref{subsec:related}).

\item In Section \ref{sec:BMPD} we apply the tilt-and-shift representation
to Brownian motion with parabolic drift. What we find can be non-rigorously summarized as follows:
if we start from $\BMPD(u)$ then at time $t \geq 0$ we see
$\BMPD(u+t-\Phi(t))$.
 We show that the resulting $\mathrm{MCLD}(\lambda)$ process 
is the scaling limit of the list of component sizes in the mean field frozen percolation model
of \cite{br_frozen_2009}
 started 
from a near-critical \ER graph, thus extending the result \cite[Corollary 24]{aldous_mc}, which is the $\lambda=0$ case.
 
\item In Section \ref{section_particle_forest} we give a particle representation of the mean field forest
fire model of \cite{bt_br_forest}. We demonstrate how this representation allows us to give a new
 probabilistic
interpretation of the non-linear controlled PDE problem 
 (and the associated characteristic curves) which played
a central role in the 
theory developed in \cite{bt_br_forest} and \cite{ec_nf_bt_forest}.
\end{itemize}

\subsection{Related work}
\label{subsec:related}

\subsubsection{Rigid representations of coalescent processes}

In this subsection we discuss links to previous work 
on rigid representations for the multiplicative coalescent, which is somewhat scattered (and in some cases unpublished). The various approaches and methods of Armend\'ariz \cite{armendariz}, Uribe Bravo \cite{uribe_thesis},
Broutin and Marckert \cite{broutin-marckert} and
Limic \cite{limic_new} resulting in versions of the 
``tilt'' representation are all related to ours, although these similarities are often implicit. Let us now try to sketch some of these connections.

\smallskip
In \cite{armendariz} Armend\'ariz draws attention to the fact that Brownian motion 
with parabolic drift has the
\textit{exponential excursion levels} property 
which we informally described in Section 
\ref{section_contributions}. This observation
had been made by Aldous in  \cite[equation (72)]{aldous_mc}
(see also Bertoin \cite{bertoin_frag}).
Based on this, 
Armend\'ariz constructs 
a representation of the MC with finitely many blocks in \cite[Section 4, Proposition 2]{armendariz} -- this representation
appears to be equivalent to the $\lambda=0$ case of our \emph{particle representation} (see Section \ref{section_particle_representation}), although the methods are different.

\smallskip

In Chapter 4.2 of his PhD thesis \cite{uribe_thesis}, Uribe Bravo recalls and further elaborates
the representation of \cite[Section 4]{armendariz} (without proofs). 
In particular, he points out that (what we call) the initial heights of particles are
in size-biased order and formulates the crucial 
property corresponding to our \eqref{height_gaps_exp_indeed} about the height gaps of particles.
Again, the representation outlined in \cite[Chapter 4.2]{uribe_thesis} is equivalent to the $\lambda=0$ case
of our particle representation; moreover the representation stated (without proof) in 
\cite[Chapter 4.3]{uribe_thesis} is equivalent to our tilt representation of MC 
(c.f.\ Section \ref{subsec:tilt}) in the case of finitely many blocks.
In \cite[Chapter 4.4]{uribe_thesis} Uribe Bravo gives a sketch proof of
Proposition \ref{proposition:Aldous} which is quite similar to our short sketch proof of Proposition \ref{theorem:Armendariz}
given in Section \ref{subsection_sketch_proof_levy_excursions}.

\smallskip
At first, the approach of Broutin and Marckert in \cite{broutin-marckert} seems quite different from ours.
Their construction relies on exploration processes for
random graphs, defined using the \emph{Prim ordering}
for a graph with edge-weights; our methods make no explicit
reference to the graph structure underlying the MC (or MCLD).
However, one can observe that the components of the dynamic \ER graph process, arranged in Prim order, give  
 the \emph{interval coalescent} process 
(see Section \ref{section:interval_coalescent_repr}) without deletion, started from an initial state with equal-sized blocks.
Broutin and Marckert also extend the representation to obtain
a construction of the \textit{standard augmented 
multiplicative coalescent}
of Bhamidi, Budhiraja and Wang \cite{BhBuWa}.

\smallskip
In the recent paper \cite{limic_new},
Limic obtains tilt representations of all
eternal multiplicative coalescents. 
Theorem 2 of \cite{limic_new} can be written as follows. 
Let $f_t(x)=W^{\tau, \kappa, \uc}(x)+tx$, where $W^{\tau, \kappa, \uc}$ is a
L\'evy processes without replacement. Then
the process $\ORDX(f_t), t\in\R$ is an eternal multiplicative
coalescent.  Limic's approach 
involves the generalisation of the 
\emph{breadth-first walks} used in \cite{aldous_mc}
and \cite{aldous_limic}, to give a system 
of \emph{simultaneous breadth-first walks}
relating to the state of the multiplicative coalescent
at different times, see \cite[Proposition 7]{limic_new}.
In \cite[Section 3]{limic_new}
Limic recalls from \cite[Section 4]{armendariz} and 
\cite[Chapter 4.2]{uribe_thesis} the definition of \emph{Uribe's diagram} (which is equivalent to the $\lambda=0$ case of our particle representation) and proves in \cite[Proposition 5]{limic_new}
(using what she calls a partition-valued ``look-down'' formulation)
that this representation is indeed a valid representation of MC.
In \cite[Section 4]{limic_new} Limic explores the connection between Uribe's diagram
and  simultaneous breadth-first walks. 
We describe an alternative approach to the 
tilt representation of the set of eternal coalescents
in Section \ref{sec:eternal}.
Note that the connection between 
Prim's algorithm \cite{broutin-marckert} and simultaneous breadth-first walks \cite{limic_new} 
is the topic of ongoing research \cite{uribe_future}, c.f.\ Open question 3 of \cite[Section 7]{limic_new}.

\smallskip

Finally we mention that also in the case of the
\emph{standard additive coalescent} 
\cite{Aldous-Pitman-additive},
a tilt representation can be given. 
Let $\mathrm{e}(x)$ be a standard Brownian excursion
on $[0,1]$, and for $\lambda>0$ define the function $h_{\lambda}$
on $[0,1]$ by $h_\lambda(x)=\mathrm{e}(x)-\lambda x$. 
A consequence of the results of Bertoin \cite{bertoin_frag} 
is that the process $\ORDX(h_{e^{-t}}), t\in\R$ is a version of 
the standard additive coalescent. An alternative
proof of this result is given by Broutin and Marckert
\cite{broutin-marckert}.

\subsubsection{Mean-field graph models of self-organized criticality (SOC)}

The mean-field frozen percolation model \cite{br_frozen_2009} is a dynamic random graph model
where the initial number of vertices is $n$, two connected components of size $k$ and $l$
merge at rate $\frac{kl}{n}$ and a connected component of size $k$ disappears at rate $\lambda(n)k$.
Note that up to scaling, this is MCLD$(\lambda)$. If $1/n \ll \lambda(n) \ll 1$ then the evolution of the 
densities of small components is asymptotically described by
Smoluchowski's coagulation differential equations with multiplicative kernel
 \cite[Theorem 1.2]{br_frozen_2009}, the solutions of which
exhibit SOC  \cite[Theorem 1.5]{br_frozen_2009}.

\smallskip

The definition of the mean-field forest fire model \cite{bt_br_forest} is the same as that of
the above described frozen percolation model, with one difference:
instead of removing a component, we only remove its edges, with
the vertices remaining as singletons.
A variant of Smoluchowski's system of equations describes the asymptotic densities of small components
\cite[Section 1.3]{bt_br_forest}, leading to SOC.
The proof of the well-posedness of this infinite system of differential equations involves
a non-linear PDE which is a controlled variant of the Burgers equation.

\smallskip

In \cite{ec_nf_bt_forest} Crane, Freeman and T\'oth  describe the 
time  evolution of the size of the component  of a fixed vertex in the above forest fire model in the $n \to \infty$ limit. They also give a probabilistic meaning to the characteristic curves of the 
controlled Burgers equation in  \cite[Remark 3.11]{ec_nf_bt_forest}. 
In Section \ref{section_controlled_burgers} we give a different probabilistic meaning
to these characteristic curves by comparing them to particle trajectories in the particle representation.

\smallskip

The model proposed by Aldous in \cite[Section 5.5]{Aldousfrozen} is studied
by Merle and Normand in \cite{merle_normand}: 
 two connected components of size $k$ and $l$
merge at rate $\frac{kl}{n}$ and  connected components disappear if their size exceeds
a threshold  $\omega(n)$. In order to achieve SOC, one chooses
 $1 \ll \omega(n) \ll n$. Theorem 1.1 of \cite{merle_normand} 
 states that the densities of small components converge to the solution of
 Smoluchowski's coagulation equations with multiplicative kernel.
 Note that this is a special case of the main result of Fournier and Laurencot in \cite{fournier_laurencot_smol},
 who study
 discrete models of Smoluchowski's coagulation equation with more general coagulation kernels.
 Theorem 1.3 of \cite{merle_normand} identifies the Benjamini-Schramm limit of the
 ``threshold deletion'' graph model as $n \to \infty$. 
 
 Note that by comparing
 \cite[Theorem 1.2]{br_frozen_2009} and \cite[Theorem 1.1]{merle_normand} 
one sees that (in the SOC regime) the densities of small components converge to the same hydrodynamic limit 
in the model with linear deletion and the model with threshold deletion.
  However, if one is interested in the scaling limit of big component dynamics, 
 the exact deletion mechanism does crucially enter the picture. We discuss scaling limits of the model with linear deletion
 in Section \ref{sec:BMPD}. The scaling limit of the threshold deletion model is not yet
 known, but in Remark \ref{remark_merle_normand_particle} we give a
  particle representation of it.

In \cite{merle_normand_2}  Merle and Normand identify the Benjamini-Schramm limit
of a different self-organized critical model of aggregation, where vertices
can only connect to a fixed number of other vertices.

\subsubsection{Scaling limits of critical random graph models}
In \cite[Corollary 2]{aldous_mc} Aldous identifies the scaling limit of component sizes in a near-critical 
\ER graph in terms of the excursion lengths of $\BMPD$ and in \cite[Corollary 24]{aldous_mc}
he shows that the standard MC is the scaling limit of the evolution of component sizes of
the dynamic \ER graph in the critical window. Let us now list some related results.

\smallskip

The papers \cite{basin, bh_ho_leeu, dembo_levit_v,   a_joseph,  nachmias_peres_reg_perc, riordan} explore the
universality class of  graph models 
 whose scaling limits are described by $\BMPD$ and the standard $\mathrm{MC}$.

\smallskip

The family of eternal multiplicative coalescent processes are characterized in \cite{aldous_limic}.
The scaling limits of some classes of inhomogeneous random graph models are given by L\'evy processes without replacement and non-standard eternal $\mathrm{MC}$ processes, see  
\cite{aldous_limic,  bh_novel, a_joseph}. 

\smallskip

The continuum scaling limit of the metric structure of critical random graphs is studied in the
``$\BMPD$'' universality class in \cite{goldschmidt_et_al,  bh_sen_rw, bhamidi_sen_wang}
and in the
``L\'evy processes without replacement'' universality class in \cite{ bh_h_ss_metric}.

\smallskip

Martin and Yeo \cite{martin_yeo} study the 
\ER random graph within the critical window,
conditioned to be acyclic; alternatively expressed,
this is the uniform distribution over forests with 
a given number of vertices and edges. Analogously to
Proposition \ref{proposition:Aldous},
they obtain a scaling limit for the collection of component sizes
which is described by the sequence of excursions
of an appropriate reflected diffusion; here the drift
of the diffusion depends on space as well as on time.

\smallskip

In the future paper \cite{james_balazs_window_process} we describe the possible scaling limits
that can arise from a frozen percolation process with lightning rate $\lambda n^{1/3}$, started from an empty graph. The possible limit objects are eternal $\mathrm{MCLD}(\lambda)$ processes.
The ``arrival'' of the process at the critical window gives rise to a non-stationary eternal $\mathrm{MCLD}(\lambda)$ scaling limit,
while the scaling limit in the ``self-organized critical'' regime is a stationary $\mathrm{MCLD}(\lambda)$.



\subsection{Plan of the paper}
\label{subsec:plan}

In Section \ref{sec:main} we give the main results 
about tilt representations for MC and tilt-and-shift representations
for MCLD, and discuss their interpretation.

\smallskip

In Section \ref{section:rigid_finite_state_space} we prove the tilt-and-shift representation
result of MCLD in the case of finitely many blocks. Along the way, we introduce the interval coalescent representation and the particle representation and show how all these representations relate to each other.

\smallskip

In Section \ref{section:prep_measure_excusrions} we collect some preparatory results about random point measures
and excursions that we will use in Section \ref{section:extension_to_ltwo}.

\smallskip

In Section \ref{section:extension_to_ltwo} we extend our rigid representation results to any
 initial state $\um\in \ltwoord$ by approximating $\um$ with a sequence of truncated
 initial states. Since the particle representation is essentially the same as the tilt-and-shift representation,
 our proof involves a careful analysis of the effect of the insertion of a new particle on the death times of other particles.

\smallskip
 
In Section \ref{section_applications} we present the applications of the theory of rigid representations that we mentioned in Section \ref{section_contributions}.

\section{Main results}\label{sec:main}
We begin by introducing some notation. Let
\begin{align*}
\linford &= \{ \;\um=(m_1, m_2, \dots) \; : \;
m_1\geq m_2\geq \dots \geq 0 \; \} \\
\lpord &= \{ \; \um \in \linford \; : \; \sum_{i=1}^\infty m_i^p<\infty \} \text{ for } 0<p<\infty
\} \\
\lzeroord &= \{ \; \um \in \linford \; : \; \exists \; i_0 \in \N \; : 
\; m_i=0 \; \text{ for all } \; i \geq i_0 \; \}
\end{align*}

We will use the topology of coordinate-wise convergence on $\linford$.

For $\um\in\lzeroord$ with $n$ non-zero entries, we will sometimes ignore 
the infinite trailing string of zeros
and regard $\um$ as an element of $\R_{>0}^n$.

For $\um,\um' \in \ltwoord$ we define the distance $\dist(\um ,\um')$ by \eqref{eq_def_d_metric}.
The metric space $\left( \ltwoord, \, \dist(\cdot,\cdot) \right)$ is complete and separable. 

Next, we introduce some definitions relating to excursions.

\begin{definition}\label{def:excursion}
Let $g:[0,\infty)\to\R\cup\{-\infty\}$
be a function which is \cadlag{} and lower semi-continuous
(i.e.\ all jumps are downards).
For $0\leq l< r<\infty$, the interval $[l,r)$ is an 
\emph{excursion above the minimum}
of $g$ if:
\begin{itemize}
\item[(i)] $g(x)>g(l)$ for all $x<l$.
\item[(ii)] $r=\inf\{ \, x:g(x)<g(l) \, \}$.
\end{itemize}
We say that $r-l$ is the \emph{length} of the excursion and
$g(l)$ is the \emph{level} of the excursion.
 We say that the excursion is \emph{strict} if $g(x)>g(l)$ for any $x\in(l,r)$.
\end{definition}
From now on, we say simply ``excursion" to mean excursion above the minimum,
and we say that $l$ is a ``minimum" if $l$ is the left endpoint of an excursion of $g$. 

\begin{definition}\label{def:ORDX}
Suppose that for any $\varepsilon>0$, $g$ has only finitely many
excursions with length greater than $\varepsilon$. 
Then let $\ORDX(g)\in\linford$ be the sequence
of the lengths of the excursions of $g$, arranged in non-increasing order.
\end{definition}

We  write $\barg$ for the function defined by 
 \begin{equation}\label{eq:def_barg}
 \barg(x)=\inf_{0\leq u\leq x} g(u).
 \end{equation} 
Note that if $g$
is a lower semi-continuous \cadlag{} function, then 
the excursions of $g$ and $\barg$ have the same lengths and levels;
in particular, 
 \begin{equation}\label{ordx_of_bar_is_ordx}
   \ORDX(g)=\ORDX(\barg).
 \end{equation}

\subsection{Tilt representation of  multiplicative coalescent}
\label{subsec:tilt}

The aim of this section is to formulate the ``rigid'' representation of the
 MC process from any initial condition in Theorem \ref{thm:tilt}. 

\begin{definition}\label{def_inverse_cdf}
 Given a locally finite measure $\mu$ on $(-\infty, 0]$, we define the \emph{inverse cumulative distribution function}
 $f_\mu: [0,+\infty) \to [-\infty, 0] $ of $\mu$ by
\begin{equation}\label{def_eq_inverse_sdf_of_a_measure}
f_{\mu}(x)=\sup \{\; y \leq 0 \; : \; \mu[y, 0) > x \; \}, \qquad x \geq 0.
\end{equation}
In particular, $f_{\mu}(x) = -\infty$ for any $x \geq \mu(-\infty,0)$.
\end{definition}
Note that $f_\mu$ is non-increasing, lower semi-continuous and 
\cadlag.

\begin{definition}\label{def_exp_measure_law}
Given $\um=(m_1,m_2, \dots) \in \ltwoord$, 
 we define the independent exponential random variables
\begin{equation}\label{exponentials_E_i}
 E_i \sim \mathrm{Exp}(m_i), \; i=1,2,\dots.
\end{equation}
(If $m_i=0$, we formally define $E_i=+\infty$.)
We say that the random measure $\mu$ has
$\exmeasure(\um)$ distribution if  $\mu$ is a
point measure
 with point masses of weight $m_i$ at locations $-E_i$, $i \in \N$:
\begin{equation}\label{mu_exp_point_measure_def}
\mu=\sum_{i=1}^{\infty} m_i \cdot \delta_{Y_i}, \quad Y_i=-E_i.
\end{equation}
\end{definition}

If $\mu \sim \exmeasure(\um)$, then
the total mass $\mu(-\infty,0]=\sum_i m_i$ is infinite if $\um \notin  \loneord$.
However, as long as $\um \in \ltwoord$, the mass distribution is 
locally finite: in Lemma \ref{lemma:exponential} we will show that almost surely 
$\mu(A)<\infty$ for every bounded set $A \subseteq (-\infty,0]$.

\begin{definition}\label{def:f_zero_from_exponentials}
Let $\um \in\ltwoord$ and $\mu_0 \sim \exmeasure(\um)$.
Let $f_0: [0,+\infty) \to [-\infty, 0]$ be the inverse  cdf of 
$\mu_0$, i.e.,
\begin{equation}\label{f0def_b}
f_0(x)\stackrel{ \eqref{def_eq_inverse_sdf_of_a_measure} }{=} f_{\mu_0}(x).
\end{equation}
\end{definition}

\begin{remark}\label{remark:f_0}
 Let $f_0$ be defined by Definition \ref{def:f_zero_from_exponentials}.

\begin{enumerate}[(i)]
\item \label{remark_f0_i_lebesgue} An alternative characterization of the function $f_0$  is as follows:
$f_0$ is the non-increasing \cadlag{} function
 such that the interval $I_j$ on which it takes the value $-E_j$
has length $m_j$, moreover the Lebesgue measure of the complement of
 $\cup_{j=1}^{\infty} I_j$ is zero. 

\item \label{remark_f0_ii_finite} If $\um=(m_1,\dots,m_n) \in \lzeroord$ and 
$E_{\sigma_1}<\dots< E_{\sigma_n}$ is the increasing rearrangement of $E_i, \, 1 \leq i \leq n$, then
 an equivalent way to write the function  $f_0$  is
 \begin{equation}\label{f0def_for_lzeroord}
 f_0(x)=\begin{cases} 
 -E_{\sigma_k} &\quad \text{if} \quad \sum_{l=1}^{k-1} m_{\sigma_l } \leq x <\sum_{l=1}^{k} m_{\sigma_l }, \quad
 1 \leq k \leq n,\\
 -\infty &\quad \text{if} \quad x \geq \sum_{i=1}^n m_{\sigma_l}.
 \end{cases}
 \end{equation}
 For an illustration of \eqref{f0def_for_lzeroord}, see Figure \ref{fig:f_zero_step}.
  
\item Recalling Definitions \ref{def:excursion} and \ref{def:ORDX} we see that 
the excursion lengths of $f_0$ are 
given by the entries
of $\um \in \ltwoord$; that is, $\ORDX(f_0)=\um$.  
 \end{enumerate}
 \end{remark}
 
\begin{figure}[h]
\begin{center}
\large
\includegraphics[width=\textwidth]{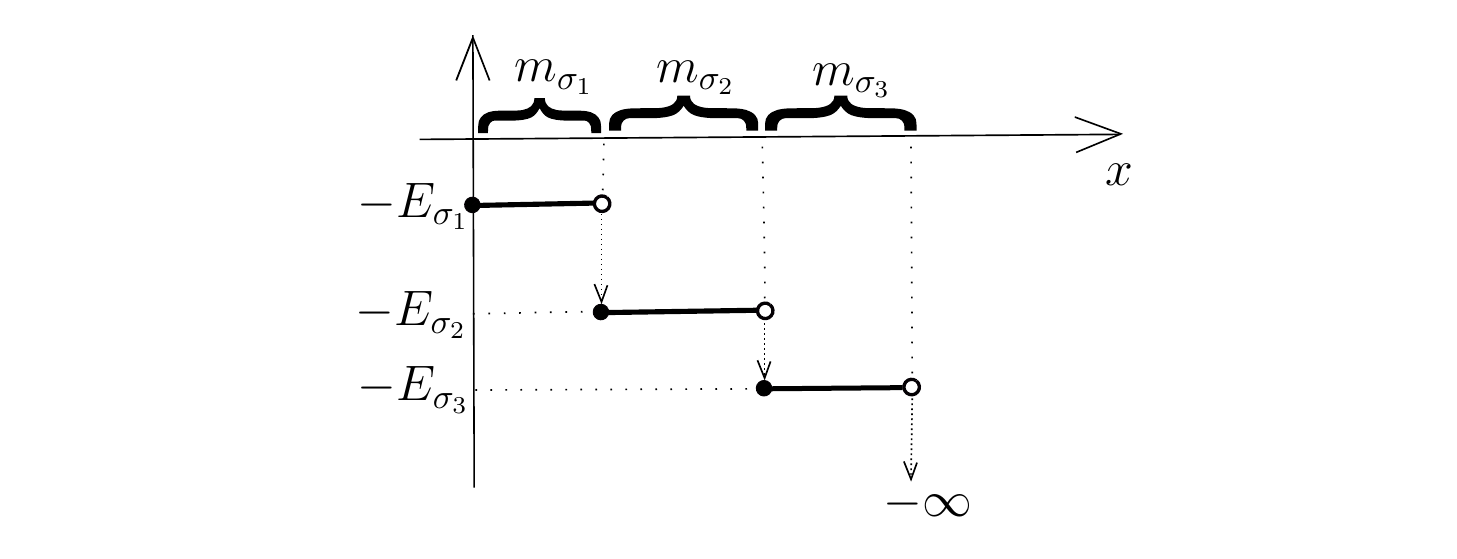}
\caption{An illustration of the function $f_0$ defined in \eqref{f0def_for_lzeroord} when $n=3$.
}
\label{fig:f_zero_step}
\end{center}
\end{figure}

\begin{definition}\label{def:f_t_tilt_from_f_zero} 
Let $f_0$ be defined by Definition \ref{def:f_zero_from_exponentials}.
Let us  define 
\begin{equation}\label{def_eq_f_t_from_f_0}
f_t(x) = f_0(x)+tx, \qquad x \geq 0.
\end{equation}
We say that the function $f_t$ 
is a ``tilt" of $f_0$. 
\end{definition}

In Lemma \ref{lemma:ftgood} we will show that, with probability 1, for all $t$
the function
$f_t$ satisfies the criteria of Definition \ref{def:ORDX}.

\begin{theorem}
\label{thm:tilt}
Let $\um \in \ltwoord$.  The process $\ORDX(f_t), t\geq 0$ has the law of the multiplicative 
coalescent  started from $\um$. 
\end{theorem}
We will prove Theorem \ref{thm:tilt} for $\um \in \lzeroord$ in Section \ref{section:rigid_finite_state_space},
and extend this result to $\um \in \ltwoord$ in Section \ref{section:extension_to_ltwo}.

We say that Theorem \ref{thm:tilt} gives a ``rigid'' representation of the MC process, because
all of the randomness is contained in the initial state of the representation and the rest of the dynamics is rigid, i.e.,
deterministic.

\subsection{Tilt-and-shift representation of $\mathrm{MCLD}(\lambda)$}
\label{subsec:tilt-and-shift}
Similarly to the rigid representation of the MC
in terms of the excursion lengths of $(f_t(\cdot))$ in Theorem \ref{thm:tilt}, we
will give a rigid representation of the $\mathrm{MCLD}(\lambda)$ in terms of the excursion lengths
of another function $(g_t(\cdot))$ in Theorem \ref{thm:mcld_extension_introduction} below. 
We begin with the case of finitely many components.

\begin{definition}\label{def_tilt_shift_step_function}
Given $\um \in \lzeroord$, we define $g_0(x) \equiv f_0(x)$, were $f_0(x)$ is defined by
\eqref{f0def_b} (or, equivalently, \eqref{f0def_for_lzeroord}). We will now define $g_t(x)$ for $t, x \geq 0$ 
such that for all $x \geq 0$, the function $t \mapsto g_t(x)$ is \cadlag. 
 The time evolution of $g_t(\cdot)$ consists of two parts:
 \begin{enumerate}
 \item \emph{Tilt:} If $g_{t-}(0) < 0$ then we let $\frac{\mathrm{d}}{\mathrm{d}t} g_t(x) = \lambda + x$.
 \item \emph{Shift:} If $g_{t-}(0)=0$, then we let $g_{t}(x)=g_{t-}(x+x^*(t))$, where
\begin{equation}\label{eq_def_finite_shift}
 x^*(t)=\inf \{ \, x >0 \; : \; g_{t-}(x) <0 \}
\end{equation}
 is the length of the \emph{first excursion} of $g_{t-}(\cdot)$ (see Definition \ref{def:excursion}).
 \end{enumerate}
 Let us define $\nu$ to be the measure on $[0,\infty)$ given  by
  \begin{equation}\label{def_eq_nu_sum_dirac}
  \nu = \sum_{0 \leq t < \infty} x^*(t) \cdot \delta_{t} 
  \end{equation}
 where $x^*(t)>0$ is the size of the shift to the left at time $t$ (see \eqref{eq_def_finite_shift}); and if no shift occurred at time $t$, then we let $x^*(t)=0$. Let us also define 
 \begin{equation}\label{def_eq_Phi_from_nu}
 \Phi(t)=\nu[0,t],
 \end{equation}
  the total amount of left shifts up to time  $t$.
 \end{definition}
For an illustration of Definition \ref{def_tilt_shift_step_function} see Figure \ref{fig:tilt_and_shift}.

\begin{figure}[h]
\begin{center}
\large

\includegraphics[width=\textwidth]{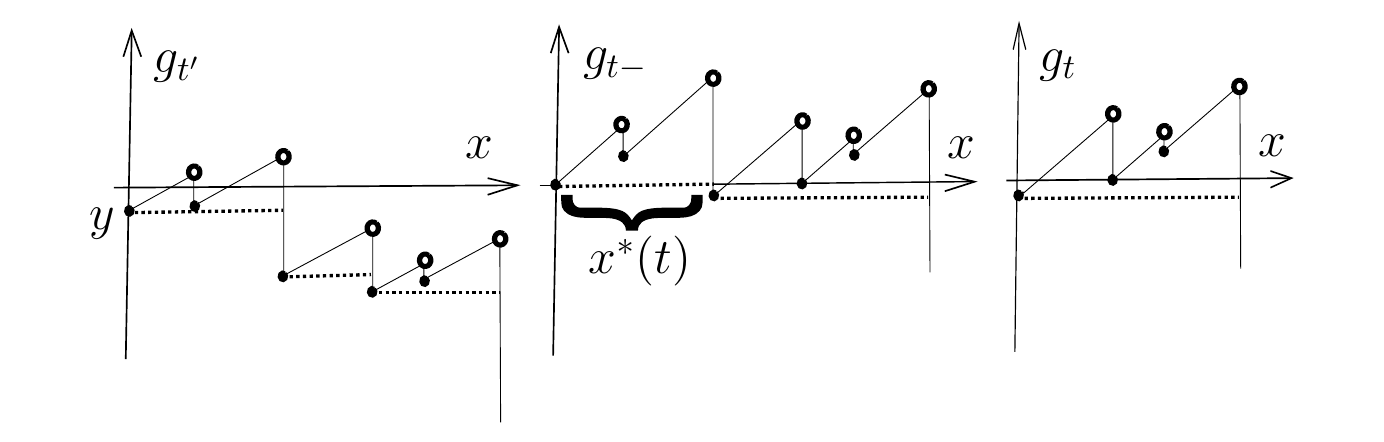}
\caption{An illustration of Definition \ref{def_tilt_shift_step_function}. 
Dashed lines are used to denote the
 running minimum. 
The second graph $g_{t-}$ is obtained from the first graph $g_{t'}$ using the \emph{tilt} operation.
Observe that some excursions of $g_{t'}$ got merged with the tilt operation. 
The third graph $g_{t}$
is obtained from the second graph by \emph{shifting} it  to the left by $x^*(t)$, i.e.,  the length
of the first excursion of the second graph. 
If we denote by $y$ the level of the first excursion of $g_{t'}$ then $t=t'+|y|/\lambda$ is the time instant
when this excursion reaches level zero and thus gets shifted.
}
\label{fig:tilt_and_shift}
\end{center}
\end{figure}

Recall the definition of the $\mathrm{MCLD}(\lambda)$ from \eqref{mcld_informal_def} and the notion of $\ORDX$ from 
Definition \ref{def:ORDX}.
\begin{proposition}
\label{prop:tilt_and_shift_l0} Let $\um \in \lzeroord$, $\lambda>0$ and let $f_0$ be defined by Definition \ref{def:f_zero_from_exponentials}. If we define $g_t(\cdot)$ by Definition \ref{def_tilt_shift_step_function}, then
the process $\ORDX(g_t), t\geq 0$ has the law of the $\mathrm{MCLD}(\lambda)$ process
 $\bm_t, t \geq 0$ started from $\bm_0=\um$. 
\end{proposition}

We will prove Proposition \ref{prop:tilt_and_shift_l0} in Section \ref{section:rigid_finite_state_space}.

 \begin{remark}\label{remark_mass_deleted_Phi_t}
In the $\mathrm{MCLD}(\lambda)$ interpretation, $\Phi(t)$  corresponds to the total amount of mass deleted up to time $t$.
\end{remark}

\medskip

 From Definition \ref{def_tilt_shift_step_function} it follows that we have
 \begin{align}\label{g_t_from_g_0}
 g_t(x)&=g_0\big(x+\Phi(t)\big)+\lambda t + \int_0^t \big(x + \Phi(t)-\Phi(s) \big) \, \mathrm{d}s
 \\
 \nonumber
 &=g_0\big(x+\Phi(t)\big) + \big(x+\Phi(t)+\lambda\big)t
 -\int_0^t\Phi(s) \, \mathrm{d}s.
 \end{align}

\begin{remark}\label{remark_no_first_excursion}
 Extending the dynamics of $g_t$ for  initial states in $\um \in \ltwoord$
  will amount to finding  $\Phi(\cdot)$  corresponding to $g_0:=f_0$, c.f.\ Definition \ref{def:f_zero_from_exponentials}.
We will then define
$g_t$ using the formula \eqref{g_t_from_g_0}: the question  is how to define $\Phi(\cdot)$ in a way that will appropriately extend Definition \ref{def_tilt_shift_step_function} from $\um \in \lzeroord$ to $\um \in \ltwoord$. 
The technical issue that we have to overcome is that for a typical $t \geq 0$ our functions $g_{t-}(\cdot)$ 
 does not have a ``first excursion'' (c.f.\ \eqref{eq_def_finite_shift}) if $\um \in \ltwoord \setminus \loneord$.
  For example, a Brownian motion with parabolic drift does not have a first excursion, 
 yet in Section \ref{sec:BMPD} we will apply our  tilt-and-shift representation  to
  $\BMPD(u)$.
 \end{remark}

We are ready to state the main result of the paper.

\begin{theorem}\label{thm:mcld_extension_introduction}
For any $\um \in \ltwoord$ let us define $g_0(x) \equiv f_0(x)$, 
where $f_0$ is constructed using Definition \ref{def:f_zero_from_exponentials}.
There exists a random  measure $\nu$ such that if we define  $\Phi(t)=\nu[0,t]$ 
and $g_t(x)$  by \eqref{g_t_from_g_0} then
\begin{enumerate}[(i)]
\item \label{mcld_rep_tilt_shift_statement_i}
the process $\ORDX(g_t), t\geq 0$ has the law of the $\mathrm{MCLD}(\lambda)$ process
  started from $\um$,
  \item \label{g_t_zero_is_zero}   if $\um \in \ltwoord \setminus \loneord$, then $g_t(0) = 0$ for any $t \geq 0$,
 \item \label{Phi_stopping_time_ii}  for any $t, x \geq 0$,
the event $\{\, \Phi(t) \leq x \, \}$ is measurable 
with respect to 
the $\sigma$-algebra
$\mathcal{F}^+_x:=
\bigcap_{\varepsilon>0}  \sigma \left( g_0(x'), 0 \leq x' \leq x+\varepsilon \right)$.
 \end{enumerate}  
\end{theorem}

We will prove Theorem \ref{thm:mcld_extension_introduction} in 
Section \ref{section:extension_to_ltwo}.

\begin{remark}\label{remark_how_we_extend_mcld_to_l2}
\begin{enumerate}[(i)]
\item 
Put another way, Theorem \ref{thm:mcld_extension_introduction}(iii)
says that for any fixed $t$, the random variable
 $\Phi(t):=\nu[0,t]$ is a stopping time with respect to the filtration
$\left(\mathcal{F}^+_x\right)_{x \geq 0}$. We will make good
use of this together with the strong Markov property of
Brownian motion in Section \ref{sec:BMPD}.
 \item 
By Theorem \ref{thm:mcld_extension_introduction}(iii),
The control function $\Phi(\cdot)$
 is measurable with respect to the $\sigma$-algebra 
$\sigma \left( g_0(x), x \geq 0 \right)$. Therefore we have
a ``rigid'' representation of $\mathrm{MCLD}(\lambda)$, since
the function $g_t(\cdot)$ defined by \eqref{g_t_from_g_0} is determined by the initial state $g_0(\cdot)$.

\item \label{remark_iii_how_we_truncate} We construct the  measure $\nu$ that appears in Theorem \ref{thm:mcld_extension_introduction}
by extending our earlier construction given in Definition \ref{def_tilt_shift_step_function}
from $\lzeroord$ to $\ltwoord$ in the sense that we obtain $\nu$ as the weak limit as $ n \to \infty$ 
of the measures
 $\nu^{(n)}$ corresponding to initial conditions $\um^{(n)}$ truncated at index $n$,
 see Lemma \ref{lemma_vague_conv} and Corollary \ref{corollary_portemanteau}. 
\item If $\um \in \ltwoord \setminus \loneord$, then $g_0$ is a continuous function satisfying
 $g_0(x)/x \to -\infty$ as $x \to \infty$
(see Lemmas \ref{lemma:ftgood}, \ref{lemma:ucbi}).
We conjecture that for any such function $g_0$,
  there is a unique measure $\nu$ such that the function
 $g_t(x)$ controlled  by $\Phi(t)=\nu[0,t]$ according to \eqref{g_t_from_g_0} satisfies
\begin{itemize}
\item  $g_t(0)\equiv 0$ for all $t \geq 0$,
\item $g_{t-}(x) \geq 0$ for any $0 \leq x \leq \nu(\{t\})$ for all $t \geq 0$.
\end{itemize} 
\item\label{equivalent_start}
The function $g_0$ is constant on each of its excursions.
Suppose that $h_0$ is another \cadlag{} and lower semi-continuous function such that $g_0=\bar{h}_0$,
 that is,
the excursions of $h_0$ have the same lengths and levels
as those of $g_0$. 
Define, analogously to \eqref{g_t_from_g_0},
\begin{equation}\label{h_tilt_shift}
h_t(x)=h_0\big(x+\Phi(t)\big)+\lambda t + \int_0^t \big(x + \Phi(t)-\Phi(s) \big) \, \mathrm{d}s,
\end{equation}
so that $h_t$ is constructed using the same tilt-and-shift
procedure as $g_t$ (and using the same control function $\Phi$).
Then it is straightforward to see that
$\bar{h}_t=\bar{g}_t$, and so, as at 
\eqref{ordx_of_bar_is_ordx},
$\ORDX(g_t)$ and $\ORDX(h_t)$
are the same process. We will 
make use of this observation in Section 
\ref{sec:BMPD} when we consider the tilt-and-shift 
construction started from Brownian motion with a parabolic drift. 
\end{enumerate}
\end{remark}

\section{Rigid representations: finite state space}
\label{section:rigid_finite_state_space}

In this section we restrict the MCLD($\lambda$) 
process to the space $\lzeroord$ of 
states with only finitely many blocks.
The ultimate goal of this section is to prove Proposition \ref{prop:tilt_and_shift_l0}.

\smallskip

In Section \ref{sec:expsizebiased} we recall the notion of a size-biased rearrangement of $\um \in \lzeroord$
and how this notion can be extended to $\ltwoord$ using an appropriate family of independent exponential random variables.

\smallskip

In Section \ref{section:interval_coalescent_repr} we define the interval coalescent with linear deletion 
(or briefly $\mathrm{ICLD}(\lambda)$) and show that an $\mathrm{ICLD}(\lambda)$ with a size-biased initial state
gives a representation of the $\mathrm{MCLD}(\lambda)$.

\smallskip

In Section \ref{section_particle_representation} we define our particle representation and show that
if the initial heights of particles form an appropriate family of independent exponential random variables, then
we obtain a representation of of the $\mathrm{ICLD}(\lambda)$ with a size-biased initial state.

\smallskip

In Section \ref{tiltproofl1} we show that a copy of the particle system is embedded in the tilt-and-shift
representation introduced in Definition \ref{def_tilt_shift_step_function}.

\subsection{Construction of size-biased sequences
using independent exponential random variables}
\label{sec:expsizebiased}

Now we recall some useful definitions from \cite[Section 3.3]{aldous_mc}.

\begin{definition}\label{def:size_biased_for_lzero}
Let $\um=(m_1, m_2, \dots, m_n) \in \lzeroord$. 
A random total linear order $\prec$ on $[n]$
is a \textit{size-biased order} 
(with respect to $\um$)
if for each 
permutation $i_1, i_2, \dots, i_n$ of $[n]$, 
\begin{equation}\label{sizebiased}
\P(i_1\prec i_2\prec\dots\prec i_n)=
\prod_{r=1}^n \frac{m_{i_r}}{m_{i_r}+m_{i_{r+1}}+\dots+m_{i_n}}.
\end{equation}
We say that $\underline{b}=(b_1,\dots,b_n)$ is the size-biased reordering of $\um$ if $b_{k}=m_{i_k}$, where
 $i_1\prec i_2\prec\dots\prec i_n$.
\end{definition}

\begin{definition}\label{def_size_biased_with_exponentials}
Suppose $\um \in\linford$. 
Let $E_i\sim\Exp(m_i)$ independently for each $i$. 
Define a random linear order on $\N$ by $i\prec j$ if and only if $E_i<E_j$.
\end{definition}

The proof of the next claim follows from the memoryless property of exponential random variables and we omit it.
\begin{claim}\label{claim:sb}
The law of $\prec$ introduced in Definition \ref{def_size_biased_with_exponentials}
 is size-biased (with respect to the sizes $m_i$),
in the sense that for any $n \in \N$ the restriction of $\prec$ to $[n]$ is size-biased 
with respect to $(m_1, m_2, \dots, m_n)$, c.f.\
 Definition
\ref{def:size_biased_for_lzero}.
\end{claim}

\begin{remark}\label{remark_size_biased_extensions}
\begin{enumerate}[(i)]
\item \label{dense_size_biased} 
 There is a smallest element with respect to the order $\prec$
if and only if $\um \in \loneord$. 
If $\um \in \linford \setminus \loneord$  then  the values $E_i$ are dense in $\R_+$, see Lemma \ref{lemma:dense}. 
\item \label{exponential_size_biased_ii}
 The excursions (see Definition \ref{def:excursion}) of the random function $f_0$ defined in Definition \ref{def:f_zero_from_exponentials} appear in size-biased order.
\end{enumerate}
\end{remark}

\subsection{A size-biased interval representation of $\mathrm{MCLD}(\lambda)$ }
\label{section:interval_coalescent_repr}

We introduce a related process $\bb_t, t\geq 0$
in which only neighbouring blocks are allowed to merge,
and only the leftmost block is allowed to be deleted.


\begin{definition}\label{def:ic_state_space}
With some abuse of notation, 
we denote by $\lzero=\bigcup_{n\geq 0} \R_{>0}^n$ 
the space of finite sequences with positive entries.
Given 
$\ub=(b_1,\dots,b_n) \in \lzero$, let $\Down(\ub) \in \lzeroord$ denote the reordering of $\ub$
into non-increasing order. 
\end{definition}

\begin{definition}[Interval coalescent with linear deletion, $\mathrm{ICLD}(\lambda)$]
\label{def:interval_coalescent}
The state space of the continuous-time Markov process $(\bb_t)$ is $\lzero$. 
   The dynamics consist of coalescence and deletion:
\begin{enumerate}[(i)]
\item  If $\ub,\ub' \in \lzero$ where 
$\ub' \in \R_{>0}^{n-1}$ arises from $\ub \in \R_{>0}^{n}$ by merging the blocks
$b_k$ and $b_{k+1}$ for some $1 \leq k <n$;
that is
\begin{equation}\label{icchange}
b'_i=\begin{cases}
b_i, &i=1,2,\dots, k-1\\
b_{k}+b_{k+1}, &i=k\\
b_{i+1}, &i=k,\dots,n-1
\end{cases}
\end{equation}
  then the rate of the transition from $\ub$ to $\ub'$ 
  is 
  \begin{equation} \label{eq:merge_rate_ic}
   \mathcal{R}_{IC}(\ub,\ub')=b_k \cdot \sum_{i=k+1}^n b_{i}.
  \end{equation}
\item If $\ub,\ub' \in \lzero$ where $\ub'\in 
\R_{>0}^{n-1}$ arises from $\ub \in \R_{>0}^{n}$ by
 deleting the leftmost block, i.e., 
 \[b'_i=b_{i+1},\qquad 1 \leq i \leq n-1\] 
then the rate of this transition is 
\begin{equation} \label{eq:deletion_rate_ic}
\mathcal{R}_{IC}(\ub,\ub')=\lambda \cdot \sum_{i=1}^n b_{i}.
\end{equation}
\end{enumerate}
All other rates of the $\mathrm{ICLD}(\lambda)$ are $0$.
\end{definition}

We are now ready to state the main result of Section \ref{section:interval_coalescent_repr}.

\begin{theorem}
\label{thm:intervalcoalescent}
Let $\lambda \geq 0$ and $\um \in \lzeroord$. 
Let $\bb_0$ be a random size-biased reordering (see Definition \ref{def:size_biased_for_lzero}) of $\um$. 
Let $\bb_t, t\geq 0$ be an  $\mathrm{ICLD}(\lambda)$
started from the initial state $\bb_0$. Then
\begin{enumerate}[(i)]
\item  the law of the process $\Down(\bb_t), t\geq 0$ 
is that of the $\mathrm{MCLD}(\lambda)$ process $\bm_t, t\geq 0$ started from $\bm_0=\um$, and
\item given the $\sigma$-algebra $\sigma( \Down(\bb_{s}), 0 \leq s \leq t)$,
the conditional distribution of $\bb_{t}$ is that of a 
random size-biased reordering of $ \Down(\bb_{t})$.
\end{enumerate}
\end{theorem}
We will prove Theorem \ref{thm:intervalcoalescent} in Section \ref{subsection:proof_intervalcoalescent}.

\begin{remark}
\label{remark_interval_coalescent_l1}
\begin{enumerate}[(i)]
\item To the best of our knowledge, even the $\lambda =0$ case of Theorem \ref{thm:intervalcoalescent} is novel, and it gives an interval coalescent representation
of the multiplicative coalescent process (on the state space $\lzeroord$).
\item  Theorem \ref{thm:intervalcoalescent} generalizes in 
a natural way to any initial condition $\um \in \loneord$. For initial conditions with infinite total mass,
the situation is more complicated since 
under the natural extension of the concept of size-biased order
(see Definition \ref{def_size_biased_with_exponentials}) there is no 
smallest element of the order, and any
two elements are separated by infinitely many other elements in the order,
c.f.\ Remark \ref{remark_size_biased_extensions}\eqref{dense_size_biased}.
\end{enumerate}
\end{remark}

\subsubsection{Proof of Theorem \ref{thm:intervalcoalescent} }
\label{subsection:proof_intervalcoalescent}

\begin{lemma}\label{lemma_total_jump_rates}
Let $\ub\in\lzero$ and $\um=\Down(\ub)$. 
Then the total jump rate of the $\mathrm{MCLD}(\lambda)$
out of the state $\um$ is the same as the total 
jump rate of the $\mathrm{ICLD}(\lambda)$ out of the state $\ub$. 
\end{lemma}
\begin{proof}
The total jump rate in 
the $\mathrm{MCLD}(\lambda)$ is $\lambda\sum m_i+\sum_{i<j}m_i m_j$.
Since $\um$ is a reordering of $\ub$, 
we obtain the same total rate for the $\mathrm{ICLD}(\lambda)$ by adding 
the deletion rate in (\ref{eq:deletion_rate_ic})
to the sum over $1\leq k<n$ of the
coalescence rate in (\ref{eq:merge_rate_ic}).
\end{proof}
%


%
%

\begin{lemma}\label{lemma_two_procedures}
Let $\um\in\lzeroord$. 
The following two procedures give the same
distribution of $\ub'$:
\begin{itemize}
\item[(1)]
Let $\ub$ be a size-biased reordering of $\um$,
and then, given $\ub$, obtain $\ub'$ from $\ub$ 
by performing a single step of the jump chain of the $\mathrm{ICLD}(\lambda)$. 
\item[(2)]
Obtain $\um'$ from $\um$ by performing 
a single step of the jump chain of the $\mathrm{MCLD}(\lambda)$,
and then, given $\um'$, let $\ub'$ be a size-biased reordering of $\um'$.
\end{itemize}
In particular, in (2) it is the 
case that the conditional distribution of $\ub'$
given $\Down(\ub')$ is that of a size-biased reordering
of $\Down(\ub')$, so the same is also true in (1).
\end{lemma}

Before proving Lemma \ref{lemma_two_procedures}, we use it to prove Theorem 
\ref{thm:intervalcoalescent}.

\begin{proof}[Proof of Theorem 
\ref{thm:intervalcoalescent}]
Assume that the initial state has
$n$ non-zero blocks. Then the process $\bb_t$
will make $n$ jumps before being absorbed in the 
all-$0$ state. 
Let $\tau_0=0$ and let $\tau_1, \dots, \tau_n$ be the jump 
times of the process, 
so that 
$\left(\bb_{\tau_0}, \bb_{\tau_1}, \dots, \bb_{\tau_n}\right)$
is the jump-chain.

Let $\bm_t=\Down(\bb_t)$. 
 We claim that
for $0\leq k\leq n$, the following properties hold:
\begin{itemize}
\item[P1($k$):]
$\left(\bm_{\tau_0}, \bm_{\tau_1},\dots, \bm_{\tau_k}\right)$
has the distribution of the first $k$ steps
of the jump-chain of the $\mathrm{MCLD}(\lambda)$. 
\item[P2($k$):]
Given $\bm_{\tau_1},\dots, \bm_{\tau_k}$,
the distribution of $\bb_{\tau_k}$ is that of a 
random size-biased reordering of $\bm_{\tau_k}$. 
\end{itemize}

As soon as we prove these properties, the proof of Theorem 
\ref{thm:intervalcoalescent} follows using
 Lemma
\ref{lemma_total_jump_rates} and the memoryless property.

We now prove P1($k$) and P2($k$) by induction on $k$.
The case $k=0$ is immediate from the definition 
of $\bb_0$ in the statement of Theorem \ref{thm:intervalcoalescent}.

For the induction step, suppose that P1($k$) and P2($k$) hold for a given $k$ with $0\leq k<n$. 
Now we condition on $\bm_{\tau_1},\dots, \bm_{\tau_k}$ and apply Lemma \ref{lemma_two_procedures}
with the choices
\[ \um=\bm_{\tau_k}, \qquad \ub=\bb_{\tau_k}, \qquad \um'=\bm_{\tau_{k+1}}, \qquad \ub'=\bb_{\tau_{k+1}}. \]

From P2($k$), we know that given $\bm_{\tau_1},\dots, \bm_{\tau_k}$, the distribution of $\bb_{\tau_k}$
is a random size-biased reordering of $\bm_{\tau_k}$;
then, as in (1) of Lemma \ref{lemma_two_procedures},
we obtain $\bb_{\tau_{k+1}}$ by taking a step of 
the $\mathrm{ICLD}(\lambda)$ chain from $\bb_{\tau_k}$. 

Hence, using the final observation of Lemma 
\ref{lemma_two_procedures}, the distribution of 
$\bb_{\tau_{k+1}}$ given 
$\bm_{\tau_1},\dots, \bm_{\tau_k}, \bm_{\tau_{k+1}}$,
is a random size-biased reordering of $\bm_{\tau_{k+1}}$,
and P2($k+1$) holds. 

But also, comparing with (2) of Lemma \ref{lemma_two_procedures}, 
the distribution we obtain of 
$\bm_{\tau_{k+1}}=\Down(\bb_{\tau_{k+1}})$
is the same as we would have obtained by taking a 
step of the $\mathrm{MCLD}(\lambda)$ chain from $\bm_{\tau_k}$. 
So indeed we can extend P1($k$) to give P1($k+1$) also.

This completes the induction step. In this way 
we obtain that P1($n$) and P2($n$) hold, and this completes
the proof of Theorem \ref{thm:intervalcoalescent}.
\end{proof}

 
 \medskip
 
\begin{proof}[Proof of Lemma \ref{lemma_two_procedures}] 
To avoid awkward notation, we will write the proof under the extra
assumption that $\um$ has all its block sizes
distinct, and so does $\um'$ for each $\um'$
that can be obtained from $\um'$ by a coalescence
step. The general statement
can then be easily obtained by continuity. 

Under the above stated extra assumption we can write the transition rates
of the $\mathrm{MCLD}(\lambda)$ chain as follows. 
If $\um'$ is obtained from 
$\um$ by merging the two blocks of size $m_I$ and $m_J$,
then the rate of the transition from $\um$ to $\um'$ is
\begin{equation}\label{mcld_rate_merge_def}
\mathcal{R}_{MC}(\um,\um')=m_I m_J.
\end{equation}
If $\um'$ is obtained from $\um$ by deleting the block of size $m_I$,
then the rate of the transition from $\um$ to $\um'$ is 
\begin{equation}\label{mcld_rate_delete_def}
\mathcal{R}_{MC}(\um,\um')=\lambda m_I.
\end{equation}

Let $\pi_{\um}$ denote the probability distribution on $\lzero$ 
which arises from the size-biased reordering of
$\um$. Again if $\um$ has distinct block sizes, we can 
write
\begin{equation}\label{eq:def_pi_sizebiased}
\pi_{\um}(\ub)= \ind[\, \Down(\ub)=\um \, ]\cdot \prod_{i=1}^n \frac{b_i}{\sum_{j=i}^n b_j }.
\end{equation}

Now, to prove Lemma \ref{lemma_two_procedures}
it is enough to demonstrate the following claim:
if $\um \in \lzeroord$ and $\ub' \in \lzero$,
with $\um'=\Down(\ub')$, then
\begin{equation}\label{ic_mc_identity}
\sum_{\ub:\, \Down(\ub)=\um} \pi_{\um}(\ub)\mathcal{R}_{IC}(\ub,\ub')=
\mathcal{R}_{MC}(\um,\um') \pi_{\um'}(\ub').
\end{equation}
(Rescaled by the total jump rate, 
which by Lemma \ref{lemma_total_jump_rates}
are the same for the $\mathrm{ICLD}(\lambda)$ and the $\mathrm{MCLD}(\lambda)$,
the left side represents the probability  
of obtaining a given value $\ub'$ 
using procedure (1) in Lemma \ref{lemma_two_procedures},
while the right side represents the same for  
procedure (2).)

\medskip

To prove (\ref{ic_mc_identity}), there are
two cases that we have to handle, corresponding to coalescence and deletion.

We first treat the case of coalescence, that is we assume that
 the state $\Down(\ub')=\um'$ arises from $\um$ by
  merging some blocks with sizes $m_I$ and $m_J$.
Then $\ub'$ has an interval of size $m_I+m_J$, say $b'_k=m_I+m_J$.
There are exactly two reorderings $\ub$ of $\um$ for which $\mathcal{R}_{IC}(\ub,\ub')>0$, namely
\begin{align*}
 \ub^1 &=(b'_1, \dots, b'_{k-1}, m_I, m_J, b'_{k+1}, \dots, b'_n ),\\
 \ub^2 &=(b'_1, \dots, b'_{k-1}, m_J, m_I, b'_{k+1}, \dots, b'_n ). 
\end{align*}
Now let us rewrite the two sides of \eqref{ic_mc_identity}. The left side is
\begin{multline*}
\sum_{\ub: \Down(\ub)=\um} \pi_{\um}(\ub)\mathcal{R}_{IC}(\ub,\ub')=
\pi_{\um}(\ub^1)\mathcal{R}_{IC}(\ub^1,\ub')+\pi_{\um}(\ub^2)\mathcal{R}_{IC}(\ub^2,\ub')
\stackrel{ \eqref{eq:merge_rate_ic},\eqref{eq:def_pi_sizebiased} }{=}
\\
\left( \prod_{i=1}^{k-1} \frac{b'_i}{\sum_{j=i}^n b'_j } \right) 
 \frac{m_I}{\sum_{j=k}^n b'_j }
 \frac{m_J}{\sum_{j=k}^n b'_j -m_I} \cdot \\
 \left(\prod_{i={k+1}}^{n} \frac{b'_i}{\sum_{j=i}^n b'_j }\right) 
 \cdot m_I \cdot \left(\sum_{j=k}^n b'_j -m_I\right)+ \\
 \left( \prod_{i=1}^{k-1} \frac{b'_i}{\sum_{j=i}^n b'_j } \right) 
 \frac{m_J}{\sum_{j=k}^n b'_j } 
 \frac{m_I}{\sum_{j=k}^n b'_j -m_J} \cdot \\
 \left(\prod_{i={k+1}}^{n} \frac{b'_i}{\sum_{j=i}^n b'_j }\right) 
 \cdot m_J \cdot \left(\sum_{j=k}^n b'_j -m_J\right).
\end{multline*}
Using \eqref{mcld_rate_merge_def} and \eqref{eq:def_pi_sizebiased}, the right side can be rewritten as
\begin{equation*}
\mathcal{R}_{MC}(\um,\um') \pi_{\um'}(\ub') \stackrel{
 }{=} m_I  m_J \cdot
 \left( \prod_{i=1}^{k-1} \frac{b'_i}{\sum_{j=i}^n b'_j } \right) 
 \frac{m_I+m_J}{\sum_{j=k}^n b'_j }
 \left(\prod_{i={k+1}}^{n} \frac{b'_i}{\sum_{j=i}^n b'_j }\right) .
\end{equation*}
These are the same, so \eqref{ic_mc_identity} holds in the case of coalescence.

\medskip

We now turn to the case of deletion; that is, we assume that the state $\Down(\ub')=\um'$ arises from $\um$ by 
deleting a block of size $m_I$.
There is one rearrangement $\ub$ of $\um$ for which 
$\mathcal{R}_{IC}(\ub,\ub')>0$, namely
\begin{align}
\ub^0 &=(m_I, b'_1, \dots, b'_n ).
\end{align}
Thus 
\begin{align*}
\sum_{\ub:\um=\Down(\ub)} \pi_{\um}(\ub)\mathcal{R}_{IC}(\ub,\ub')
&=
\pi_{\um}(\ub^0)\mathcal{R}_{IC}(\ub^0,\ub')
\\
&\stackrel{ \eqref{eq:deletion_rate_ic}, \eqref{eq:def_pi_sizebiased}  }{=}
\frac{m_I}{m_I + \sum_{j=1}^n b'_j  } \pi_{\um'}(\ub') \cdot \lambda \cdot 
\left(m_I+ \sum_{j=1}^n b'_j \right)\\
&=
\lambda m_I  \cdot \pi_{\um'}(\ub')
\\
&=\mathcal{R}_{MC}(\um,\um') \pi_{\um'}(\ub')
\end{align*}
So \eqref{ic_mc_identity} holds in this case also.
This completes the proof of  Lemma \ref{lemma_two_procedures}.
\end{proof}

\subsection{Particle representation}
\label{section_particle_representation}

The representation of the $\mathrm{MCLD}(\lambda)$
in Section  \ref{section:interval_coalescent_repr} moved some of the randomness
of the process into the choice of an initial condition
(using a size-biased reordering). Thereafter the possible
transitions of the process were restricted
(only neighbouring blocks were allowed to merge,
and only the first block could be deleted).

In this section we take this to an extreme by giving
a natural construction of the process in which 
\emph{all} the randomness is in the initial condition; 
the evolution of the process thereafter is entirely
deterministic, but nonetheless the process 
projects to the $\mathrm{MCLD}(\lambda)$. 
Such processes might be called ``rigid".


 \begin{definition}[Particle system]
 \label{def_particle_rep}
 Let $\um=(m_1,\dots, m_n) \in \lzeroord$.
 Let $E_1, \dots,$ $E_n$ be independent
 with $E_i\sim \Exp(m_i)$. 
For $i=1,\dots, n$, 
let $Y_i(0)=-E_i$ be the initial height of a 
particle $i$
with mass $m_i$.
The heights of particles evolve over time; we describe the joint evolution of the heights  $Y_1(t), \dots, Y_n(t)$ using a system of ordinary differential equations.
  
Analogously to the definition  \eqref{mu_exp_point_measure_def}, we define 
 \begin{equation}\label{def_eq_mu_t}
 \mu_t=\sum_{i=1}^{n} m_i \cdot \delta_{Y_i(t)}.
 \end{equation}
 The system of differential equations governing $Y_1(t), \dots, Y_n(t)$ is
\begin{equation}\label{particle_dynamics_lambda}
\frac{\mathrm{d}}{\mathrm{d}t} Y_i(t)= 
\lambda \cdot \ind  [ \, Y_i(t)<0 \, ]+\mu_t( Y_i(t),0).
\end{equation}
We say that the particle $i$ ``dies'' at time $t_i$, where $t_i$ is defined by 
\begin{equation}\label{def_eq_t_i_death_time}
t_i:= \min\{\, t \, : \, Y_i(t)=0 \, \}.
\end{equation}

A ``time-$t$ block'' consists of the union of all the particles 
that share the same (strictly negative) height at time $t$. 
\end{definition}
 In words, particles start at negative locations and move up. If a particle reaches zero 
  then it stops there and dies.
  Before it dies, the speed of a particle  is equal to $\lambda$ plus the total weight
 of particles strictly above it and strictly below zero. 
Observe that if two particles ever reach the same height, then they stay together forever.

\medskip

Recall the definition of the $\mathrm{ICLD}(\lambda)$ process from Definition \ref{def:interval_coalescent}.

\begin{proposition}
\label{prop:particles}
Let $\bb_t\in\lzero$ 
be the vector of sizes of the time-$t$
blocks of the particle system, in decreasing
order of their height. 
Then the process $\bb_t, t\geq 0$ has the law 
of $\mathrm{ICLD}(\lambda)$, started from an initial
state $\bb_0$.
\end{proposition}
Before proving Proposition \ref{prop:particles}, let us state the following corollary.
\begin{corollary}
\label{corollary_particle_mcld}
 The process $\bm_t=\Down(\bb_t)$ 
has the law of MCLD($\lambda$) started from the
state $\um$. 
\end{corollary}
\begin{proof}[Proof of Corollary \ref{corollary_particle_mcld}]
It follows from Definition \ref{def_particle_rep} and Claim \ref{claim:sb} that
$\bb_0$ is a size-biased reordering
of $\um$. Now the statement of  Corollary \ref{corollary_particle_mcld}
 follows from Theorem \ref{thm:intervalcoalescent}
and Proposition \ref{prop:particles}.
\end{proof}

Before we prove Proposition \ref{prop:particles}, let us introduce a useful notation.

\begin{definition}\label{def_heightgap_distribution}
Given $\ub=(b_1,\dots,b_n) \in \lzero$ we say that the the random vector 
\[(Y^{(1)}, \dots, Y^{(n)})\]
 has law
$\expheightgaps(\ub)$ if $0>Y^{(1)}>\dots > Y^{(n)}$,
\begin{align}
\nonumber
-Y^{(1)}&\sim \Exp(b_1+b_2+\dots+b_n)\\
\nonumber
Y^{(1)}-Y^{(2)}&\sim \Exp(b_2+\dots+b_n)\\
\nonumber
&\vdots\\
\label{heightgaps}
Y^{(k)}-Y^{(k+1)}&\sim 
\Exp\left(\sum_{i=k+1}^n b_i\right)\\
\nonumber
&\vdots\\
\nonumber
Y^{(n-1)}-Y^{(n)}&\sim \Exp(b_n),
\end{align}
and all these gaps are independent.  
\end{definition}

\begin{proof}[Proof of Proposition \ref{prop:particles}]
Suppose there are $n(t)$ time-$t$ blocks. 
Let 
\begin{equation}\label{rearranged_particle_eq_def}
\left( Y^{(1)}(t), \dots, Y^{(n(t))}(t) \right)
\end{equation}
be 
the vector of their heights in decreasing order.
Then  $\mu_t$, defined at \eqref{def_eq_mu_t},
is also given by
\begin{equation}\label{alt_mu_t}
\mu_t=\sum_{i=1}^{n(t)} \left(\bb_t\right)_i \delta_{Y^{(i)}(t) }.
\end{equation}

Observe that in particular, $\left(Y^{(1)}(0), \dots, Y^{(n)}(0)\right)$ is the decreasing rearrangement of the initial 
heights of the particles, and $Y^{(i)}(0)$ 
is the initial height of a particle of mass 
$(\bb_0)_i$. 

From here onwards, let us condition 
throughout on $\bb_0=\ub=(b_1, \dots, b_n)$.

By repeatedly applying the memoryless
property of the exponential distribution, 
we see that given $\bb_0=\ub$, we have
\begin{equation}\label{height_gaps_exp_indeed}
 \left(Y^{(1)}(0), \dots, Y^{(n)}(0)\right) \sim \expheightgaps(\ub), 
 \end{equation}
where $\expheightgaps(\cdot)$ was introduced in Definition \ref{def_heightgap_distribution}.

Let us consider the time $\tau_1$ at which the first
jump of the process $(\bb_t)$ occurs. 
This jump may be either a deletion 
(if a block reaches height 0 and so dies)
or it may be a coalescence 
(if two blocks reach the same height). 
Between times $0$ and $\tau_1$, the $k$th 
highest particle moves at speed 
$\lambda+\sum_{i=1}^{k-1}b_i$.
Hence the gap
$Y^{(k)}(t)-Y^{(k+1)}(t)$ decreases at rate $b_k$
for $1\leq k\leq n-1$, and the gap $-Y^{(1)}(0)$
decreases at rate $\lambda$. 

Let us therefore define 
\begin{align}
\label{time_first_reaches_top}
t^{(0)}&:= \frac{Y^{(1)}(0)}{\lambda}
\sim \Exp\big(\lambda(b_1+\dots+b_n)\big)\\
\label{time_k_merges_with_k_plus_one}
t^{(k)}&:=\frac{Y^{(k)}(0)-Y^{(k+1)}(0)}{b_k}
\sim \Exp\left(b_k\sum_{i=k+1}^n b_i\right)
\text{ for }1\leq k\leq n-1.
\end{align}
These variables are all independent. 

Then $\tau_1=\min_{0\leq k\leq n-1} t^{(k)}$,
and we have
\[
\tau_1 \sim \Exp\left( \lambda\sum_{i=1}^n b_i
+\sum_{1\leq i<j\leq n} b_i b_j\right).
\]
Further, let $\kappa=\argmin_{0\leq k\leq n-1}t^{(k)}$.
This argmin is uniquely defined with probability 1. 
If $\kappa=0$, then the event at time $\tau_1$ 
is a deletion of the highest block, while
if $\kappa=k\geq 1$ then the event at time $\tau_1$
is a coalescence of the $k$th and $(k+1)$st highest blocks.
Then $\kappa$ and $\tau_1$ are independent; 
the probability that $\kappa=k$ is proportional 
to $\lambda\sum_{i=1}^n b_i$ for $k=0$, 
and to $b_k\sum_{i=k+1}^n b_i$ for 
$1\leq k\leq n-1$. 

Comparing the previous paragraph with the
rates of the $\mathrm{ICLD}(\lambda)$ at \eqref{eq:merge_rate_ic}
and \eqref{eq:deletion_rate_ic},
we see that the distribution of the new
state $\ub'=\bb_{\tau_1}$,
together with the time $\tau_1$ at which it occurs,
is the same as for the $\mathrm{ICLD}(\lambda)$ process.

From here on, let us condition further on 
$\tau_1$ and $\bb_{\tau_1}=\ub'$ as well as on
$\bb_0=\ub$. Again applying
the memoryless property, the conditional
distribution of the remaining height gaps
just before $\tau_1$,
excluding the one which reaches $0$ at that time,
is unchanged from what it was at time $0$. 

The heights and masses of blocks at time $\tau_1$ are given by
\begin{equation}\label{Y_i_after_jump}
Y^{(i)}(\tau_1)=\begin{cases}
Y^{(i)}(\tau_1-)&\text{ for }i\leq \kappa,\\
Y^{(i+1)}(\tau_1-)&\text{ for }i>\kappa,
\end{cases}
\end{equation}
and
\begin{equation}\label{b_i_after_jump}
b_i'=\begin{cases}
b_i&\text{ for }i<\kappa,\\
b_\kappa+b_{\kappa+1}&\text{ for }i=\kappa,\\
b_{i+1}&\text{for }i>\kappa.
\end{cases}
\end{equation}
We therefore obtain that 
\begin{align*}
-Y^{(1)}(\tau_1) &\sim 
\Exp\left(\sum_{i=1}^{n-1}b_i'\right)
\\
Y^{(k)}(\tau_1)-Y^{(k+1)}(\tau_1) 
&\sim
\Exp\left(\sum_{i=k+1}^{n-1}b'_i\right)
\text{ for } k=1,\dots, n-2,
\end{align*}
and all these gaps are independent. 
Thus we have shown that 
\[ \left(Y^{(1)}(\tau_1), \dots, Y^{(n-1)}(\tau_1)\right) \sim \expheightgaps(\ub'). \]

So we can repeat the argument above starting
from time $\tau_1$, until the time $\tau_2$ of the next jump, obtaining that $\bb_t$ continues to evolve as
an $\mathrm{ICLD}(\lambda)$ process. Continuing recursively
until the time of the $n$th jump 
(when the last death occurs and $\bb_t$ becomes identically zero), we obtain the desired result.
The proof of Proposition \ref{prop:particles} is complete. 
\end{proof}

We will now  prove a corollary of the results and methods developed in Sections \ref{section:interval_coalescent_repr}
 and \ref{section_particle_representation}. This corollary will only be used in Section \ref{section_applications}.

\medskip

We consider the particle system introduced in Definition \ref{def_particle_rep}.
Let $\bb_t\in\lzero$ 
be the vector of sizes of the time-$t$
blocks of the particle system, in decreasing
order of their height and let $\bm_t=\Down(\bb_t)$.
Let $\tau_0=0$ and denote by $\tau_k, k \geq 1$ the $k$th jump time of the process $(\bb_t)$.
Recall the notion of the point measure $\mu_t$ from \eqref{def_eq_mu_t} (or from \eqref{alt_mu_t}) 
and notion of $\exmeasure(\um)$ from
Definition \ref{def_exp_measure_law}.

\begin{corollary}\label{corollary_exp_measure_mcld}
\begin{enumerate}[(i)]
\item \label{cor_exp_meas_mcld_discrete} For any $k \geq 0$,  the conditional distribution of $\mu_{\tau_k}$ given
$\left( \bm_{\tau_i}, 0 \leq i \leq k \right)$ is $\exmeasure(\bm_{\tau_k})$.
\item \label{cor_exp_meas_mcld_continuous} For any $t \geq 0$,  the conditional distribution of $\mu_{t}$ given
$\left( \bm_{s}, 0 \leq s \leq t \right)$ is $\exmeasure(\bm_t)$.
\end{enumerate}
\end{corollary}

Before we prove Corollary \eqref{corollary_exp_measure_mcld}, let us state the key claim needed for the proof.
The proof of this claim follows from the memoryless property  and we omit it.
\begin{claim}\label{claim_exp_ord_size_biased}
Let $\um=(m_1,\dots,m_n) \in \lzeroord$ and let $\bb$ be a size-biased reordering of $\um$.
Conditioned on $\bb=\ub$, let 
 $(Y^{(1)}, \dots, Y^{(n)}) \sim \expheightgaps(\ub)$ (c.f.\ Definition \ref{def_heightgap_distribution}). 
 Then the law of $\mu=\sum_{i=1}^n (\bb)_i \delta_{Y^{(i)} } $ is $\exmeasure(\um)$ (c.f.\ Definition
 \ref{def_exp_measure_law}).
\end{claim}

\begin{proof}[Proof of Corollary \ref{corollary_exp_measure_mcld}]

Given $\bm_{\tau_1},\dots, \bm_{\tau_k}$,
the distribution of $\bb_{\tau_k}$ is that of a 
random size-biased reordering of $\bm_{\tau_k}$ 
(c.f.\ P2($k$) in the proof of Theorem \ref{thm:intervalcoalescent}
and
the proof of Corollary \ref{corollary_particle_mcld}).

In the proof of Proposition \ref{prop:particles} we have seen that given 
$\bb_{\tau_1},\dots, \bb_{\tau_k}$, the conditional joint distribution of the heights is
\[\left(Y^{(1)}(\tau_k), \dots, Y^{(n(\tau_k))}(\tau_k) \right) \sim \expheightgaps(\bb_{\tau_k}),
\qquad n(\tau_k)=n-k.\]

Combining the above observations with Claim \ref{claim_exp_ord_size_biased} we obtain that the statement of Corollary \ref{corollary_exp_measure_mcld}\eqref{cor_exp_meas_mcld_discrete} indeed holds.

Now Corollary \ref{corollary_exp_measure_mcld}\eqref{cor_exp_meas_mcld_continuous} follows from
the combination of part \eqref{cor_exp_meas_mcld_discrete}, the observation
that between the jumps of the process $(\bb_t)$ the height gaps between particles decrease at constant speed
and the memoryless property of the height gap distribution (c.f.\ Definition \ref{def_heightgap_distribution}).
\end{proof}

\subsection{Tilt-and-shift representation}
\label{tiltproofl1}

In this section we connect the 
particle system introduced in Definition \ref{def_particle_rep}
to the ``tilt-and-shift" dynamics 
presented in Definition \ref{def_tilt_shift_step_function} and prove
Proposition \ref{prop:tilt_and_shift_l0}.

\begin{definition}\label{def:EX}
Assume  $g:[0,\infty)\to\R\cup\{-\infty\}$ has only finitely many  excursions.
Denote by \[\EX(g)\in \lzero\] the sequence
of the lengths of the excursions of $g$, in order of appearance. 
\end{definition}

Let $\um=(m_1, \dots, m_n)\in \lzeroord$. 
Let $E_i\sim \Exp(m_i)$ independently for $i=1,\dots, n$. 

Consider the particle system of Section \ref{section_particle_representation} in which
a particle of size $m_i$ starts at height $Y_i(0)=-E_i$
for each $i=1,\dots, n$. 

Define the function $g_0\equiv f_0$ as at 
\eqref{f0def_b} or \eqref{f0def_for_lzeroord}, and then 
$g_t, t>0$ using the tilt-and-shift procedure of Definition
\ref{def_tilt_shift_step_function}.

\begin{proposition}
\label{prop:tilt_and_shift_ic} 
Let $\bb_t\in\lzero$ be the vector of sizes of
the time-$t$ blocks of the particle system 
in decreasing order of height. 
Then
\begin{equation}
\label{eq:tilt_particles}
\left(\bb_t, t\geq 0\right) = \left(\EX(g_t), t\geq 0\right) 
\text{ with probability }1. 
\end{equation}
\end{proposition}
Then from Proposition \ref{prop:particles}, Corollary \ref{corollary_particle_mcld}
and Theorem \ref{thm:intervalcoalescent}
we can immediately deduce the following result:
\begin{corollary}
\label{cor:tilt_ic}
$\left(\EX(g_t), t\geq 0\right)$ has the law of ICLD($\lambda$) 
started from an initial state which is a size-biased 
reordering of $\um$. 
Hence $\left(\ORDX(g_t), t\geq 0\right)$ has the law of MCLD($\lambda$) 
started from $\um$ and thus Proposition \ref{prop:tilt_and_shift_l0} holds.
\end{corollary}

\begin{proof}[Proof of Proposition \ref{prop:tilt_and_shift_ic}]

We will show how to find a copy of the particle system
embedded within the tilt-and-shift process. 
Namely, the excursions of $g_t$ correspond to the 
time-$t$ blocks in the particle system. 
Specifically, an excursion $[x, x')$ of $g_t$
corresponds to a time-$t$ block whose
size is the length $x'-x$ of the excursion, and 
whose height is the level $g_t(x)$ of the excursion.

The idea is that the ``tilt" part of the
construction produces an upward drift of the 
excursion levels which corresponds to the upward movement
of the blocks in the particle system; 
this drift causes neighbouring excursions to merge,
corresponding to the merging of blocks 
in the particle system when they reach the same height. 
Meanwhile the ``shift" mechanism
which removes the leftmost excursion when its level reaches
0 corresponds to the death of a block when it reaches
height 0 in the particle system.

\medskip

Recall that $Y^{(1)}(t),\dots,Y^{(n(t))}(t)$ are the 
heights of the time-$t$ blocks of the particle system
in decreasing order, and $\bb_t$ is the vector of the
sizes of those blocks (in decreasing order of height).

Let us denote $\bb_0=(b_1, \dots, b_n)$. Write also
\begin{equation}\label{eq:def:minima_x_k}
 x_0=0, \quad x_1=b_1, \quad x_2=b_1+b_2, \quad \dots,\quad  x_n=b_1+\dots+b_n.
 \end{equation}
Then the excursions of $g_0$ are the intervals
\begin{equation}\label{g0excursions}
[x_0, x_1), \quad [x_1, x_2), \quad \dots, \quad [x_{n-1}, x_n)
\end{equation}
and the levels of these excursions are
\begin{equation}\label{gYsame}
g_0(x_0)=Y^{(1)}(0), \quad \dots,\quad g_0(x_{n-1})=Y^{(n)}(0)
\end{equation}
(with $g_0(x_n)=-\infty$). 

Recall the time $\tau_1$, the time of the first
merging or deletion in the particle system, defined in
the proof of Proposition \ref{prop:particles},
which we can also write as
\begin{equation}\label{tau1again}
\tau_1=\inf\big\{t:Y^{(1)}(t-)=0 
\text{ or } 
Y^{(k)}(t-)=Y^{(k+1)}(t-) \text{ for some }1\leq k<n\big\}.
\end{equation}
Since each block moves upwards at a rate equal
to the sum of the sizes of blocks above it plus $\lambda$,
we have that for $1\leq k\leq n$ and $t\in[0,\tau_1)$,
\begin{align}
\nonumber
\frac{\mathrm{d}}{\mathrm{d}t} Y^{(k)}(t) 
&=b_1+\dots+b_{k-1}+\lambda\\
&=x_{k-1} + \lambda.
\label{dY}
\end{align}

But also if we define
\begin{equation}\label{tildetau1def}
\tilde{\tau}_1=
\inf\big\{t:g_{t-}(0)=0
\text{ or } 
g_{t-}(x_{k-1})=g_{t-}(x_k) \text{ for some }1\leq k<n\big\},
\end{equation}
then by the definition of the tilt mechanism
in Definition \ref{def_tilt_shift_step_function},
we have that for $1\leq k\leq n$ and
$t\in[0,\tilde{\tau}_1)$,
\begin{equation}\label{dg}
\frac{\mathrm{d}}{\mathrm{d}t} g_t(x_{k-1})=x_{k-1}+\lambda.
\end{equation}

Since the derivatives on the right-hand sides of \eqref{dY}
and \eqref{dg} are the same, and the values at $t=0$ are
also the same by \eqref{gYsame}, we have that
$\tau_1=\tilde{\tau}_1$ and that 
\begin{equation}\label{particle_height_excursion_level_agrees}
 Y^{(k)}(t)=g_t(x_{k-1}), \quad
 k=1,\dots, n, \quad t\in[0,\tau_1).
\end{equation} 
By \eqref{tildetau1def}, also $g_t(x_{k-1})>g_t(x_k)$ 
for all such $k$ and $t$, and so for all $t\in[0,\tau_1)$,
the excursions of $g_t$ are again given by 
\eqref{g0excursions}. So indeed we find that 
throughtout $[0,\tau_1)$ the
block heights and the excursion levels continue to correspond,
and the block sizes and excursion lengths do not change.

Now we look at what happens at time $\tau_1$, which
is the first time that the particle system has a merge
or deletion event. Recall from the proof of 
Proposition \ref{prop:particles} the value $\kappa$
which describes which kind of event happens at time $\tau_1$.

\begin{itemize}
\item If $1\leq \kappa\leq n-1$ then
the event is a merge of the blocks with sizes $b_{\kappa}$
and $b_{\kappa+1}$. In that case we have that 
$Y^{\kappa}(\tau_1-)=Y^{\kappa+1}(\tau_1-)$
 and so also 
(by \eqref{particle_height_excursion_level_agrees}) we have
$g_{\tau_1-}(x_{k-1})=g_{\tau_1-}(x_k)$.
Hence at time $\tau_1$ also the excursions
$[x_{\kappa-1}, x_{\kappa})$ and $[x_{\kappa}, x_{\kappa+1})$ 
merge into a single 
excursion which is $[x_{\kappa-1}, x_{\kappa+1})$.

\item If instead $\kappa=0$ then $Y^{(1)}(\tau_1-)=0$  
and $g_{\tau_1 -}(0)=0$.
Then  time $\tau_1$ is the death time of the block of size $b_1$;
also, following \eqref{eq_def_finite_shift},
at time $\tau_1$ there is a shift of size
$x^*(\tau_1 -)= x_1=b_1$.
Then we obtain $g_{\tau_1}(x)=g_{\tau_1 -}(x+b_1)$.
\end{itemize}
In both of these cases the heights and masses of the particle system at time $\tau_1$ are given by
\eqref{Y_i_after_jump} and \eqref{b_i_after_jump}. If we now define
\begin{equation}\label{eq:def:minima_x_k_after_jump}
 x'_0=0, \quad x'_1=b'_1, \quad x'_2=b'_1+b'_2, \quad \dots\quad  x'_{n-1}=b'_1+\dots+b'_{n-1}, 
 \end{equation}
then we obtain that the excursions of $g_{\tau_1}$
are the intervals
\begin{equation}\label{new-excursions}
[x'_0, x'_1), \quad [x'_1, x'_2), \quad \dots, \quad [x'_{n-2}, x'_{n-1})
\end{equation}
and the levels of these excursions are
\begin{equation}\label{gYsameagain}
g_{\tau_1}(x'_0)=Y^{(1)}(\tau_1), \quad \dots,\quad g_{\tau_1}(x'_{n-2})=Y^{(n-1)}(\tau_1).
\end{equation}
From here, proceeding from 
\eqref{new-excursions} and \eqref{gYsameagain}
just as we did from 
\eqref{g0excursions} and \eqref{gYsame},
we can repeat the argument above
from time $\tau_1$ until the time $\tau_2$ of the next jump of $\EX(g_t)$. 
Continuing recursively
until the time of the $n$th jump 
(when the last death occurs), we obtain that
the excursion lengths and block sizes continue to correspond,
as required for Proposition \ref{prop:tilt_and_shift_ic}. 
\end{proof}

\begin{remark}
Observe that written symbolically, we have shown that
\begin{equation}\label{eq_particle_tilt_shift_equiv}
\bar{g}_t(\cdot) \equiv f_{\mu_t}(\cdot), \qquad t \geq 0,
\end{equation}
where $\bar{g}_t$ is the minimum process of $g_t$ as defined at
\eqref{eq:def_barg}; 
from the expression \eqref{alt_mu_t} for $\mu_t$,
and the definition of $f_\mu$ in Definition 
\ref{def_inverse_cdf},
we have that $f_{\mu_t}$ is the 
non-increasing piecewise constant \cadlag{} function
taking value $Y^{(i)}(t)$ on an interval of length $(\bb_t)_i$
for $i=1,\dots, n(t)$, and otherwise takes the value $-\infty$. 

In this formulation the 
statement of Proposition \ref{prop:tilt_and_shift_ic} 
can be seen to follow immediately, since 
$\bb_t=\EX(f_{\mu_t})$ and $\EX(\bar{g}_t)=\EX(g_t)$.
\end{remark}

\begin{corollary}
Recalling the definition of $\nu$ from \eqref{def_eq_nu_sum_dirac} and $t_i$ from 
\eqref{def_eq_t_i_death_time}, we  have
\begin{equation}\label{def_eq_death_time_mass_measure}
\nu= \sum_{i=1}^n m_i \cdot \delta_{t_i}.
\end{equation}
\end{corollary}


\begin{remark}
\label{remark_tilt_l1}
We observed in Remark \ref{remark_interval_coalescent_l1} that the $\mathrm{ICLD}(\lambda)$ process
naturally generalizes to any initial condition $\um \in \loneord$. Similarly, the definitions and 
 results of Sections \ref{section_particle_representation} and \ref{tiltproofl1} also extend to  $\um \in \loneord$.
 As a corollary, Proposition \ref{prop:tilt_and_shift_l0} generalizes to $\um \in \loneord$.
 For initial conditions with infinite total mass we cannot naively extend Definition \ref{def_tilt_shift_step_function},
  as explained in Remark \ref{remark_no_first_excursion}.
  The extension of the tilt-and-shift representation to $\um \in \ltwoord \setminus \loneord$ will be carried out in Section \ref{section:extension_to_ltwo}.
  
\end{remark}

\section{Preparatory results about $\mu_0$ and excursions}
\label{section:prep_measure_excusrions}

In this section we collect preliminary results that we will later use in Section \ref{section:extension_to_ltwo}
when we extend the rigid representation result of Section \ref{section:rigid_finite_state_space} from $\lzeroord$
to $\ltwoord$.

In Section \ref{section_of_lemma:exponential} we state and prove some of the analytic properties
of the random point measure $\mu \sim \exmeasure(\um)$ (c.f.\ Definition \ref{def_exp_measure_law})  given some $\um \in \ltwoord$.

In Section \ref{subsection_good_functions} we state and prove results related to excursions and the
$\ORDX$ functional (c.f.\ Definitions \ref{def:excursion}, \ref{def:ORDX}).

\subsection{Some facts about random point measures}
\label{section_of_lemma:exponential}

\begin{lemma}
\label{lemma:exponential}
Let $\um \in\ltwoord$ and $\mu \sim \exmeasure(\um)$.
  With probability 1,
\begin{enumerate}[(i)]
\item \label{exponential_i_bounded}
 $\mu(A)<\infty$ for every bounded set $A \subseteq (-\infty,0]$.
\item \label{exponential_ii_sparser}
  $\mu[y,y+1]\to 0$, as $y\to -\infty$.
\end{enumerate}
\end{lemma}

\begin{proof}
For any $0\leq a<b$, 
\begin{multline*}
\E \left( \mu[-b,-a] \right) \stackrel{ \eqref{mu_exp_point_measure_def} }{=} \sum_i m_i\P(E_i\in[a,b]) 
\stackrel{ \eqref{exponentials_E_i} }{=}\\
\sum_i m_i e^{-a m_i}\left(1-e^{-(b-a)m_i}\right)
<\sum_i m_i^2 (b-a)
<\infty
\end{multline*}
since $\um \in\ltwoord$, and this is already enough to give (i).

For (ii), we have 
\begin{multline*}
\E \left( \mu[-k-1,-k] \right)
=\sum_i m_i e^{-k m_i}\left(1-e^{-m_i}\right)
\leq \\
\sum_i m_i^2 e^{-k m_i}
\to 0 \quad \text{ as }\; k\to\infty,
\end{multline*}
and also 
\begin{multline*}
\Var \left( \mu[-k-1,-k] \right) =
\sum_i  m_i^2 \Var \left( \ind [ k \leq E_i \leq k+1] \right) \leq \\
\sum_i  m_i^2 \mathbb{P} \left(  k \leq E_i \leq k+1 \right).
\end{multline*}
Thus $\sum_k \Var \left( \mu[-k,-k-1] \right) \leq \sum_i m_i^2<\infty$.

Then let $k$ be large enough such that
$\E \left(\mu[-k-1,-k]\right)  \leq \delta/2$.
Then by Chebyshev's inequality,
\begin{equation*}
\P( \mu[-k,-k-1]>\delta)
\leq \frac{\Var \left( \mu[-k,-k-1] \right) }{\big(\delta- \delta/2 \big)^2}.
\end{equation*}
Hence $\sum_k \P( \mu[-k,-k-1]>\delta) <\infty$ and 
the result in (ii) follows from Borel-Cantelli.
\end{proof}

\begin{lemma}\label{lemma:dense}
If $\um \in \ltwoord \setminus \loneord$ and $\mu \sim \exmeasure(\um)$, 
then with probability 1 we have $\mu[-b,-a]>0$ for any $0<a<b$.
\end{lemma}
\begin{proof}
It is enough to prove that for all pairs of rational numbers $0<a<b$ we have $\mu[-b,-a]>0$ with probability
1. This follows from the second Borel-Cantelli lemma and the fact that
\[ \sum_{i=1}^{\infty} \mathbb{P}(\, a \leq E_i \leq b \,)
\stackrel{ \eqref{exponentials_E_i} }{=}
 \sum_{i=1}^{\infty} e^{-a m_i}\left(1-e^{-(b-a)m_i}\right) \stackrel{(*)}{=} +\infty \quad \text{if} \quad
 \um \in \ltwoord \setminus \loneord,
 \]
 where $(*)$ follows from  $e^{-a m}\left(1-e^{-(b-a)m}\right) \approx (b-a)m$ as $m \to 0$.
\end{proof}

\subsection{Good functions}
\label{subsection_good_functions}

We define a set $\cG$ of ``good" functions. 
Recall the notion of excursions from Definition \ref{def:excursion}.

\begin{definition}\label{def:good_functions}
If $g$ is a function from $[0,\infty)$ to $\R\cup\{-\infty\}$,
we say $g\in\cG$ if:
\begin{enumerate}[(i)]
\item\label{good_i_cadlag} $g$ is lower semi-continuous and \cadlag.
\item\label{good_ii_minusinf} If $g(x)=-\infty$ then $g(x')=-\infty$ for all $x'>x$.
\item\label{good_iii_arrangable} For any $\varepsilon>0$, 
$g$ has only finitely many excursions above its minimum with 
length greater than or equal to $\varepsilon$.
\item\label{good_iv_lebesgue} Let $x_{\max}=\sup\{ \, x:x>-\infty \, \}\leq \infty$. 
The Lebesgue measure of the set of points in $(0,x_{\max})$ 
which are not contained in some excursion above the minimum is 0. 
\end{enumerate}
\end{definition}

  If $g\in\cG$, then $\ORDX(g)$  (see Definition \ref{def:ORDX}) is well-defined.

\begin{lemma}\label{lemma:ftgood}
Suppose $f_0$ is defined from $\um \in\ltwoord$ by Definition \ref{def:f_zero_from_exponentials}.
Define $f_t$ by \eqref{def_eq_f_t_from_f_0}. Then with probability 1, we have $f_t\in\cG$ for all $t\geq 0$.
\end{lemma}
\begin{proof}
Properties \eqref{good_i_cadlag}, \eqref{good_ii_minusinf} 
in Definition \ref{def:good_functions}
can be deduced for $f_t$ directly from the definitions \eqref{f0def_b}  and 
\eqref{def_eq_f_t_from_f_0}. 
Property \eqref{good_iv_lebesgue} for $f_t$ follows from the fact that
 \eqref{good_iv_lebesgue} holds for $f_0$ (see Remark \ref{remark:f_0}\eqref{remark_f0_i_lebesgue})
 and the observation that every excursion of $f_0$ is contained in an excursion of $f_t$.

It remains to justify property \eqref{good_iii_arrangable}.
 The function $f_0$ is
non-increasing, and Lemma \ref{lemma:exponential}\eqref{exponential_ii_sparser}
tells us that the length of the interval on which $f_0$ takes
values in $[y,y+1]$ tends to 0 as $y\to -\infty$. 
Hence for every $\varepsilon$ there exists a 
$K_\varepsilon$ such that 
$f_0(x)-f_0(x-\varepsilon)<-1$ for all $x \geq K_\varepsilon$.
As a result, $f_t(x)-f_t(x-\varepsilon)<-1+t\varepsilon$. 
If $\varepsilon<1/t$, we find that $f_t(x)<f_t(x-\varepsilon)$,
so all excursions intersecting $(K_\varepsilon,\infty)$ must have length 
less than $\varepsilon$, as desired.

\end{proof}

\begin{lemma}\label{lemma:cadlag_tilt}
Given some $f_0 \in \cG$ define $f_t$ by \eqref{def_eq_f_t_from_f_0} and assume that $f_t\in\cG$ for all $t\geq 0$.
The function $\ORDX(f_t)$ from $[0,\infty)$ to $\ell_\infty^{\downarrow}$
is \cadlag.
\end{lemma}

\begin{proof}
Let us write $\um(t)=(m_1(t), m_2(t), \dots)=\ORDX(f_t)$. 
Since we use the topology of coordinatewise convergence on $\ell_\infty^{\downarrow}$,
it is enough to show that the function $ t \mapsto m_i(t)$ is \cadlag{} for all $i$. 

Consider $0\leq t'<t$. Since $f_t$ is obtained from $f_{t'}$ by 
adding on an increasing function, any minimum of $f_t$ is also 
a minimum of $f_{t'}$, and any excursion of $f_{t'}$ 
is a sub-interval of an excursion of $f_t$. 

\smallskip

Fix $t$ and suppose $[l,r)$ is an excursion of $f_t$. 
Take $\varepsilon$ with $0<\varepsilon<2l$.
Recalling the notion of $\bar{f}$ from \eqref{eq:def_barg}, we have
 $\bar{f}_t(l-\varepsilon/2)>f_t(l)$; 
hence if $\delta$ is sufficiently small, then 
\[\bar{f}_{t+\delta}(l-\varepsilon/2)>f_t(l)+\delta l=f_{t+\delta}(l),\]
and so $f_{t+\delta}$ has a minimum in $[l-\varepsilon/2,l]$. 

Also, there is some $x\in (r,r+\varepsilon/2)$ with $f_t(x)<f_t(l)$. 
Hence if $\delta$ is sufficiently small, then $f_{t+\delta}$ has a minimum 
in $[r,r+\varepsilon/2]$. 

So for any $\varepsilon$, we can find $\delta$
such that the length of the excursion of $f_{t+\delta}$ 
which includes $(l,r)$ is at most $r-l+\varepsilon$. 

\smallskip

Now we will argue that for any $\varepsilon>0$ there exists small enough  $\delta$ such that the 
length $m_1(t+\delta)$ is at most $m_1(t)+\varepsilon$. 

Fix any $T>t$ and consider $\delta\in(0, T-t)$. Since the excursions of $f_{t+\delta}$ are contained in the excursions of $f_T$, any excursion of $f_{t+\delta}$ of length more than
 $m_1(t)+\varepsilon$ must be contained in an excursion of $f_T$ whose length also exceeds that. 
There are only finitely many such excursions of $f_T$ since $f_T \in \cG$. Let $\cU$ be 
the union of those excursions, which has finite total length, say $L$.

Now let us look at all the excursions of $f_t$ contained in $\cU$. 
There are at most countably many. We can take a finite number of them
whose total length is at least $L-\varepsilon$. 
Each of them has length no more than $m_1(t)$. From the property above,
if we choose $\delta$ small enough, then at time $t+\delta$,
none of them is contained in an excursion of length more than 
$m_1(t)+\varepsilon$. But also, since the remaining length of $\cU$ 
outside this set is only $\varepsilon$, then also no other 
point in $\cU$ is contained in an excursion of length more than 
$m_1(t)+\varepsilon$.

It follows that $m_1(t+\delta)\leq \varepsilon+ m_1(t)$ as desired.

In similar fashion we can also obtain 
that $\sum_{i=1}^k m_i(t+\delta)\leq \varepsilon+\sum_{i=1}^k m_i(t)$
for sufficiently small $\delta$, for any $k$. 
But note that $\sum_{i=1}^k m_i(t)$ is non-decreasing in $t$. 
So for each $k$, $\sum_{i=1}^k m_i(t)$ is right-continuous with left limits,
and hence the same is true for $m_i(t)$ for each $i$.
\end{proof}

\begin{corollary}\label{corollary_countably_many_double_arches}
Given some $f_0 \in \cG$ define $f_t$ by \eqref{def_eq_f_t_from_f_0} and assume that $f_t\in\cG$ for all $t\geq 0$.
Then the set of times $t$ such that $f_t$ has a non-strict excursion (c.f.\ Definition \ref{def:excursion}) is
countable.
\end{corollary}
\begin{proof}
By Lemma \ref{lemma:cadlag_tilt} the function $t \mapsto \ORDX(f_t)$ 
is \cadlag, therefore it has countably many jumps, c.f.\ \cite[Section 12, Lemma 1]{billingsley}.
Thus we only need to show that if $f_{t_0}$ has a non-strict excursion for some $t_0>0$ then $t \mapsto \ORDX(f_t)$
jumps at $t_0$. If $[l,r)$ is a non-strict excursion of $f_{t_0}$, then
$f(l)=f(x)$ for some $x \in (l,r)$. Now if $t<t_0$ and $[l',r')$ is an excursion of $f_t$, then $x$ is not in the interior of 
$[l',r')$. This implies that for any $t<t_0$ the function $f_t$  has at least two disjoint excursions contained in $[l,r)$ that are separated by $x$, and these excursions
 merge at time $t_0$, thus if $k_0$ is the smallest index $k$ for which $m_k(t_0)<r-l$ then the
non-decreasing function $t \mapsto \sum_{i=1}^{k_0-1} m_i(t)$ jumps at time $t_0$. 
\end{proof}

\begin{definition}\label{def_uniformly_good}
A family of good functions $f^{(i)} \in \cG, i \in \mathcal{I}$ is said to be \emph{uniformly good} if
 for any $\varepsilon$ there exists $K_{\varepsilon} \in \R$ such that for any  $i \in \mathcal{I}$
the excursions of $f^{(i)}$ intersecting $[K_{\varepsilon},\infty)$ are all shorter than $\varepsilon$.
\end{definition}

\begin{lemma}\label{lemma_unif_good_conv_ordx}
Let  $f \in \cG$ be continuous and assume that all of the excursions of
$f$ are strict (c.f.\ Definition \ref{def:excursion}). Let $f^{(n)} \in \cG$, $n \in \N$ be a sequence of (not necessarily continuous) functions that converge to $f$ uniformly on bounded intervals. Let us also assume that the family consisting of $f$ and $f^{(n)}, n \in \N$ is uniformly good.
Then $\ORDX(f^{(n)}) \to \ORDX(f)$ as $n \to \infty$ in the product topology on $\linford$.
\end{lemma}

\begin{proof}

Suppose $(l,r)$ is an excursion of $f$. 
For any given $\gamma>0$ (with $\gamma<l$),
there is a $\delta>0$ such that the following properties hold:
\begin{itemize}
\item[(i)]
$f(x)\geq f(l)+\delta$ for all $x\in[0,l-\gamma]$;
\item[(ii)]
$f(x)\geq f(l)+\delta$ for all $x\in [l+\gamma, r-\gamma]$;
\item[(iii)]
$f(x)\leq f(l)-\delta$ for some $x\in[r,r+\gamma]$.
\end{itemize}
Here (i) holds since $f$ has a minimum at $l$ and,
being continuous, must achieve its bounds on $[0,l-\gamma]$;
(ii) holds since the excursion is strict, and
(iii) holds since by  the definition of excursions,
there must be points arbitrarily close to the right of $r$
which take value lower than $f(l)$.

Now suppose $n$ is large enough that
$|f^{(n)}(x)-f(x)|<\delta/2$ for all $x\in[0,r+\gamma]$. 
Then we obtain the following properties:
\begin{itemize}
\item[(i)]
$f^{(n)}(x)\geq f(l)+\delta/2$ for all $x\in[0,l-\gamma]$;
\item[(ii)]
$f^{(n)}(x)\geq f(l)+\delta/2$ for all $x\in [l+\gamma, r-\gamma]$;
\item[(iii)]
$g^{(n)}(x)\leq f(l)-\delta/2$ for some $x\in[r,r+\gamma]$;
\item[(iv)]
$g^{(n)}(l)\in(f(l)-\delta/2, f(l)+\delta/2)$.
\end{itemize}
Then $f^{(n)}$ must have an excursion which starts somewhere
in $[l-\gamma, l+\gamma]$ and ends somewhere in 
$[r-\gamma, r+\gamma]$.  

Now let $\varepsilon>0$ and choose $K_\varepsilon$ such that
 the excursions of $f$ and $f^{(n)}, n \in \N$  intersecting $[K_{\varepsilon},\infty)$ are all shorter than $\varepsilon$.
 
Now by Definition \ref{def:good_functions}\eqref{good_iv_lebesgue}
there exists a finite collection
of excursions $(l_j, r_j)$ of $f$, whose union
covers all of $[0,K_\varepsilon+\varepsilon]$ except for a set of total length
less than $\varepsilon/2$. Let $k$ be the total number of these
excursions. 
Apply the above argument to all of these excursions
with $\gamma=\varepsilon/4k$.
Then if $n$ is sufficiently large, we have that
for each of these excursions of $f$,
there is a corresponding excursion of $f^{(n)}$ 
whose length is within $\varepsilon/2k$; 
the remaining length in $[0,K_\varepsilon+\varepsilon]$
amounts to no more than $\varepsilon$; and
we know that outside $[0,K_\varepsilon+\varepsilon]$, 
all excursions (either of $f^{(n)}$ or $f$)
have length less than or equal to $\varepsilon$.

It follows that for any $i>0$, the $i$th largest 
excursion of $f^{(n)}$ and the $i$th largest excursion
of $f$ differ by at most $\varepsilon$. 
Hence indeed $\ORDX(f^{(n)})$ converges componentwise 
to $\ORDX(f)$, as desired.

\end{proof}

\section{Extension of rigid representation to $\ltwoord$}
\label{section:extension_to_ltwo}

In this section we will extend the rigid representation 
results of Section \ref{section:rigid_finite_state_space} to
 any initial condition $\um \in \ltwoord$. As we have discussed in Remark \ref{remark_tilt_l1}, extension
 from $\lzeroord$ to $\loneord$ is automatic, so in this section we will assume that
 $\um \in \ltwoord \setminus \loneord$.

In Section \ref{subsection_extension_mc_tilt_to_l2} we prove that the MC admits a tilt representation, i.e., Theorem \ref{thm:tilt}.

In Section \ref{subsection:mcld_extension_to_ltwo_proof} we prove that the MCLD admits a tilt-and-shift representation,
i.e., Theorem \ref{thm:mcld_extension_introduction}.

\subsection{Extension of MC tilt representation to $\ltwoord$}
\label{subsection_extension_mc_tilt_to_l2}

The aim of this section is to prove Theorem \ref{thm:tilt}.

\medskip

 $f_0$ is defined from $\um \in\ltwoord$ by Definition \ref{def:f_zero_from_exponentials}.
 $f_t$ is defined by \eqref{def_eq_f_t_from_f_0}.
\begin{definition}\label{def_good_times}
By Lemma \ref{lemma:ftgood} and Corollary \ref{corollary_countably_many_double_arches} 
the (random) set of times $t$ such that $f_t$ has a non-strict excursion (c.f.\ Definition \ref{def:excursion}) is
countable.
Hence for all but countably many $t$, 
the probability that all excursions of $f_t$ are strict is equal to $1$.
 Let $\cT$ denote this (deterministic)  set of ``good" times $t$.
\end{definition}

From Lemma \ref{lemma:ftgood} and
Lemma \ref{lemma:cadlag_tilt} it follows that
$t \mapsto \ORDX(f_t)$ is a \cadlag{} process with respect to the product topology on $\linford$.
The graphical representation  of the multiplicative coalescent $\bm_t$ constructed in
 \cite[Section 1.5]{aldous_mc}
  is also a \cadlag{} process with respect to the
topology of the $\dist(\cdot,\cdot)$-metric defined in \eqref{eq_def_d_metric} (see \cite[Lemma 2.8]{james_balazs_mcld_feller}), thus it is
\cadlag{} with respect to the weaker product topology on $\linford$.

Hence, since $\cT$ is dense, 
if we can show that for any finite 
collection $t_1, \dots, t_r\in \cT$, 
we have 
\begin{equation}\label{tiltconclusion}
(\ORDX(f_{t_i}), 1\leq i\leq r)
\isd
(\bm_{t_i}, 1\leq i\leq r),
\end{equation}
then indeed the law of $\ORDX(f_t)$ is that of the MC.

For each $n$, let $\um^{(n)}$ be given by 
\begin{equation}
 \label{truncation_of_initial_state}
m^{(n)}_i = \begin{cases} m_i,&i\leq n\\ 0,&i>n\end{cases}.
\end{equation}
For each $n$, 
$\um^{(n)}\in\lzeroord$, and $\um^{(n)}\to\um$
in $\ltwoord$ as $n\to\infty$.

We couple the processes starting from $\um^{(n)}$, $n\geq 1$,
by using the same height variables $Y_i=-E_i$ throughout. 
If we define \[f^{(n)}_t(x)=f_0^{(n)}(x)+tx,\] 
then by the $\lambda =0$ case of Proposition \ref{prop:tilt_and_shift_l0} we have
 \begin{equation}\label{finite_dim_agree_for_truncated_mc}
 (\ORDX(f_{t_i}^{(n)}), 1\leq i\leq r)\isd (\bm^{(n)}_{t_i}, 1\leq i\leq r),
 \end{equation}
where $\bm^{(n)}_t, t \geq 0$ is the MC started from $\um^{(n)}$.

By the Feller property of the MC (see \cite[Proposition 5]{aldous_mc})  we have
\begin{equation}\label{eq_truncated_conv_feller_mc}
(\bm^{(n)}_{t_i}, 1\leq i\leq r)\tod (\bm_{t_i}, 1\leq i\leq r), \qquad n \to \infty 
\end{equation}
(with respect to the topology of $\ltwoord$ and hence also coordinatewise).

We will let $n\to\infty$, and show that  
\begin{equation}\label{mc_convergence_ORDX_ft}
\begin{array}{c}
\ORDX(f^{(n)}_t) \to \ORDX(f_t) \quad \text{for  all} \quad  t\in\cT \\
\text{coordinate-wise with probability 1.}
\end{array}
 \end{equation}

Putting together \eqref{finite_dim_agree_for_truncated_mc},\eqref{mc_convergence_ORDX_ft} and
\eqref{eq_truncated_conv_feller_mc}
we obtain \eqref{tiltconclusion} as required.

\medskip

It remains to show \eqref{mc_convergence_ORDX_ft}. We will achieve this by checking that
the conditions of Lemma \ref{lemma_unif_good_conv_ordx} almost surely hold if $t \in \cT$.
We may assume that $\um \in \ltwoord \setminus \loneord$, 
as discussed in the first paragraph of Section \ref{section:extension_to_ltwo}.

\begin{lemma}\label{lemma:tailsmall}
Fix $t>0$. With probability 1, the family of functions that consists of
$f_t^{(n)}$, $n\geq 1$ and $f_t$ is uniformly good (c.f.\ Definition \ref{def_uniformly_good}).
\end{lemma}

\begin{proof}
Recalling Definitions \ref{def_inverse_cdf} and \ref{def:f_zero_from_exponentials} we see that
 \[ f_0^{(n)}=f_{\mu_0^{(n)}}, \qquad \text{where} \qquad
\mu_0^{(n)}=\sum_{i=1}^{n} m_i \cdot \delta_{Y_i}.\] 
Now
$ \mu_0-\mu_0^{(n)}=\sum_{n<i} m_i \cdot \delta_{Y_i}$ is a non-negative measure for each $n \in \N$, 
thus $\mu_0$ dominates $\mu_0^{(n)}$ for each $n$ and we obtain the proof of Lemma \ref{lemma:tailsmall}
 by repeating the argument of proof Lemma \ref{lemma:ftgood},
 uniformly in $n$.
\end{proof}

\begin{lemma}\label{lemma:ucbi} 
If $\um \in \ltwoord \setminus \loneord$ then for any $t \geq 0$ the
 function $f_t(\cdot)$ is continuous and $f_t^{(n)}\to f_t$ 
uniformly on bounded intervals. 
\end{lemma}

\begin{proof}
Since $f_t^{(n)}(x)=f_0^{(n)}(x)+tx$,
and $f_t(x)=f_0(x)+tx$, it is enough to show the statements of the lemma for $t=0$. 

The function $f_0$ is non-increasing, moreover by Lemma \ref{lemma:dense}
the values $E_i, i \geq 0$ are dense in $[0,\infty)$, thus
 by Remark \ref{remark:f_0}\eqref{remark_f0_i_lebesgue} the values
$f_0$ takes are dense in $(-\infty,0)$. 
Hence $f_0$ is continuous.

Since $f_0$ is a continuous function on $[0,\infty)$, it is uniformly continuous on any bounded sub-interval.

Fix any $U<\infty$. Let us define $n_0= \min\{ \, n \, : \, \sum_{i=1}^n m_i >U \, \}$.
For all $n \geq n_0$ we have $f^{(n)}_0(U)\geq f^{(n_0)}_0(U)=: -S$.

Consider $x \leq U$, $n \geq n_0$. 
By \eqref{f0def_for_lzeroord} have $f^{(n)}_0(x)=Y_k$, where $k$ is such that 
\[
x\in\left[
\sum_{j\leq n\, :\, E_j<E_k} m_j\,,
\,\,
m_k + \sum_{j\leq n \, : \, E_j<E_k} m_j\right).
\]
By considering the interval on which 
the function $f_0$ takes the same value $Y_k$, we have 
\[ f_0^{(n)}(x)=f_0(x+\delta) \qquad \text{where} \qquad
\delta=\sum_{j>n \, : \, E_j<E_k} m_j.
\]
For all $n \geq n_0$ we have $\delta \leq \sum_{j>n\, :\, E_j<S} m_j$.
This goes to 0 as $n\to\infty$ by Lemma \ref{lemma:exponential} and dominated convergence.

Then by the uniform continuity of $f_0$ on bounded intervals,
we have that $f_0^{(n)}\to f_0$ uniformly on $[0,U]$, as desired.
\end{proof}

Finally, we can insert the properties derived in
 Definition \ref{def_good_times} and  Lemmas \ref{lemma:tailsmall}, \ref{lemma:ucbi} into Lemma \ref{lemma_unif_good_conv_ordx}
 to obtain \eqref{mc_convergence_ORDX_ft}.
This completes the proof of  Theorem \ref{thm:tilt}.

\subsection{Extension of MCLD tilt-and-shift representation to $\ltwoord$}
\label{subsection:mcld_extension_to_ltwo_proof}
\label{section_approx_by_truncation}

The aim of this section is to prove Theorem \ref{thm:mcld_extension_introduction}.

\medskip

In Section \ref{subsection_perturbation} we give quantitative bounds on the effect of the insertion of a new 
particle in a finite particle system on the death times of other particles. 

In Section \ref{section_truncation_and_approximation} we construct the death times of each particle in the infinite particle
system associated to the initial state $\um \in \ltwoord$ by inserting particles one by one and showing  that death times converge.

In Section \ref{section_tilt_shift_cont_lemmas} we collect some technical lemmas that allow us to deduce
the convergence of the ordered sequence of excursion lengths from
the approximation results of Section \ref{section_truncation_and_approximation}.

In Section \ref{section_proof_of_tilt_and_shift_mcld_extension} we extend the MCLD tilt-and-shift representation from
$\loneord$
to $\ltwoord$, i.e., we prove Theorem \ref{thm:mcld_extension_introduction}.

\subsubsection{Perturbation of the particle system}
\label{subsection_perturbation}

Recall the particle system introduced in Section \ref{section_particle_representation}.
The main result of Section \ref{subsection_perturbation} is Lemma \ref{lemma_insertion_of_new_particle_death_times}, which
  quantifies the effect of the insertion of a new particle on the death times of other particles.
This perturbation result will play an important
role when we extend our rigid representation 
of MCLD($\lambda$) from $\lzeroord$ to $\ltwoord$
in Section \ref{section_truncation_and_approximation} using truncation and approximation. Note that Lemma \ref{lemma_insertion_of_new_particle_death_times}
is a \emph{deterministic} result, i.e., it holds for any initial configuration of particles.

\begin{lemma}\label{lemma_insertion_of_new_particle_death_times}
For any $\um=(m_1,\dots, m_n) \in \lzeroord$ and  $Y_1(0),\dots, Y_n(0)$, let us define $Y_1(t),\dots, Y_n(t)$,
 $\mu_t$ and $t_1,\dots,t_n$  as in Definition \ref{def_particle_rep}. 
 
Let us initialize a new particle system $ \widetilde{Y}_0(t),\dots, \widetilde{Y}_{n}(t)$
by letting $\widetilde{Y}_i(0)=Y_i(0)$ for any $i=1,\dots,n$ and by adding a new particle
with initial height $\widetilde{Y}_0(0)$ and mass $\widetilde{m}$. 

Let us then define  $\widetilde{Y}_0(t),\dots, \widetilde{Y}_{n}(t)$,
 $\widetilde{\mu}_t$ and $\widetilde{t}_0,\dots,\widetilde{t}_{n}$ 
  analogously to Definition \ref{def_particle_rep}. Then  we have
  \begin{equation}\label{eq_insertion_of_new_particle_death_times}
  |\widetilde{t}_i-t_i| \leq \ind  [ \widetilde{Y}_{0}(0) > Y_i(0)]
  \frac{\widetilde{m} |Y_i(0)| }{ \lambda^2 } \exp \left( \frac{ \mu_0( Y_i(0),0) }{\lambda} \right), \quad
  i=1,\dots, n.
  \end{equation}
\end{lemma}

The rest of Section \ref{subsection_perturbation} is devoted to the proof of Lemma \ref{lemma_insertion_of_new_particle_death_times}.

\medskip

Without loss of generality we may assume
$ Y_1(0)$, $Y_2(0)$, \dots $Y_n(0)$ are all distinct,
 because if $Y_i(0)=Y_j(0)$ for some $i \neq j$ then
we can replace our particle system by another one with fewer particles in which
these two particles are merged.

With a slight abuse of notation, for the rest of Section \ref{subsection_perturbation} we will
relabel our particles so that we have $Y_1(0) > Y_2(0) > \dots > Y_n(0)$ and
denote by $m_i$ the weight and $t_i$ the death time of the particle with initial location $Y_i(0)$. 
Thus we have \[t_1 \leq t_2 \leq \dots \leq t_n.\]

The next lemma gives a recursive formula for $t_1,\dots,t_n$.
\begin{lemma} Let us formally define $t_0=0$. We have
\begin{equation}\label{recursive_formula_for_death_times}
t_i= t_{i-1} \vee \frac{ |Y_i(0)| - \sum_{j=1}^{i-1} m_j t_j  }{\lambda}, \qquad 1 \leq i \leq n.
\end{equation}
\end{lemma}

\begin{proof} First observe that we have
\begin{multline}\label{formula_relating_Y_i_0_and_death_times}
|Y_i(0)| \stackrel{ \eqref{def_eq_t_i_death_time} }{=} Y_i(t_i)-Y_i(0)=
\int_0^{t_i} \dot{Y}_i(t)\, \mathrm{d}t \stackrel{ \eqref{particle_dynamics_lambda} }{=}
\int_0^{t_i} \left( \lambda+ \mu_t( Y_i(t),0) \right)\, \mathrm{d} t 
\stackrel{ \eqref{def_eq_mu_t} }{=}\\
\int_0^{t_i} \left( \lambda+ \sum_{j=1}^{i-1} m_j \ind [\, Y_i(t)<Y_j(t)<0 \, ]  \right) \, \mathrm{d} t.
\end{multline}
We will prove \eqref{recursive_formula_for_death_times} by considering two cases separately.

\smallskip

{\bf First case:}
If $t_i>t_{i-1}$ then 
\begin{equation}\label{t_i_greater_than_t_i_minus_one_case}
|Y_i(0)| \stackrel{(*)}{=}\\
\int_0^{t_i} \left( \lambda+ \sum_{j=1}^{i-1} m_j \ind [ t < t_j ]  \right) \, \mathrm{d} t=
\lambda t_i+ \sum_{j=1}^{i-1} m_j t_j,
\end{equation}
where in $(*)$ we used \eqref{formula_relating_Y_i_0_and_death_times} and the fact that $t_i>t_{i-1}$ implies $Y_j(t)>Y_i(t)$ for any $t <t_i$ and $j \leq i-1$. Rearranging 
\eqref{t_i_greater_than_t_i_minus_one_case} we obtain that if $t_i>t_{i-1}$ then we have $t_i=t^*_i$, where
\[ t^*_i:=\frac{ |Y_i(0)| - \sum_{j=1}^{i-1} m_j t_j  }{\lambda}, \]
therefore \eqref{recursive_formula_for_death_times} holds.

\smallskip

{\bf Second case:}
Now we assume $t_i=t_{i-1}$. First note that  
\begin{equation}\label{t_i_equal_to_t_i_minus_one_case}
 |Y_i(0)| \stackrel{\eqref{formula_relating_Y_i_0_and_death_times}}{\leq} 
 \int_0^{t_i} \left( \lambda+ \sum_{j=1}^{i-1} m_j \ind [ t < t_i ]  \right) \, \mathrm{d} t=
\lambda t_i+ \sum_{j=1}^{i-1} m_j t_j.
\end{equation}
Rearranging \eqref{t_i_equal_to_t_i_minus_one_case} we obtain $t_i^* \leq t_i$. Now if $t_i=t_{i-1}$, then
$t_i^* \leq t_{i-1}$, and therefore \eqref{recursive_formula_for_death_times} holds.
\end{proof}

Recall the notion introduced in the statement of Lemma \ref{lemma_insertion_of_new_particle_death_times}.
   Denote by
\begin{equation}\label{def_eq_i_star}
i^*= \inf \{ \, i \; : \; Y_{i}(0) < \widetilde{Y}_0(0)\, \}.
 \end{equation}
In particular, we define $i^*=\infty$ if $Y_{i}(0) \geq \widetilde{Y}_0(0) $ for all $i \in [1,n]$.

  By \eqref{particle_dynamics_lambda}
the speed of $Y_{i}(t)$ only depends on the locations of particles strictly above it,
 so we have $\widetilde{Y}_i(t) \equiv Y_i(t)$ for any $1 \leq i <i^*$, thus we have
 $t_i = \widetilde{t}_i$   for any $1 \leq i <i^*$, and
 \eqref{eq_insertion_of_new_particle_death_times} trivially
 follows for these particles.

\smallskip
For any $i^* \leq i \leq n$ we will think about $\widetilde{t}_i=\widetilde{t}_i(\widetilde{m})$ as a function
of the variable $\widetilde{m} \geq 0$, which represents the weight of the inserted particle.
 With this definition we have $\widetilde{t}_i(0)=t_i$.
Let us define the Lipschitz constant $L_i$ by 
\[ L_i= \sup_{0 \leq \widetilde{m} < \widetilde{m}'}
 \frac{ |\widetilde{t}_i(\widetilde{m}')-\widetilde{t}_i(\widetilde{m})| }{\widetilde{m}'-\widetilde{m}}.
\]
In order to prove \eqref{eq_insertion_of_new_particle_death_times}, we only need to show 
\begin{equation}\label{lipschitz_explicit_final_bound}
L_i \leq \frac{ |Y_i(0)| }{ \lambda^2 } \exp \left( \frac{ \mu_0( Y_i(0),0) }{\lambda} \right), \quad
  i=i^*,\dots, n.
\end{equation}

We begin  with the following claim. The proof is trivial and we omit it.

\begin{claim}\label{claim_lipschitz_max}


If $F$ and $G$ are both Lipschitz-continuous functions of 
the variable $\widetilde{m}$ with Lipschitz-constants $L_F$ and and $L_G$, and 
$H(\widetilde{m}):=F(\widetilde{m}) \vee G(\widetilde{m})$, then the Lipschitz-constant
$L_H$ satisfies $L_H \leq L_G \vee L_H$.



\end{claim}


Denote by $\widetilde{t}_0(\widetilde{m})$ the death time of particle $\widetilde{Y}_0(t)$ in the particle 
system $\widetilde{Y}_0(t), \widetilde{Y}_1(t), \dots, \widetilde{Y}_n(t)$. Note that we have
 $\widetilde{t}_0(\widetilde{m})=\widetilde{t}_0(0)$, i.e., the death time of the inserted particle
 does not depend on its weight, so we will omit dependence of $\widetilde{t}_0$ on $\widetilde{m}$.

\begin{lemma}\label{lemma_exact_infinitesimal_change}
For any  $i^* \leq i \leq n$ we have
\begin{equation}
\label{L_i_recursive_ineq}
L_i  \leq 
 \frac{1}{\lambda} \left( \widetilde{t}_0 + \sum_{j=i^*}^{i-1} m_j \cdot L_j \right).
\end{equation}
\end{lemma}
\begin{proof}
For $i^* \leq i \leq n$, the recursion \eqref{recursive_formula_for_death_times} for the new particle system reads
\begin{equation}\label{recursive_death_with_inserted_particle}
\widetilde{t}_i(\widetilde{m})
=
\widetilde{t}_{i-1}(\widetilde{m}) 
\vee
\frac{ |Y_i(0)| - \sum_{j=1}^{i^*-1} m_j t_j - \widetilde{m} \widetilde{t}_0 -  \sum_{j=i^*}^{i-1} m_j \widetilde{t}_j(\widetilde{m}) }{\lambda}.
\end{equation}

We will prove \eqref{L_i_recursive_ineq} by induction on $i$.
We begin with $i=i^*$. Since $\widetilde{t}_{i^*-1}(\widetilde{m})= t_{i^*-1}$ and thus $L_{i^*-1} =0$, 
we obtain from Claim \ref{claim_lipschitz_max} and \eqref{recursive_death_with_inserted_particle} that
$L_{i^*} \leq \widetilde{t}_0 / \lambda$, thus \eqref{L_i_recursive_ineq} holds if $i=i^*$.

As for the induction step, we obtain from Claim \ref{claim_lipschitz_max} and \eqref{recursive_death_with_inserted_particle} that
\begin{equation}\label{rec_lip}
L_i \leq L_{i-1} \vee \frac{1}{\lambda} \left( \widetilde{t}_0 + \sum_{j=i^*}^{i-1} m_j \cdot L_j \right).
\end{equation}
Now \eqref{L_i_recursive_ineq} holds for $i$ by \eqref{rec_lip} and the induction hypothesis (i.e., the assumption that
\eqref{L_i_recursive_ineq} holds for $i-1$).

\end{proof}

From the recursive inequalities \eqref{L_i_recursive_ineq} one readily deduces by induction on $i$ the following
explicit bound: 
\begin{equation}\label{L_i_explicit_bound}
L_i \leq  
 \frac{ \widetilde{t}_0}{\lambda} \cdot \prod_{j=i^*}^{i-1} \left( 1+ \frac{m_j}{\lambda} \right),
\qquad
i^* \leq i \leq n.
\end{equation}

Now we are ready to prove \eqref{lipschitz_explicit_final_bound}
 for any $i^* \leq i \leq n$:
\begin{multline*}
L_i
\stackrel{  \eqref{L_i_explicit_bound} }{\leq} 
\frac{ t_i}{\lambda} \cdot \prod_{j=1}^{i-1} \left( 1+ \frac{m_j}{\lambda} \right)
\leq
\frac{ t_i}{\lambda} \cdot 
\exp\left( \sum_{j=1}^{i-1}  \frac{m_j}{\lambda} \right)
\stackrel{ \eqref{particle_dynamics_lambda}, \eqref{def_eq_t_i_death_time} }{\leq} \\
 \frac{ |Y_i(0)| }{ \lambda^2 } \exp \left(  \sum_{j=1}^{i-1}  \frac{m_j}{\lambda} \right)=
   \frac{ |Y_i(0)| }{ \lambda^2 } \exp \left( \frac{ \mu_0( Y_i(0),0) }{\lambda} \right).
\end{multline*}
The proof of \eqref{lipschitz_explicit_final_bound} and Lemma \ref{lemma_insertion_of_new_particle_death_times} is complete.

\subsubsection{Truncation and approximation}
\label{section_truncation_and_approximation}

The main result of Section \ref{section_truncation_and_approximation} is Lemma \ref{lemma_vague_conv}
which extends the ``shift'' operator from $\lzeroord$ to $\ltwoord$
using truncations and approximation, c.f.\ Remarks \ref{remark_no_first_excursion} and 
 \ref{remark_how_we_extend_mcld_to_l2}\eqref{remark_iii_how_we_truncate}.

\medskip

Given some $\um \in \ltwoord$, let us generate $E_1,E_2,\dots$ as in \eqref{exponentials_E_i}.
Define the
 truncation $\um^{(n)}$ by \eqref{truncation_of_initial_state}.
 We define
  $g^{(n)}_0(\cdot) \equiv f^{(n)}_0(\cdot)$ using  $E_1,E_2,\dots, E_n$ by \eqref{f0def_b}. 
  Note that we still denote by $f_t(\cdot)$ the function
  constructed from the un-truncated $\um$ by \eqref{f0def_b} and \eqref{def_eq_f_t_from_f_0}, and also 
 that we use the same sequence of  random variables $E_1,E_2,\dots$ to obtain a coupling of
 $g^{(n)}_0(\cdot), n \in \N$.
    We define $g^{(n)}_t(\cdot)$ from $g^{(n)}_0(\cdot)$ using
   Definition \ref{def_tilt_shift_step_function}.
This gives rise to the measure $\nu^{(n)}$ by \eqref{def_eq_nu_sum_dirac} and the function $\Phi^{(n)}$ by
\eqref{def_eq_Phi_from_nu}. By \eqref{g_t_from_g_0} we have
\begin{equation}\label{g_t_from_g_0_n}
 g_t^{(n)}(x)=g_0^{(n)}(x+\Phi^{(n)}(t))+\lambda t + 
 \int_0^t \left(x + \Phi^{(n)}(t)-\Phi^{(n)}(s) \right) \, \mathrm{d}s.
 \end{equation}
Our next result states that $\mathbb{P}$-almost surely $\nu^{(n)}$ vaguely converges to some $\nu$ as $n \to \infty$.

\begin{lemma}\label{lemma_vague_conv}
If $\nu^{(n)}, \, n \in \N$ is defined as above then $\mathbb{P}$-almost surely
 there exists a locally finite measure $\nu$ on $[0, \infty)$
 such that 
 for any
compactly supported continuous function $h: [0,\infty) \to \R$ we have
\begin{equation}\label{vague_conv_of_measures}
\lim_{n \to \infty} \int_0^{\infty} h(t)\, \mathrm{d}\nu^{(n)}(t)=\int_0^{\infty} h(t)\, \mathrm{d}\nu(t), \qquad
\mathbb{P}-\text{a.s.}
\end{equation}
\end{lemma}

\begin{corollary}\label{corollary_portemanteau}
By the portemanteau theorem \eqref{vague_conv_of_measures} implies 
\begin{equation}\label{eq_Phi_n_conv_to_Phi}
\lim_{n \to \infty} \Phi^{(n)}(t) = \Phi(t) \quad \text{if} \quad \nu(\{t\})=0, \quad
\text{where} \quad \Phi(t): = \nu([0,t])
\end{equation}
\end{corollary}

\begin{proof}[Proof of Lemma \ref{lemma_vague_conv}]

We will use the particle representation (see Section \ref{section_particle_representation})
 \[Y_1^{(n)}(t), \dots, Y_n^{(n)}(t)\]  of $g_t^{(n)}(\cdot)$. We can then write
$\nu^{(n)}= \sum_{i=1}^n m_i \cdot \delta_{t^{(n)}_i}$, see \eqref{def_eq_death_time_mass_measure}.
Note that 
\[Y_i^{(n)}(0)=Y_i(0) \qquad \text{for any} \qquad 1 \leq i \leq n.
\]

Let us assume that the function $h: [0,\infty) \to \R$ for which we want to show \eqref{vague_conv_of_measures} is supported on $[0,T]$.

Recalling  the definition 
$\mu_0=\sum_{i=1}^{\infty} m_i \cdot \delta_{Y_i(0)}$, we analogously define 
$\mu_0^{(n)}=\sum_{i=1}^{n} m_i \cdot \delta_{Y_i(0)}$. Then by Lemma \ref{lemma:exponential} there exists a $\mathbb{P}$-almost surely finite random variable  $K_0$ such that
\begin{equation}\label{eq_sparse_weight}
 \sup_{n \geq 0} \frac{ \mu_0^{(n)}[-K,0]}{K} = 
 \frac{ \mu_0[-K,0]}{K} \leq \frac{1}{2T}, \qquad \text{for any} \quad K \geq K_0.
\end{equation}
If $|Y_i(0)|=E_i > K_0$ then for any $t \geq 0$ we have
\begin{equation}\label{speed_ineq}
\frac{\mathrm{d}}{\mathrm{d}t} Y_i^{(n)}(t) \stackrel{ \eqref{particle_dynamics_lambda} }{\leq} 
\lambda  + \mu^{(n)}_t(Y_i^{(n)}(t),0 )  \leq 
 \lambda + \mu^{(n)}_0(Y_i(0),0 ) 
\stackrel{ \eqref{eq_sparse_weight} }{\leq} 
\lambda + \frac{ |Y_i(0)|}{2T}.
\end{equation}
This implies that if $Y_i(0)< Y := - (K_0 \vee 2 \lambda T)$, then 

\[ Y^{(n)}_i(T) 
\stackrel{ \eqref{speed_ineq} }{\leq} 
Y_i(0) + 
\left(\lambda + \frac{ |Y_i(0)|}{2T}\right) \cdot T <0,
\]
which implies that the time of death  $t_i^{(n)}$ of particle $i$ (see \eqref{def_eq_death_time_mass_measure})
satisfies 
\begin{equation}\label{far_starting_point_late_death}
h(t_i^{(n)})=0 \quad \text{if} \quad Y_i(0) < Y.
\end{equation}

Our aim is to show that the sequence $\int_0^{\infty} h(t)\, \mathrm{d}\nu^{(n)}(t), n \in \N$ is Cauchy.
 In order to show this we let $n \leq m$ and bound
\begin{multline}\label{bound_for_vague}
\left| \int_0^{\infty} h(t)\, \mathrm{d}\nu^{(m)}(t) - \int_0^{\infty} h(t)\, \mathrm{d}\nu^{(n)}(t) \right| 
 \stackrel{ \eqref{def_eq_death_time_mass_measure}, \eqref{far_starting_point_late_death} }{\leq} \\
  \sum_{i=1}^n m_i \cdot \left| h(t^{(m)}_i)-h(t^{(n)}_i) \right| \cdot \ind  [ Y_i(0) \geq Y] +
  \Vert h \Vert_{\infty} \cdot \sum_{i=n+1}^m m_i \cdot \ind  [ Y_i(0) \geq Y] .
\end{multline}
In order to bound the first term on the right-hand side of \eqref{bound_for_vague} we observe that if  
$1\leq i \leq n$ and
$Y_i(0) \geq Y$ then 
\begin{multline}\label{bound_for_death_time_error}
\left| t^{(m)}_i-t^{(n)}_i \right| \leq \sum_{k=n}^{m-1} \left| t^{(k+1)}_i-t^{(k)}_i \right|
\stackrel{ \eqref{eq_insertion_of_new_particle_death_times} }{\leq} \\
 \sum_{k=n}^{m-1} \ind  [ Y_{k+1}(0) \geq Y]
 \frac{m_{k+1} |Y_i(0)| }{ \lambda^2 } \exp \left( \frac{ \mu^{(k)}_0( Y_i(0),0) }{\lambda} \right)
 \leq \\
 \frac{|Y|}{\lambda^2} \exp \left( \frac{ \mu_0( Y,0) }{\lambda} \right) 
 \sum_{k=n}^{\infty}  m_{k+1} \cdot \ind  [ Y_{k+1}(0) \geq Y].
\end{multline}
Note that Lemma \ref{lemma:exponential} implies that with probability 1 we have
\[
\sum_{i=n}^{\infty} m_i \cdot \ind  [ Y_i(0) \geq Y]  \to 0 \quad \text{ as } \quad n \to \infty.
\] 
If we combine this with \eqref{bound_for_vague}, \eqref{bound_for_death_time_error} and the fact that
 $h(\cdot)$ is uniformly continuous, we can conclude that $\int_0^{\infty} h(t)\, \mathrm{d}\nu^{(n)}(t), \, n \in \N$ is 
 a Cauchy sequence for any $h \in \mathcal{C}_0(\R)$, from which it follows that the exists $\nu$ for which 
 \eqref{vague_conv_of_measures} holds.
 \end{proof}

\begin{lemma}\label{lemma_only_excursions_burn}
If $\nu$ is the random measure obtained in Lemma \ref{lemma_vague_conv}, then
for every $t \geq 0$ there exists $y \in (-\infty,0]$ such that
\begin{equation}\label{nu_mu_corresp}
\nu[0,t]=\mu_0[y,0].
\end{equation}
\end{lemma}

 \begin{proof}
 Note that 
$\nu^{(n)}= \sum_{i=1}^n m_i \cdot \delta_{t^{(n)}_i}$ for every $n \in \N$, where
$Y_i \geq Y_j$ implies $t^{(n)}_i \leq t^{(n)}_j$ for every $1 \leq i,j \leq n$.
Since $\nu^{(n)}$ is an atomic measure with masses $(m_i)_{i=1}^n$ located at  
$\left( t^{(n)}_i \right)_{i=1}^n$   and $\nu^{(n)} \to \nu$ vaguely,
we can conclude that $\nu$ is also an atomic measure with masses $(m_i)_{i=1}^\infty$
located at  $\left( t_i \right)_{i=1}^\infty$ where $ \lim_{n \to \infty} t_i^{(n)} = t_i$, thus
$Y_i \geq Y_j$ implies $t_i \leq t_j$ for every $ i,j \in \N$. 
From this \eqref{nu_mu_corresp} readily follows.
\end{proof}

\subsubsection{Tilt-and-shift continuity lemmas}
\label{section_tilt_shift_cont_lemmas}

We will prove Theorem \ref{thm:mcld_extension_introduction}\eqref{mcld_rep_tilt_shift_statement_i} 
in Section \ref{section_proof_of_tilt_and_shift_mcld_extension}
by replicating
 the argument given in Section \ref{subsection_extension_mc_tilt_to_l2}. 
 In Section \ref{section_tilt_shift_cont_lemmas} we collect some ingredients of this proof.

 \medskip

Given some $\um \in \ltwoord \setminus \loneord$ we defined $g_0(\cdot)=f_0(\cdot)$
 by \eqref{f0def_b}, $\nu$  by 
Lemma \ref{lemma_vague_conv}, $\Phi(\cdot)$ by \eqref{eq_Phi_n_conv_to_Phi} and $g_t(\cdot)$  by \eqref{g_t_from_g_0}.
The next lemma is the tilt-and-shift version of Lemma \ref{lemma:cadlag_tilt}.

\begin{lemma}\label{lemma_g_t_ORDX_cadlag}
 If $\um \in \ltwoord \setminus \loneord$ then the function
$ t \mapsto \ORDX(g_t)$ is \cadlag{} with respect to the product topology on $\linford$.
\end{lemma}

\begin{proof}
Let us fix $t \geq 0$ and define the auxiliary functions
\[g^{*}_{t + \Delta t}(x)=g_t(x+\Phi(t+\Delta t)-\Phi(t) ), \qquad
g^{**}_{t + \Delta t}(x)=g_t(x)+ \Delta t \cdot x.
\]
From \eqref{g_t_from_g_0} we obtain
\begin{align}\label{g_t_plus_delta_from_g_t_1}
g_{t +\Delta t}(x) &=
g^{*}_{t + \Delta t}(x)  + \lambda \Delta t + 
\int_{t}^{t + \Delta t} ( x +\Phi(t + \Delta t)-\Phi(s)) \, \mathrm{d}s,\\
\label{g_t_plus_delta_from_g_t_2}
g_{t +\Delta t}(x) &=g^{**}_{t + \Delta t}(x +\Phi(t+\Delta t)-\Phi(t) )
+\lambda \Delta t - \int_t^{t + \Delta t} \Phi(s) \, \mathrm{d}s.
\end{align}

If $\um,\um' \in \linford$, we say that $\um \preceq \um'$ if 
$ \sum_{j=1}^i m_j \leq \sum_{j=1}^i m'_j$ 
for any $i \in \N$.

We are going to show
\begin{gather}
\label{g_t_bound_lower}
\ORDX(g^*_{t+\Delta t}) \preceq \ORDX(g_{t+\Delta t}), \\
\label{g_t_bound_upper}
\ORDX(g_{t+\Delta t}) \preceq \ORDX(g^{**}_{t+\Delta t}),\\
\label{g_t_conv_lower}
\ORDX(g^*_{t+\Delta t}) \to \ORDX(g_t) \quad \text{ in }\; \linford \quad \text{ as } \quad \Delta t \searrow 0,\\
\label{g_t_conv_upper}
\ORDX(g^{**}_{t+\Delta t}) \to \ORDX(g_t) \quad \text{ in }\; \linford \quad \text{ as } \quad \Delta t \searrow 0.
\end{gather}

As soon as we show \eqref{g_t_bound_lower}--\eqref{g_t_conv_upper}, we immediately obtain 
\begin{equation*}\label{g_t_ORDX_right_continuous}
 \ORDX(g_{t+\Delta t}) \to \ORDX(g_t) \quad 
\text{in} \quad \linford \quad  
\text{as} \quad \Delta t \searrow 0,
\end{equation*}
i.e., the right-continuity of $ t \mapsto \ORDX(g_t)$ with respect to the $\linford$ topology.
The proof of the existence of left limits is similar and we omit it.

\medskip

 \eqref{g_t_bound_lower} follows from the fact that $g_{t +\Delta t}$ is obtained from
$g^*_{t+\Delta t}$ by adding an increasing function (see \eqref{g_t_plus_delta_from_g_t_1}),
thus the collection of excursions of 
$g_{t + \Delta t}$ are obtained by merging some excursions of $g^*_{t+\Delta t}$.

\eqref{g_t_bound_upper} follows from the fact that $g_{t +\Delta t}$ is obtained from
$g^{**}_{t+\Delta t}$ by a shift to the left plus an addition of a constant (see \eqref{g_t_plus_delta_from_g_t_2}),
thus the excursions of 
$g_{t + \Delta t}$ are obtained by deleting/splitting some excursions of $g^*_{t+\Delta t}$.

If we apply Lemma \ref{lemma_only_excursions_burn} with $\mu_t$ in place of $\mu_0$ then
 it follows that for every $\Delta t \geq 0$ there exists some $y \in (-\infty,0]$
such that $\Phi(t+\Delta t)-\Phi(t)=\mu_t[y,0]$ (see \eqref{def_eq_mu_t}), thus the collection of excursions of 
$g^{*}_{t + \Delta t}$ is obtained by removing some excursions of $g_{t}$ whose total length
is $\Phi(t+\Delta t)-\Phi(t)$. From this 
\eqref{g_t_conv_lower}  follows, since $\Phi(t+\Delta t)-\Phi(t) =\nu(t, t+\Delta t ] \to 0$ as  $\Delta t \searrow 0$.

  From \eqref{g_t_from_g_0} and Lemma
\ref{lemma:ftgood} it follows that $g^{**}_{t + \Delta t}(x) \in \cG$ for any $\Delta t \geq 0$, thus
  Lemma \ref{lemma:cadlag_tilt} implies \eqref{g_t_conv_upper}. The proof of Lemma \ref{lemma_g_t_ORDX_cadlag} is
  complete.
\end{proof}

\begin{corollary}\label{corollary_at_most_countable_double_arch_g_t}
 If $\um \in \ltwoord \setminus \loneord$ then the set of times $t$ such that $g_t$ has a non-strict excursion (c.f.\ Definition \ref{def:excursion}) is countable.
\end{corollary}
\begin{proof}

By  Lemma \ref{lemma_g_t_ORDX_cadlag} the function $t \mapsto \ORDX(g_t)$ 
is \cadlag, therefore it has countably many jumps, c.f.\ \cite[Section 12, Lemma 1]{billingsley}.
Also, the measure $\nu$ has countably many atoms.

Thus we only need to show that if $g_{t_0}$ has a non-strict excursion for some $t_0>0$ and if $\nu(\{ t_0 \})=0$
 then $t \mapsto \ORDX(g_t)$
jumps at $t_0$. 

The rest of the proof is identical to that of Corollary \ref{corollary_countably_many_double_arches}, with the additional observation that the jump at time $t_0$ created by the merger of excursions coming from the ``tilt'' operation cannot be cancelled by the deletion of excursions coming from the ``shift'' operation, since $t \mapsto \Phi(t)$ is continuous from the left at time $t_0$.

\end{proof}

\begin{claim}\label{claim_g_t_n_unif_conv}
 Let $\um \in \ltwoord \setminus \loneord$ and 
 define $g_t(\cdot)$  using  \eqref{g_t_from_g_0} and
$g^{(n)}_t(\cdot)$ using \eqref{g_t_from_g_0_n}. Then $g_t(\cdot)$ is continuous and
if $t \in [0, \infty)$ satisfies $\nu(\{t\})=0$ then 
\begin{equation}\label{g_n_t_converges_to_g_t}
g^{(n)}_t(\cdot) \to g_t(\cdot) \quad \text{ uniformly on compacts.}
\end{equation}
\end{claim}
\begin{proof}
\eqref{g_n_t_converges_to_g_t} follows from Corollary \ref{corollary_portemanteau},
 Lemma \ref{lemma:ucbi} and  \eqref{g_t_from_g_0}, \eqref{g_t_from_g_0_n}.
\end{proof}


\begin{definition}\label{def_good_times_g_t}
By  Corollary \ref{corollary_at_most_countable_double_arch_g_t} 
the (random) set of times $t$ such that either $g_t$ has a non-strict excursion
or $\nu(\{t\})>0$ is
countable.
Hence for all but countably many $t$, 
the probability that $\nu(\{t\})=0$ and all excursions of $g_t$  are strict is equal to $1$.
 Let $\cT^*$ denote this (deterministic)  set of ``good" times $t$.
\end{definition}

\begin{lemma}\label{lemma_approximation_ORDX_tilt_shift}
 If  $\um \in \ltwoord \setminus \loneord$, $t \in \mathcal{T}^*$ 
then almost surely $\ORDX(g^{(n)}_t) \to \ORDX(g_t)$ in the product topology on $\linford$.
\end{lemma}
\begin{proof}
First note that the functions $g_t(\cdot)$ and $g^{(n)}_t(\cdot), n \in \N$ are uniformly good 
(c.f.\ Definition \ref{def_uniformly_good}): this follows from Lemma \ref{lemma:tailsmall} and
 the fact that $g_t$ is a ``shifted'' version of $f_t$:
 \begin{equation}\label{g_t_is_shifted_f_t}
  g_t(x)\stackrel{ \eqref{def_eq_f_t_from_f_0}, \eqref{g_t_from_g_0} }{=}
 f_t(x+\Phi(t)) + \lambda t - \int_0^t \Phi(s)\, \mathrm{d}s, 
 \end{equation}
 and similarly, $g^{(n)}_t$ is a shifted version of $f^{(n)}_t$.
  Now Lemma \ref{lemma_approximation_ORDX_tilt_shift} follows from Definition \ref{def_good_times_g_t},
  Claim \ref{claim_g_t_n_unif_conv}
and Lemma \ref{lemma_unif_good_conv_ordx}.
\end{proof}

\subsubsection{Proof of Theorem \ref{thm:mcld_extension_introduction}}
\label{section_proof_of_tilt_and_shift_mcld_extension}

\begin{proof}[Proof of Theorem \ref{thm:mcld_extension_introduction}\eqref{mcld_rep_tilt_shift_statement_i}]

Given $\um \in \ltwoord \setminus \loneord$, we constructed two stochastic processes:
\begin{itemize}
\item the graphical construction in \cite[Section 3]{james_balazs_mcld_feller} of the $\mathrm{MCLD}(\lambda)$ process
 $\bm_t, t\geq 0$,
 
 \item  the process $\ORDX(g_t), t \geq 0$, defined by \eqref{g_t_from_g_0},
 where the initial state $g_0(\cdot):= f_0(\cdot)$ is defined by \eqref{f0def_b} and the control function $\Phi(\cdot)$ is defined by \eqref{eq_Phi_n_conv_to_Phi}.
\end{itemize}
 We now want to show that these two $\linford$-valued processes have the same law.
Both processes are \cadlag{} with respect to the product topology on $\linford$ by
 \cite[Proposition 1.1]{james_balazs_mcld_feller}  and Lemma \ref{lemma_g_t_ORDX_cadlag}.

Hence, since the set $\cT^*$ introduced in Definition \ref{def_good_times_g_t} is dense, 
if we can show that for any finite 
collection $t_1, \dots, t_r\in \cT^*$, 
we have 
\begin{equation}\label{tilt_shift_conclusion}
(\ORDX(g_{t_i}), 1\leq i\leq r)
\isd
(\bm_{t_i}, 1\leq i\leq r),
\end{equation}
then indeed Theorem \ref{thm:mcld_extension_introduction}\eqref{mcld_rep_tilt_shift_statement_i} will follow.

By Proposition \ref{prop:tilt_and_shift_l0} we have
 \begin{equation}\label{finite_dim_agree_for_truncated_mcld}
 (\ORDX(g_{t_i}^{(n)}, 1\leq i\leq r)\isd (\bm^{(n)}_{t_i}, 1\leq i\leq r),
 \end{equation}
where $\bm^{(n)}_t, t \geq 0$ is the $\mathrm{MCLD}(\lambda)$ process started from $\um^{(n)}$.
By the Feller property of $\mathrm{MCLD}(\lambda)$ 
(see \cite[Theorem 1.2]{james_balazs_mcld_feller}) we have
\begin{equation}\label{eq_truncated_conv_feller_mcld}
(\bm^{(n)}_{t_i}, 1\leq i\leq r)\tod (\bm_{t_i}, 1\leq i\leq r), \qquad n \to \infty 
\end{equation}
(with respect to the topology of $\ltwoord$ and hence also coordinatewise).
Putting together \eqref{finite_dim_agree_for_truncated_mcld}, Lemma \ref{lemma_approximation_ORDX_tilt_shift} and
\eqref{eq_truncated_conv_feller_mcld}
we obtain \eqref{tilt_shift_conclusion}.
The proof of Theorem \ref{thm:mcld_extension_introduction}\eqref{mcld_rep_tilt_shift_statement_i} is complete.
\end{proof}

\begin{proof}[Proof of Theorem \ref{thm:mcld_extension_introduction}\eqref{g_t_zero_is_zero}]
It is enough to show that for any $K>0$ and $\varepsilon>0$ 
we almost surely have $-\varepsilon \leq g_t(0) \leq 0$ for any $0 \leq t \leq K$.
Let us fix $K, \varepsilon>0$.
Recall from \eqref{eq_sparse_weight}-\eqref{far_starting_point_late_death} that 
 there exists $Y<0$ such that if $Y_i(0)<Y$ then $t_i^{(n)}>K$ for any $i \leq n$.
 By Lemma \ref{lemma:dense} there exists an almost surely finite $n_0$ such that
$\mu_0^{(n_0)}[y-\varepsilon,y]>0$ for any $Y \leq y \leq 0$, thus for
 any $n \geq n_0$ and any $x \geq 0$ such that $Y \leq g_0^{(n)}(x) \leq 0$ we have
$g_0^{(n)}(x_-) - g_0^{(n)}(x) \leq \varepsilon$. In words: the gaps between consecutive particles initially located
in $[Y,0]$ are smaller than or equal to $\varepsilon$. By Definition \ref{def_particle_rep}, these  gaps
can only decrease with time, thus for any $n \geq n_0$ and $t \leq K$ there is a particle in 
$[-\varepsilon,0]$, i.e., we have $-\varepsilon \leq g_t^{(n)}(0) \leq 0$.
Now $g_t^{(n)}(0) \to g_t(0)$ as $n \to \infty$ for all except countably many values of $t \in [0,K]$
 by Claim \ref{claim_g_t_n_unif_conv}, 
moreover $g_t(0)$ is a \cadlag{} function of $t$ by \eqref{g_t_from_g_0}, therefore
$-\varepsilon \leq g_t(0) \leq 0$ holds for every $0 \leq t \leq K$.
\end{proof}

\begin{remark}\label{remark_Phi_strictly_increasing}
 Note that Theorem \ref{thm:mcld_extension_introduction}\eqref{g_t_zero_is_zero} also
implies that the function $\Phi(\cdot)$ is strictly increasing.
 Indeed, if we indirectly assume that $\Phi(s)=\Phi(t)$
for all $s \in [t, t+\Delta t]$, where $\Delta t>0$, then by \eqref{g_t_from_g_0} we obtain $g_{t+\Delta t}(0)=g_t(0)+\lambda \Delta t$, which
contradicts Theorem \ref{thm:mcld_extension_introduction}\eqref{g_t_zero_is_zero}.

\end{remark}

\begin{proof}[Proof of Theorem \ref{thm:mcld_extension_introduction}\eqref{Phi_stopping_time_ii}]



\medskip

Let us assume that $\um=(m_1,m_2,\dots) \in \ltwoord \setminus \loneord$. 
We recursively  define
\[ n_1=1,   \quad  \quad n_k=\min \{ \, i \, : \, m_i<m_{n_{k-1}}  \, \}, \; k \geq 2, \qquad
\widetilde{m}_k=m_{n_k}. \]
Thus we have $\{m_1,m_2,\dots \}=\{ \widetilde{m}_1,\widetilde{m}_2, \dots \}$ and $\widetilde{m}_1>\widetilde{m}_2> \dots$.

Let us fix $x \geq 0$ and let $y=g_0(x)$. Definition \ref{def:f_zero_from_exponentials} and Lemma \ref{lemma:dense} imply
\begin{equation}\label{measure_mu_property_at_x}
\mu_0[y,0) \geq x, \qquad \mu_0[y+\varepsilon,0)<x \qquad \text{ for any } \quad \varepsilon>0.
\end{equation}
The  restriction of the measure $\mu_0$ to $(y,0)$ can be determined by
looking at
$\mathcal{F}_x:=\sigma \left( g_0(x'), 0 \leq x' \leq x \right)$, moreover for any $k \geq 1$ the restriction of
 $\mu_0^{(n_k-1)}=\sum_{i=1}^{n_k-1} m_i \cdot \delta_{Y_i(0)}$ to $(y,0)$ is also determined by $\mathcal{F}_x$,
  simply by removing all atoms with a weight strictly
smaller than $\widetilde{m}_{k-1}$.

As we have already discussed in the proof of Lemma \ref{lemma_only_excursions_burn},
$\nu$ is an atomic measure with masses $(m_i)_{i=1}^\infty$
located at  $\left( t_i \right)_{i=1}^\infty$ where $ \lim_{n \to \infty} t_i^{(n)} = t_i$.
 Now if $Y_i(0) \in [y,0)$ for some $i \geq 1$ and $n_k \geq i$ then $t_i^{(n_k-1)}$ is
 $\mathcal{F}_x$-measurable, because the value of the death time $t_i^{(n_k-1)}$ can be determined by \eqref{particle_dynamics_lambda} if we know the initial particle configuration strictly above $Y_i(0)$
and the value of $Y_i(0)$ (we do not need to know the mass $m_i$ of particle $i$). 
 Therefore $\lim_{k \to \infty} t_i^{(n_k)} = t_i$ is also $\mathcal{G}_x$-measurable.
If we define $t_{y'}=\sup \{\, t_i \, : \, Y_i \in [y',0)\, \}$, then $y' \mapsto t_{y'}$ is continuous,
 $t_y$ is $\mathcal{F}_x$-measurable; moreover
by \eqref{measure_mu_property_at_x} we have
\begin{equation}
\{ \Phi(t) \geq x \} =\{ t \geq t_y \} \in \mathcal{F}_x.
\end{equation}
Hence $\{ \Phi(t) \leq x \} \in \bigcap_{\varepsilon>0} \mathcal{F}_{x+\varepsilon}=\mathcal{F}_x^+$. 
This completes the proof of Theorem \ref{thm:mcld_extension_introduction}(iii).
\end{proof}

\section{Applications}
\label{section_applications}

\subsection{Eternal multiplicative coalescents}
\label{sec:eternal}

In Section \ref{sec:eternal} we restrict to the case $\lambda=0$,
i.e.\ the multiplicative coalescent (with no deletion).

We  showed in Theorem \ref{thm:tilt} that a multiplicative coalescent
$\bm_t, t\geq 0$
started from any initial condition $\bm_0$ in $\ltwoord$
has a ``tilt" representation, 
as $\bm_t=\ORDX(f_t)$, where $f_0$ is a random 
function and $f_t(x)=f_0(x)+xt$. 

This leaves the question of \textit{eternal} coalescents,
i.e.\ those defined for all times  $t \in (-\infty, \infty)$.

Recall from Definition \ref{brownian_parabolic_drift} the notion of Brownian motion with parabolic drift,
or briefly $\BMPD(u)$, where $u \in \R$ is the ``tilt parameter". 
\begin{equation}\label{def_eq_mathcal_u}
\text{If $h \sim \BMPD(u)$, denote by $\mathcal{M}(u)$ the law of $\ORDX(h)$.}
\end{equation}

In \cite[Corollary 24]{aldous_mc}
 Aldous showed that there exists an eternal version of the MC, the 
 \emph{standard multiplicative coalescent}, such that the marginal distribution 
of the coalescent at any given time $t \in \R$ is given by 
$\cM(t)$.
Armend\'ariz \cite{armendariz}, and then also Broutin-Marckert \cite{broutin-marckert}, showed that
if $h_0\sim \BMPD(u)$, and $h_t(x)=h_0(x)+tx$ for all $t\in\R$, then the process $\ORDX(h_t), t\in \R$ is in fact the
standard MC. We will provide an alternative proof of this result in Corollary \ref{cor:eternal_tilt} below.

\medskip

Aldous and Limic \cite{aldous_limic}
described the set of all eternal multiplicative coalescents.
They showed that the marginal distributions of any of these
coalescents can be given by the set of excursion lengths
of a suitable stochastic process. 
To state their results we need to recall some more notation from \cite{aldous_limic}. 
\begin{definition}\label{def:levy_without_replacement}
Given $\kappa\geq 0$, $\tau\in\R$ and $\uc\in \lthreeord$,
define
\begin{equation}
\label{Wdef}
W^{\kappa,\tau,\uc}(x)=\kappa^{1/2} W(x) - \frac12\kappa x^2 
+ \sum_{i=1}^\infty \left(c_i \ind [E_i\leq x] - c_i^2 x\right) 
+\tau x
\end{equation}
for $x\geq 0$, where $W$ is a standard Brownian motion, 
and where, for each $i$, $E_i\sim \Exp(c_i)$, independently
of each other and of $W$.
\end{definition}

If $\uc=\underline{0}$ and $\kappa=1$, then 
$W^{\kappa,\tau, \uc}\sim \BMPD(\tau)$. 
For general $\uc$, the processes defined in \eqref{Wdef}
have been called \textit{L\'evy processes without replacement};
each jump of size $c_i$ occurs at rate $c_i$, but once 
such a jump has occurred it does not happen again. 
Define also the parameter space
\begin{equation}\label{Idef}
\mathcal{I}=
\Big[ (0,\infty)\times(-\infty, \infty)\times \lthreeord\Big]
\cup
\Big[ \{0\}\times (-\infty, \infty) \times \big(\lthreeord\setminus\ltwoord\big)\Big].
\end{equation}

We can phrase the results of  
\cite[Theorems 2 and 3]{aldous_limic} as follows:
\begin{theorem}\label{thm:aldous_limic}
\begin{itemize}
\item[(i)]
For each $(\kappa, \tau, \uc)\in \cI$, 
there exists an eternal multiplicative coalescent
$\bm_t, t\in\R$ such that for each $t$,
$\bm_t\isd \ORDX(W^{\kappa, \tau+t, \uc})$.
\item[(ii)]
Let $\mu(\kappa,\tau,\uc)$ be the distribution
of the MC in (i). 
Then the extreme points of the set of 
eternal MC distributions
are $\mu(\kappa, \tau, \uc)$ for $(\kappa, \tau, \uc)\in\cI$,
together with the distributions of ``constant" 
processes (such that $\bm_t=(y,0,0,0,\dots)$ 
for all $t$, for some $y\geq 0$).
\end{itemize}
\end{theorem}

Our results  can be applied to show that all these 
eternal coalescents also have a ``tilt" representation.
First, we make a definition:
\begin{definition}\label{def_exp_exc_levels}
Let $W(x), x\geq 0$ be a random function. 
We say that $W$ has the \textit{exponential excursion levels}
property if the following holds: conditional on 
the sequence of excursion lengths $\ORDX(W)=(m_1,m_2,\dots)$,
the levels of the excursions of $W$ are independent, 
with the excursion of length $m_i$
occurring at level $-E_i$ where $E_i\sim \Exp(m_i)$. 
\end{definition}

Now we can state the key property that we need:
\begin{claim}
\label{claim:Levy_excursions}
Let $(\kappa,\tau,\uc)\in\cI$.
Then  
$W^{\kappa, \tau, \uc}$
has the exponential excursion levels property.
\end{claim}
\begin{remark}\label{remark_bmpd_exp_exc}
Equation (72) of \cite{aldous_mc} states that $\BMPD(t)$ 
has the exponential excursion levels
property for any $t \in \R$. From this fact the statement of Claim \ref{claim:Levy_excursions}
follows for $W^{\kappa, \tau, \underline{0}}$  for any $\kappa>0$ and $\tau \in \R$
by scaling. For a short sketch proof of this fact, see 
the $\uc=\underline{0}$ case of 
Section \ref{subsection_sketch_proof_levy_excursions} below.
\end{remark}

Among other things, Claim \ref{claim:Levy_excursions} implies
that the excursions of $W^{\kappa,\tau,\uc}$ 
occur in size-biased order, c.f.\ Claim \ref{claim:sb}. 
From Claim \ref{claim:Levy_excursions}, we can deduce the following result:
\begin{corollary}
\label{cor:eternal_tilt}
Let $(\kappa, \tau,\uc)\in\cI$. 
Then the process $\ORDX(W^{\kappa, \tau+t, \uc}), t\in\R$
is an eternal multiplicative coalescent.
\end{corollary}

We describe below the straightforward way in which 
Corollary \ref{cor:eternal_tilt} follows from
Claim \ref{claim:Levy_excursions}. We then give a sketch
of the proof of Claim \ref{claim:Levy_excursions} in Section \ref{subsection_sketch_proof_levy_excursions}. 
The proof is not difficult but a detailed
account would occupy many more pages,
and require some more technical extensions of earlier
results (in particular the material in Section \ref{subsection_good_functions}) to the setting 
of functions with positive jumps. 

Corollary \ref{cor:eternal_tilt} is also the subject
of recent work by Vlada Limic \cite{limic_new}.
 Limic gives a proof of this result using a construction that involves
\textit{simulateous breadth-first walks} (see Section \ref{subsec:related} for a more detailed discussion).
Even without giving the full proof,
we feel that abstracting the property in 
Claim \ref{claim:Levy_excursions} gives valuable 
insight in complement to the 
different approach of \cite{limic_new}. 

\begin{proof}[Proof of Corollary \ref{cor:eternal_tilt}
(using Claim \ref{claim:Levy_excursions})]

Fix any $\tau\in\R$. Let $h_t=W^{\kappa,\tau+t,\uc}$
for $t\geq 0$. Note that 
\[
h_t(x)=h_0(x)+tx
\]
for $x\geq 0$, $t\geq 0$. 

Condition on $\ORDX(W^{\kappa,\tau\uc})=\um$. 
Define $f_0=\bar{h}_0$. Then Claim \ref{claim:Levy_excursions}
tells us precisely that the distribution of $f_0$ is the same 
as the one defined at \eqref{f0def_b}.

Define $f_t$ by $f_t(x)=f_0(x)+tx$ as at \eqref{def_eq_f_t_from_f_0}.
Then by the $\lambda=0$ case of Remark \ref{remark_how_we_extend_mcld_to_l2}\eqref{equivalent_start} we have 
$\ORDX(h_t)=\ORDX(f_t)$ for all $t \geq 0$.

But Theorem \ref{thm:tilt} says that $\ORDX(f_t), t\geq 0$ 
is a MC. 
Hence the same is true of $\ORDX(h_t), t\geq 0$.
That is, $\ORDX(W^{\kappa, \tau+t, \uc}), t\geq 0$ is
a MC for all $\tau$. 
But by considering values of $\tau$ tending to $-\infty$,
and applying Kolmogorov extension, the statement 
of Corollary \ref{cor:eternal_tilt} follows. 
\end{proof}

\subsubsection{Sketch of proof of Claim \ref{claim:Levy_excursions}}
\label{subsection_sketch_proof_levy_excursions}

One necessary tool is an extension of Lemma \ref{lemma_unif_good_conv_ordx}, in two directions. 
First, we need not just 
that the collection of excursion lengths
converges, but also that their levels converge; 
this is straightforward. Second, we need to cover
the case where the limit $f$ is allowed
to have positive jumps (since this is true of
the functions $W^{\kappa, \tau, \uc}$ whenever 
$\uc$ is not identically zero); this introduces a few 
technicalities (but doesn't require new ideas). 

\medskip

Now the idea is to take an appropriate sequence of initial conditions
$\um^{(n)}\in \lzeroord$, and times $t_n$. 
Define $f^{(n)}_0$ based on $\um^{(n)}$ according to
\eqref{f0def_for_lzeroord}, and define $f^{(n)}_{t_n}$
as in \eqref{def_eq_f_t_from_f_0}. 
Then we want to show convergence, in an appropriate
sense, of $f^{(n)}_{t_n}$ to $W_{\kappa, \tau, \uc}$; 
then, to use the extension of Lemma \ref{lemma_unif_good_conv_ordx}
to deduce that the lengths and levels of the excurions 
of $f^{(n)}_{t_n}$ converge in distribution to those 
of $W_{\kappa, \tau, \uc}$. Finally, observe that for any $n$,
the function $f^{(n)}_{t_n}$  satisfies the exponential 
excursion lengths property (this is 
a translation of Corollary 
\ref{corollary_exp_measure_mcld}\eqref{cor_exp_meas_mcld_continuous}
from the particle system context into the tilt representation 
context using the identity \eqref{eq_particle_tilt_shift_equiv}). This property is then 
inherited by the limit $W_{\kappa, \tau, \uc}$ of 
the sequence $f^{(n)}_{t_n}$ and Claim \ref{claim:Levy_excursions} follows.

The main part of the work here is showing 
the convergence of $f^{(n)}_{t_n}$ to $W_{\kappa, \tau, \uc}$
in such a way that we can deduce the convergence 
of the lengths and levels of the excursions. 

For convenience, take $\tau=0$, and $\kappa=0$ or $\kappa=1$;
other cases are almost identical. 

\medskip 

{\bf Case 1: $\uc=\underline{0}$ 
and $\kappa=1$.} The distribution of 
$W^{1,0,\underline{0}}$ is simply $\BMPD(0)$. The fact
that $\BMPD(u)$ satisfies the exponential excursion heights
property was already discussed in Remark \ref{remark_bmpd_exp_exc}. Alternatively, let the initial 
condition $\um^{(n)}$ consist of $n$ blocks each of size
$n^{-2/3}$, and let $t_n=n^{1/3}$. The function $f^{(n)}_0$
has $n$ excursions each of length $n^{-2/3}$. 
The gaps between the levels of these excursions are given 
by independent exponential random variables; 
for $0\leq k<n$, let $F_k$ be the $(k+1)$st such gap, 
with rate $n^{1/3}\frac{n-k}n$.

Then the increments of  $f^{(n)}_{n^{1/3}}$ 
 on the intervals  $[(kn^{-2/3}, (k+1)n^{-2/3}]$
 are independent over $1\leq k< n$; the $k$th such increment 
is given by $n^{-1/3}-F_k$.
Let us denote \[x=n^{-2/3}k, \quad \mathrm{d}x=n^{-2/3}, \quad
\mathrm{d} f^{(n)}_{n^{1/3}}(x)=f^{(n)}_{n^{1/3}}(x+\mathrm{d}x)-f^{(n)}_{n^{1/3}}(x).\]
 We have
\begin{align*} \mathbb{E}( \mathrm{d} f^{(n)}_{n^{1/3}}(x)) &= -kn^{-4/3}(1+O(k/n)) \approx - x \mathrm{d}x , \\
\mathrm{Var}( \mathrm{d} f^{(n)}_{n^{1/3}}(x))&=n^{-2/3}(1+O(k/n)) \approx \mathrm{d}x.
\end{align*}
  In this way we obtain 
a functional limit theorem; the distribution of $f^{(n)}_{n^{1/3}}$ 
converges (in the sense of uniform convergence on finite intervals) to $\BMPD(0)$. 

(By this method we get an alternative, self-contained proof of the results of
\cite{armendariz, broutin-marckert} claiming  that the tilt procedure applied
to the $\BMPD$ family gives the (eternal) standard MC.)

\medskip

 {\bf Case 2: $\kappa=0$.} 
Now the process in \eqref{Wdef} is given just by the
compensated jumps according to the vector $\uc$. 
In keeping with \eqref{Idef}, we now need 
$\uc\in\lthreeord\setminus\ltwoord$.

Here, for an appropriate initial condition
take $t_n=c_1^2+\dots+c_n^2$ and 
$\um^{(n)}=t_n^{-1}(c_1, c_2,\dots, c_n, 0, 0, \dots)$.
We can couple the sequence of initial functions $f_0^{(n)}$, $n \in \N$ with $W^{0,0,\underline{c}}$
in the following way.
For $k\leq n$, let $f_0^{(n)}$ have an excursion
on an interval of length $t_n^{-1} c_k$ at level $-t_n E_k$ 
where $E_k\sim \Exp(c_k)$ independently for each $k$. 

Consider what happens on this interval in $f_{t_n}^{(n)}$.
Because of the tilt, the function increases by
$t_n t_n^{-1}c_k=c_k$ over the course of the interval. 
The length $t_n^{-1} c_k$ goes to $0$ as $n\to\infty$.

By Remark \ref{remark:f_0}, the horizontal location of the start of the interval 
is 
\begin{equation}\label{horizontal}
\sum_{j=1}^n t_n^{-1} c_j \ind(E_j<E_k);
\end{equation}
conditional on $E_k$, this has mean converging to $E_k$ 
and variance converging to 0, and
(using simple martingale calculations)
can be shown to converge almost surely as $n\to\infty$ to $E_k$.
Hence in the limit process this interval produces a 
jump of size $c_k$ occurring at time $E_k$,
matching the term on the right of (\ref{Wdef}).

The vertical location of the start of the interval
is $-t_n E_k$ in the function $f_0^{(n)}$; applying the tilt
corresponding to the horizontal distance in 
\eqref{horizontal}, 
its vertical location in $f_{t_n}^{(n)}$ is
\[
-t_n E_k + t_n \sum_{j=1}^n t_n^{-1} c_j \ind(E_j<E_k),
\]
which, after simplification, is
$\sum_{j=1}^n \big(c_j\ind(E_j<E_k)-c_j^2 E_k\big)$.
Comparing with \eqref{Wdef}, this converges to
$W^{0,0,\uc}(E_k-)$ as $n\to\infty$, so in the 
limit process the jump indeed appears at the correct height. 

In the limit process, these jumps are dense. 
In this case we do not have uniform convergence of
$f_{t_n}^{(n)}$ to $W^{0,0,\uc}$, 
but by considering suitable time-changes one obtains that 
$f_{t_n}^{(n)}$ has the same
excursions as a function $h^{(t_n)}$ which 
converges to $W^{0,0,\uc}$
in the Skorohod topology, and this is enough 
to give convergence of the lengths and levels of excursions
as required. 

\medskip 

{\bf Case 3: $\uc \neq \underline{0}$ and $\kappa=1$.}
Now  the process $W^{1,0,\uc}$ has both a Brownian part and 
positive jumps. Here a suitable sequence of initial conditions
is given by taking $\um^{(n)}$ to consist 
of blocks of sizes $n^{-1/3}(c_1, c_2, \dots, c_{k(n)})$
along with $n$ blocks of size $n^{-2/3}$, where $k(n)$ 
is chosen such that $\sum_{i=1}^{k(n)} c_i^2 \ll n^{1/3}$ 
as $n\to\infty$. (This is the same regime
used in the proof of Lemma 8 of \cite{aldous_limic}). 
Similarly to the two previous cases, choose
$t_n=  \Vert \um^{(n)} \Vert_2^{-1}$.
The ideas of the two previous cases can be combined
to give the desired result for $W^{1,0,\uc}$ also.

\subsection{Tilt-and-shift of $\BMPD$}
\label{sec:BMPD}

Recall the definition of $\BMPD(u)$ from Definition \ref{brownian_parabolic_drift}.
 We show that applying the tilt-and-shift procedure
starting from an initial state $h_0$ which is a $\BMPD$ results in a MCLD process. Furthermore, the function
$h_t$ remains in the class of $\BMPD$ processes
(with a random parameter). 

\begin{proposition}\label{proposition_mcld_brownian}
Let $u\in\R$, and
let $h_0 \sim \BMPD(u)$. 
Let $g_0=\bar{h}_0$. 
Let $\Phi(t), t \geq 0$ and $g_t, t\geq 0$
be given by the tilt-and-shift procedure
in Theorem \ref{thm:mcld_extension_introduction},
and let $h_t$ be given by \eqref{h_tilt_shift}.
\begin{enumerate}[(i)]
\item \label{tilt-and-shift-Brownian}
The process $\ORDX(h_t), t\geq 0$ is a $\mathrm{MCLD}(\lambda)$ process. 
\item \label{bmpd_stays_bmpd} 
Given $\Phi(t)$ and 
$\left(h_0(x), x\leq \Phi(t)\right)$, 
the conditional law of  
$h_t$ is 
\begin{equation}\label{window_proc_appears}
\BMPD(u+t-\Phi(t)).
\end{equation}
\end{enumerate}
\end{proposition}

\begin{proof} 
We first note that by Remark \ref{remark_bmpd_exp_exc} the function $h_0$ (and thus $g_0$)
has the exponential excursion levels property (c.f.\ Definition \ref{def_exp_exc_levels}). 
Together with Definition \ref{def:f_zero_from_exponentials} this implies that $g_0$ is
 a suitable 
initial state of the tilt-and-shift representation.
Then by Theorem \ref{thm:mcld_extension_introduction},
$\ORDX(g_t)$ is a MCLD($\lambda$) process,
and, as observed in 
Remark \ref{remark_how_we_extend_mcld_to_l2}(\ref{equivalent_start}),
so is $\ORDX(h_t)$. This completes the proof of part \eqref{tilt-and-shift-Brownian}.

For part \eqref{bmpd_stays_bmpd}, since $h_0 \sim \BMPD(u)$, we
have
\begin{multline}\label{h_t_bmpd}
h_t(x) \stackrel{\eqref{h_tilt_shift}, \eqref{def_eq_bmpd} }{=}
B(x+\Phi(t))-\frac12(x+\Phi(t))^2+u \cdot (x+\Phi(t)) +\\
\lambda t + \int_0^t \left(x + \Phi(t)-\Phi(s) \right) \, \mathrm{d}s.
\end{multline}
Since $\ORDX(g_0)=\ORDX(h_0)$ is in $\ltwoord \setminus \loneord$
with probability 1, we have $h_t(0)=g_t(0)=0$ by 
Theorem \ref{thm:mcld_extension_introduction}\eqref{g_t_zero_is_zero}, and if we combine this with
\eqref{h_t_bmpd}, we obtain
\begin{equation}\label{h_t_zero_bmpd}
0 = B(\Phi(t))-\frac12(\Phi(t))^2+u\Phi(t) +\\
\lambda t + \int_0^t \left( \Phi(t)-\Phi(s) \right) \, \mathrm{d}s.
\end{equation}
Subtracting \eqref{h_t_zero_bmpd} from \eqref{h_t_bmpd} we get
\begin{equation}\label{bmpd_h_t_identity}
h_t(x)=\left( B(x+\Phi(t))-B(\Phi(t))\right) -\frac12x^2+(u+t-\Phi(t))  x, \quad x \geq 0.
\end{equation}

Now Theorem \ref{thm:mcld_extension_introduction}(iii) states that
$\Phi(t)$ is a stopping time w.r.t.\ the filtration
$(\mathcal{F}_x^+)$.
Then by the strong Markov property for Brownian motion w.r.t.\ $(\mathcal{F}_x^+)$ (see
 \cite[Theorem 2.14]{moerters-peres-book}),
we have that $B\big(x+\Phi(t)\big)-B\big(\Phi(t)\big)$ is a
standard Brownian motion independent of 
$\Phi(t)$ and $\big(B(x), 0\leq x\leq \Phi(t)\big)$.
Using \eqref{bmpd_h_t_identity} then gives 
that the conditional distribution of $h_t$ is \eqref{window_proc_appears}.
\end{proof}

\begin{remark}
Proposition \ref{proposition_mcld_brownian}\eqref{bmpd_stays_bmpd} gives
a $\mathrm{MCLD}$ process
whose marginal distributions are all given by mixtures 
of distributions $\cM(u)$ (c.f.\ Definition \eqref{def_eq_mathcal_u}). 
In \cite{james_balazs_window_process} we 
find eternal processes with this property, 
which may be stationary or non-stationary. We
show that such processes arise naturally as 
scaling limits of discrete models such as 
frozen percolation processes (see Definition \ref{def_frozen_percolation_model} below)
or forest fire processes (see Section \ref{section_particle_forest}).

In Section \ref{section_scaling_lim_of_frozen_perc} below we observe a simple case of such a scaling limit. 
\end{remark}

\subsubsection{Scaling limit of frozen percolation started from a critical \ER graph}
\label{section_scaling_lim_of_frozen_perc}

First we recall the notion of \emph{mean-field frozen percolation process} from \cite{br_frozen_2009} (using slightly different notation).

\begin{definition}[$\mathrm{FP}(n,\lambda(n))$]
\label{def_frozen_percolation_model}
We start with a graph $F^{(n)}_0$ on $n$ vertices. Between each pair of unconnected vertices an edge appears with rate $1/n$; also,
every connected component of size $k$ is deleted with rate $\lambda(n)\cdot k$. (When a component is deleted, its vertices as well as its 
edges are removed from the graph.) 
Let $F^{(n)}_t$ be the graph at time $t$.  
Denote by 
\[ 
\mathbf{M}^{(n)}(t)= \left(M_1^{(n)}(t), M_2^{(n)}(t), \dots \right) \in \lzeroord 
\] 
the sequence of component sizes of $F^{(n)}_t$, arranged in decreasing order. 
\end{definition} 
    
Then $\mathbf{M}^{(n)}(t), t \geq 0$ is a Markov process -- let us call it here the
frozen percolation component process on $n$ vertices with lightning rate $\lambda(n)$,
or briefly $\mathrm{FP}(n,\lambda(n))$.

 In order to achieve self-organized criticality, one chooses
$n^{-1} \ll \lambda(n) \ll 1$, c.f.\ \cite[Theorem 1.2]{br_frozen_2009}. The case 
$\lambda(n) \asymp n^{-1/3}$ holds special significance, c.f.\ \cite[Conjecture 1.1]{br_frozen_2009}.

The next result gives a scaling limit for $\mathrm{FP}(n,\lambda(n))$, in a setting where
$\lambda(n) \asymp n^{-1/3}$ and
 the 
initial state has the distribution of an \ER random graph 
at some point within the ``critical window".

\begin{proposition}\label{prop:FPlimit}
Fix $u \in \R$ and let $F^{(n)}_0$ be an \ER graph $\mathcal{G}(n,p)$ with edge probability $p=\frac{1+u n^{-1/3}}{n}$.
Let $\lambda>0$ and let
$\mathbf{M}^{(n)}(t), t \geq 0$ be the 
$\mathrm{FP}(n,\lambda n^{-1/3})$ process 
with initial state $F^{(n)}_0$.

Define $\mathbf{m}^{(n)}(t), t \geq 0$ by 
\begin{equation}\label{frozen_mcld_scaling}
 \mathbf{m}^{(n)}(t):=
\left( n^{-2/3} M_1^{(n)}(n^{-1/3}t), n^{-2/3}M_2^{(n)}(n^{-1/3} t), \dots \right). 
\end{equation}

Then as $n\to\infty$ 
the sequence of processes $\mathbf{m}^{(n)}(t), t \geq 0$ converge in law to the
$\mathrm{MCLD}(\lambda)$ process $\ORDX(h_t), t\geq 0$ given by 
Proposition \ref{proposition_mcld_brownian}.
\end{proposition}
\begin{proof} 
Corollary 2 of Aldous \cite{aldous_mc} (or, alternatively, the method sketched in the $\uc=\underline{0}$
case of Section \ref{subsection_sketch_proof_levy_excursions})
implies that the sequence $\mathbf{m}^{(n)}(0)$ converges in distribution as
$n \to \infty$ in the $(\ltwoord, \dist(\cdot,\cdot))$ space to $\ORDX(h_0)$, where
$h_0 \sim \BMPD(u)$. 

Further, the process $\mathbf{m}^{(n)}(t)$ defined 
by (\ref{frozen_mcld_scaling}) is a
$\mathrm{MCLD}(\lambda)$ process
(as can be readily seen
by comparing the definition \eqref{mcld_informal_def} of the $\mathrm{MCLD}(\lambda)$ process with 
Definition \ref{def_frozen_percolation_model}).

Then the statement of Proposition \ref{prop:FPlimit} follows by applying
Proposition \ref{proposition_mcld_brownian} together with the
Feller property of $\mathrm{MCLD}(\lambda)$ (see  
\cite[Theorem 1.2]{james_balazs_mcld_feller}).
\medskip

\end{proof}

\subsection{Forest fire model}
\label{section_particle_forest}

This section contains a particle representation of the mean-field forest fire model of \cite{bt_br_forest};
see Section \ref{section_part_rep_ff}. 
The representation is an adaptation of the one  in Section \ref{section_particle_representation}, and
we will briefly explain in Section \ref{section_controlled_burgers} how it
sheds some new light on a certain controlled non-linear PDE problem (see \eqref{burgers_controlled} below)
 which played a central role in the 
theory developed in \cite{bt_br_forest} and \cite{ec_nf_bt_forest}.

\medskip
 In \cite{bt_br_forest} R\'ath and T\'oth modify the dynamical Erd\H{o}s-R\'enyi model to obtain the mean-field forest fire model:
\begin{definition}[$\mathrm{FF}(n,\lambda(n))$]
\label{def_forest_fire_model}

 We start with a graph on $n$ vertices. Between each pair of unconnected vertices
 an edge appears with rate $1/n$; moreover each
 connected component of size $k$ ``burns''  with rate $\lambda(n)\cdot k$, i.e., the edges of the component are deleted.
  The total number of vertices remains $n$.
 
  Denote by $\mathcal{C}^n(i,t)$
 the connected component of vertex $i$ at time $t$.
 
We define the empirical component size densities by
\begin{equation}\label{vknt}
\mathbf{v}_k^n(t)=\frac{1}{n}\sum_{i=1}^n
\ind [ |\mathcal{C}^n(i,t)|= k ], \qquad \mathbf{v}^n(t)= \left( \mathbf{v}_k^n(t) \right)_{k=1}^n.
\end{equation}
With the above definitions $\mathbf{v}^n(t), t \geq 0$ is a Markov process, let us call it here the
forest fire component size density Markov process on $n$ vertices with lightning rate $\lambda(n)$,
or briefly $\mathrm{FF}(n,\lambda(n))$.
 \end{definition}

\subsubsection{Particle representation of the forest fire model}
\label{section_part_rep_ff}

In Definition \ref{def_particle_rep_forest} and Proposition \ref{prop:forest_particle_rep} below 
we are going to give a novel particle representation of $\mathrm{FF}(n,\lambda(n))$
by slightly modifying Definition \ref{def_particle_rep}.

\begin{definition}\label{def_dictionary}
 If $n \in \N_+$ we let
\begin{align*}
\mathcal{V}^n &= \left\{ \; \underline{v}^n  = \left( v_k^n \right)_{k=1}^n 
\; : \; 
\sum_{k=1}^n v_k^n=1 
\; \text{ and } \;
\frac{n}{k} v_k^n \in \N
\;  \text{ for all } \; k \; \right\}, \\
\mathcal{M}^n &=\left\{ \; \um^n= \left( m_j^n \right)_{j=1}^N \in \lzeroord
\; : \;
\sum_{j=1}^N m_j^n=n 
\; \text{ and } \;
 m_j^n \in \N_+
\;  \text{ for all } \; j \; \right\}.
\end{align*}
We say that the component size density vector $\underline{v}^n \in \mathcal{V}^n$ 
and the ordered list of component sizes
 $\um^n \in \mathcal{M}^n$ correspond to each other if
\begin{equation}
 v_k^n=\sum_{j=1}^N \frac{k}{n} \ind[\, m^n_j=k \, ] \;\; \text{ for all }\; k. 
\end{equation}
Note that this correspondence is one-to-one. 
\end{definition}
In plain words, $\underline{v}^n$ and $\um^n$ correspond to each other if there is a graph $G$ on $n$ vertices such that
 $\underline{v}^n$ and $\um^n$ both arise from $G$.
\begin{definition}\label{def_dictionary2}
If $\widetilde{\mu}^n$ is a finite point measure on $\R_-$ such that $\widetilde{\mu}^n(\R_-)=1$ and
the masses of the atoms of $n \widetilde{\mu}^n$ are integers then we define 
$\mathbf{v}(\widetilde{\mu}^n)$ to be the element of $\mathcal{V}^n$ corresponding to the element of  $\mathcal{M}^n$
which consists of the ordered list of masses of the atoms of $n \widetilde{\mu}^n$.
\end{definition}

Now we  define the particle representation of the $\mathrm{FF}(n,\lambda(n))$ model.

 \begin{definition}\label{def_particle_rep_forest}
 Given  $\underline{v}^n(0)=\left( v_k^n(0) \right)_{k=1}^n \in \mathcal{V}^n$
  and the corresponding  $\um^n=\left( m_j^n \right)_{j=1}^N \in \mathcal{M}^n $, 
we define the initial heights of the particles $\widetilde{Y}_i(t), 1 \leq i \leq n$ 
  by letting
  $\widetilde{Y}_i(0)=-E_j$, where $E_j \sim \mathrm{Exp}( m^n_j)$, $1 \leq j \leq N$ are independent
  and vertex $i$ initially belongs to component $j$ in the forest fire model.
       We define 
 \begin{equation}\label{mu_n_forest_particles}
 \widetilde{\mu}^n_t=\sum_{i=1}^{n} \frac{1}{n} \delta_{\widetilde{Y}_i(t)} \quad 
 ( \text{ Note: } \mathbf{v}(\widetilde{\mu}^n_0) = \underline{v}^n(0) \; ).
 \end{equation}
 If $\widetilde{Y}_i(t_-)<0$ then we let
\begin{equation}\label{particle_dynamics_lambda_forest}
\frac{\mathrm{d}}{\mathrm{d}t} \widetilde{Y}_i(t)= 
\lambda(n) +\widetilde{\mu}^n_t( \widetilde{Y}_i(t),0),
\end{equation}
and if $\widetilde{Y}_i(t_-)=0$ then we say that vertex $i$ burns and we let $-\widetilde{Y}_i(t)$ have $\mathrm{Exp}(1)$ distribution, independently from everything else.
\end{definition}
In words, a clustered family of particles with  mass $1/n$ start at negative locations,
 move up and merge with other particle clusters, but if a time-$t$ block of particles with total mass $k/n$ reaches
 $0$, then this block burns and gets replaced by $k$ particles of mass $1/n$ with i.i.d.\
 locations with negative
 $\mathrm{Exp}(1)$ distribution.

\begin{figure}\begin{center}
\includegraphics[width=\textwidth]{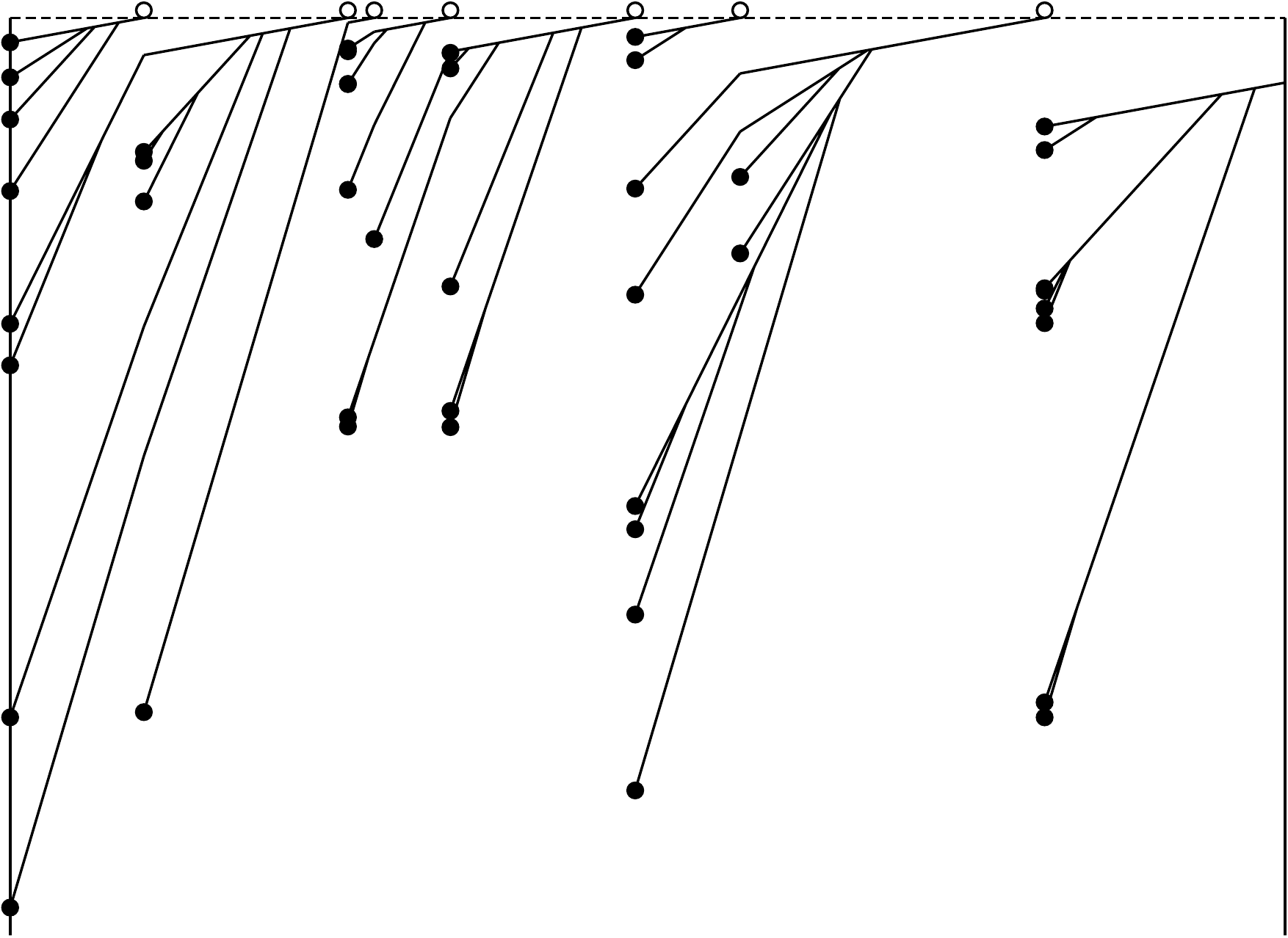}
\caption{
\label{fig:FFparticles}
A simulation of the particle system realising the
forest fire model. The system has $n=8$ particles
and $\lambda=0.4$, and is shown on the time interval $[0,2]$.
Compare to the systems realising the MCLD and the multiplicative
coalescent in Figures \ref{fig:MCLDparticles}
and \ref{fig:MCparticles}. The burning events 
(which occur when a block reaches 0) here involve
blocks of sizes 4, 7, 1, 3, 8, 2, 8 respectively. 
Note that in this case the process of burnings is a Poisson 
process of rate $n\lambda$. 
}
\end{center}
\end{figure}

\begin{proposition}\label{prop:forest_particle_rep}
\begin{enumerate}[(i)]
\item \label{forest_fire_particle_statement_i} For any $n \in \N_+$ and any initial state
$\underline{v}^n(0) \in \mathcal{V}^n$,
the process $\mathbf{v}(\widetilde{\mu}^n_t), t \geq 0$ is a $\mathrm{FF}(n, \lambda(n))$ process
with initial state $\underline{v}^n(0)$ (see  Definitions \ref{def_forest_fire_model}, \ref{def_particle_rep_forest} and \ref{def_dictionary2} for the definitions of $\mathrm{FF}(n, \lambda(n))$,  $\widetilde{\mu}^n_t$ and $\mathbf{v}(\widetilde{\mu}^n)$, respectively). 
\item \label{forest_particle_mu_exp_dist} For any $t \geq 0$, the conditional 
distribution of $n\widetilde{\mu}^n_t$  given the $\sigma$-algebra  $\sigma(\bm^n_{s}, 0 \leq s \leq t)$  is 
$\exmeasure(\bm^n_t)$ (c.f.\ Definition \ref{def_exp_measure_law}), where $\bm^n_t$ is the 
$\mathcal{M}^n$-valued random variable 
corresponding to  the $\mathcal{V}^n$-valued random variable $\mathbf{v}(\widetilde{\mu}^n_t)$ (c.f.\ Definition \ref{def_dictionary}).
\end{enumerate}
\end{proposition}

See Figure \ref{fig:FFparticles} for a simulation
of the particle system realising the forest fire model.

\begin{proof}[Proof of Proposition \ref{prop:forest_particle_rep}]
Recalling Definition \ref{def_exp_measure_law} we observe that by Definition \ref{def_particle_rep_forest} we have
 $n\widetilde{\mu}_0^n \sim \exmeasure(\underline{m}^n)$,
 where $\underline{m}^n \in \mathcal{M}^n$  corresponds to $\underline{v}^n(0) \in \mathcal{V}^n$.
 
Denote by $\tau$ the first burning time of the particle system $\widetilde{Y}_i(t), 1 \leq i \leq n$.

Denote by $Y_i(t):=\widetilde{Y}_i(nt)$ and $\mu_t:=n \widetilde{\mu}^n_{nt}=\sum_{i=1}^n \delta_{Y_i(t)}$ so that
\[ \frac{\mathrm{d}}{\mathrm{d}t} Y_i(t) \stackrel{ \eqref{particle_dynamics_lambda_forest} }{=} 
n\lambda(n) +\mu_t( Y_i(t),0), \quad 0 \leq t < \frac{\tau}{n}, \]
thus the evolution of the particle system $Y_i(t), 1 \leq i \leq n$ satisfies
Definition \ref{def_particle_rep} (with $\lambda=n \lambda(n)$)
up to time $\tau/n$, including the time-$\tau/n$ block that burns.
 
 Likewise, if $\mathbf{v}^n(t),t \geq 0$ is a $\mathrm{FF}(n,\lambda(n))$ process, then the
 $\mathcal{M}^n$-valued process corresponding to the $\mathcal{V}^n$-valued process
 $\mathbf{v}^n(nt),0 \leq t \leq \tau/n$  satisfies the definition of a $\mathrm{MCLD}(n\lambda(n))$ process,
  including the time of the first
deletion event and the component that gets deleted.

 Therefore by Corollary \ref{corollary_particle_mcld}
 the statement of  
 Proposition \ref{prop:forest_particle_rep}\eqref{forest_fire_particle_statement_i} holds for 
 $\mathbf{v}(\widetilde{\mu}^n_t), 0 \leq t \leq \tau$ (including time $\tau$, since in both Definitions
 \ref{def_forest_fire_model} and \ref{def_particle_rep_forest} we add $k$ singletons of mass $1/n$ if a block
 of size $k$ is deleted).
 
 Next we observe that the conditional 
distribution of $n\widetilde{\mu}^n_\tau$  given the $\sigma$-algebra  
$\sigma(\bm^n_{s}, 0 \leq s \leq \tau)$  is 
$\exmeasure(\bm^n_\tau)$. Indeed, this follows from
  Corollary \ref{corollary_exp_measure_mcld}\eqref{cor_exp_meas_mcld_discrete} 
  and the fact that we insert
$k$ particles of mass $1/n$ with i.i.d.\ $-\mathrm{Exp}(1)$ distribution
 if a block of $k$ particles burn at time $\tau$, which is exactly what we have to do to maintain the 
 property required by Definition \ref{def_exp_measure_law}.

 Therefore we can inductively repeat this argument using
Corollary \ref{corollary_exp_measure_mcld}\eqref{cor_exp_meas_mcld_discrete} again and again to show that
Proposition \ref{prop:forest_particle_rep}\eqref{forest_fire_particle_statement_i} holds for $\mathbf{v}(\widetilde{\mu}^n_t), 0 \leq t \leq \tau_i$,
where $\tau_i$ is the $i$'th burning time. The proof of
Proposition \ref{prop:forest_particle_rep}\eqref{forest_particle_mu_exp_dist} similarly follows from
Corollary \ref{corollary_exp_measure_mcld}\eqref{cor_exp_meas_mcld_continuous}.
\end{proof}

\begin{remark}\label{remark_merle_normand_particle}
Let us recall a closely related dynamic random graph model of self-organized criticality, studied 
by Fournier and Laurencot in 
\cite{fournier_laurencot_smol} and by Merle and Normand in \cite{merle_normand}.
One starts with a graph on $n$ vertices, the coagulation mechanism is the same as in
 Definition \ref{def_frozen_percolation_model} and  Definition \ref{def_forest_fire_model} 
 (between each pair of unconnected vertices
 an edge appears with rate $1/n$), but the deletion mechanism is different:
 connected components are deleted forever when their size exceeds a threshold  $\omega(n)$.
 In order to achieve self-organized criticality, one chooses
 $1 \ll \omega(n) \ll n$. 

Let us remark that this model also admits a particle representation:
the initial weights and heights of particles should be the same as in Definition \ref{def_particle_rep_forest}, and
the particle dynamics between deletion times should be given by \eqref{particle_dynamics_lambda_forest}. 
The deletion mechanism is obvious: if
the weight of a time-$t$ block exceeds the threshold $\omega(n)$, we remove that time-$t$ block. The proof of the validity
of this particle representation can be carried out analogously to the proof of Proposition \ref{prop:forest_particle_rep}.

This is a ``rigid'' representation, i.e., all of the randomness is encoded in the initial state. However,
in contrast to the case of the frozen percolation model (i.e., $\mathrm{MCLD}$), the above described threshold-deletion model does not admit a clean tilt-and-shift representation, since
particles from the middle can be deleted (as opposed to the case of linear deletion, where only the top particle
can be deleted).

\end{remark}

\subsubsection{The controlled Burgers equation}
\label{section_controlled_burgers}

In this section we use the particle representation of the $\mathrm{FF}(n,\lambda(n))$ process to
give a new interpretation of a certain controlled non-linear PDE problem (see \eqref{burgers_controlled} below)
 which played a central role in the 
theory developed in \cite{bt_br_forest} and \cite{ec_nf_bt_forest}. This new interpretation will be presented
in Remark \ref{remark_forest_particle_PDE}.

\medskip

One investigates the $\mathrm{FF}(n,\lambda(n))$ when $\frac{1}{n} \ll \lambda(n) \ll 1$ as
 $n \to \infty$. This is called the \emph{self-organized critical regime} of the lightning rate $\lambda(n)$.
 We assume that $\mathbf{v}_k^n(0) \to v_k(0)$ for all $k \in \N$ as $n \to \infty$, where $\sum_k k^3 v_k(0)<+\infty$.
 
Under these assumptions \cite[Theorem 2]{bt_br_forest} states that 
 \begin{equation}\label{eq_forest_densities_converge}
  \mathbf{v}_k^n(t) \to v_k(t) \quad \text{in probability as} \quad n \to \infty,
  \end{equation}
   where
$\left(v_k(t)\right)_{k=1}^{\infty}$ is the unique solution of the following system of ODE's:
 \begin{equation}\label{smolcontr}
  \forall\, k \geq 2\;\;\;
   \frac{\partial}{\partial t} v_k(t)= \frac{k}{2} \sum_{l=1}^k v_l(t) v_{k-l}(t) -k v_k(t), 
    \qquad \sum_{k=1}^{\infty} v_k(t) \equiv 1.
   \end{equation}
   This system of equations is a modification of Smoluchowski's 
   coagulation equations with multiplicative kernel (c.f.\ \cite[Section 2.1]{aldous_coag_survey}).

In order to prove that \eqref{smolcontr} is well-posed (c.f.\  \cite[Theorem 1]{bt_br_forest}), 
one looks at the 
Laplace transform 
\begin{equation}\label{V_t_x_laplace_def}
V(t,x)=\sum_{k=1}^{\infty} v_k(t) e^{-kx} -1
\end{equation}
 which satisfies the following controlled PDE
(c.f.\ \cite[(43)]{bt_br_forest}):
\begin{equation}
\label{burgers_controlled} 
\frac{\partial}{\partial t} V(t,x)=-V(t,x) \frac{\partial}{\partial x} V(t,x) + \varphi(t) e^{-x},
\qquad V(t,0)\equiv 0,
\end{equation}
where the control function $\varphi(t)$ measures the intensity of fires at time $t$:
 \begin{equation}\label{def_eq_varphi_burning_rate}
  \varphi(t)= \frac{\partial}{\partial t} r(t),\quad 
r(t)=\lim_{n \to \infty} \mathbf{r}^n(t), \quad
\mathbf{r}^n(t)=\frac{1}{n} \sum_{i=1}^n B^n(i,t),
\end{equation}
 and $B^n(i,t)$ denotes the number of times vertex $i$ 
has burnt before time $t$. Note that \eqref{burgers_controlled} is a controlled variant
of the Burgers equation.

Given a solution $V(\cdot,\cdot)$ of \eqref{burgers_controlled} one defines
the corresponding characteristic curves (c.f.\ \cite[(66)]{bt_br_forest})  as the solutions of the ODE 
\begin{equation}\label{characteristics_forest}
 \frac{\mathrm{d}}{\mathrm{d}s} \xi(s)=V(s,\xi(s)).
 \end{equation}
These curves are useful because
by \eqref{burgers_controlled} they satisfy 
$\frac{\mathrm{d}^2}{\mathrm{d}s^2} \xi(s)=\varphi(s)e^{-\xi(s)}$,
hence given $\varphi(\cdot)$ they can be  constructed  (c.f.\ \cite[(65)]{bt_br_forest})
 without solving \eqref{burgers_controlled}.

\begin{remark} \label{remark_forest_particle_PDE}
Let us assume that  $\widetilde{\mu}^n_t$  converges weakly in probability to some measure $\widetilde{\mu}_t$ as $n \to \infty$.
Denote by 
\[\widetilde{V}(t,y)=\widetilde{\mu}_t(y,0),\qquad y \leq 0.\]
We will non-rigorously derive a PDE for $\widetilde{V}(t,y)$, see \eqref{pde_for_particle_density} below.
 
We have $\widetilde{\mu}_t[y-\mathrm{d}y, y]=-\frac{\partial}{\partial y}\widetilde{V}(t,y)\mathrm{d}y$, moreover
by \eqref{particle_dynamics_lambda_forest}, each ``particle'' near the location $y$ moves
 with speed $\widetilde{V}(t,y)$ (since $\lambda(n) \ll 1$), thus $\widetilde{V}(t,y)$ increases
 by $\widetilde{\mu}_t[y-\mathrm{d}y,y]$ on the time interval $[t, t+\mathrm{d}t]$, where 
 $\mathrm{d}y=\widetilde{V}(t,y) \mathrm{d}t$. The mass $\widetilde{V}(t,y)$ also decreases by 
 $\varphi(t)\mathrm{d}t$ because of burning (see \eqref{def_eq_varphi_burning_rate})
  and increases by $(1-e^y)\varphi(t)\mathrm{d}t$ because of the re-insertion of burnt mass 
  with distribution $-\mathrm{Exp}(1)$. Putting these effects together we obtain
  \begin{equation}\label{pde_for_particle_density}
   \frac{\partial}{\partial t}\widetilde{V}(t,y)=-\widetilde{V}(t,y)\frac{\partial}{\partial y}\widetilde{V}(t,y)-
  e^y \varphi(t).
  \end{equation}
By comparing \eqref{burgers_controlled} and \eqref{pde_for_particle_density},
we observe that $V(t,x)$ solves the same PDE as $-\widetilde{V}(t,-x)$.
 Indeed, by \eqref{eq_forest_densities_converge} and Proposition \ref{prop:forest_particle_rep}\eqref{forest_particle_mu_exp_dist} we have
$\widetilde{V}(t,y) = \sum_k v_k(t)(1-e^{ky})$, which is equal to $-V(t,-y)$ by \eqref{V_t_x_laplace_def}.

Moreover, if $1 \ll n$ then $\widetilde{\mu}^n_t(y,0)\approx \widetilde{V}(t,y)$, thus by
comparing \eqref{characteristics_forest} and \eqref{particle_dynamics_lambda_forest} we see that
the trajectories $-\widetilde{Y}_i(s), s \geq 0$ of  particles can be viewed as discrete approximations of
 characteristic curves.
\end{remark}

\section*{Acknowledgments}
We thank Christina Goldschmidt for several 
very helpful discussions at various stages of this work, 
Vlada Limic for sending us a preliminary version
of \cite{limic_new} and for helpful feedback on an earlier version of this manuscript, and Ines Armend\'ariz for sending
us the paper \cite{armendariz}.

This work was supported by EPSRC grant 
EP/J019496/1. JBM was supported by EPSRC Fellowship 
EP/E060730/1, and BR was supported by OTKA (Hungarian
National Research Fund) grant K100473, the Postdoctoral Fellowship of the Hungarian Academy of Sciences and the Bolyai Research Scholarship of the Hungarian Academy of Sciences.

\bibliographystyle{abbrv}
\bibliography{mcld2}
\end{document}